\documentclass[12pt, usenames,dvipsnames]{amsart}
\usepackage{amssymb,amsmath,amsfonts,amscd,euscript,mathrsfs,bm}
\usepackage{fullpage}
\usepackage{lineno}

\DeclareMathAlphabet{\mathpzc}{OT1}{pzc}{m}{it}
\usepackage[sans]{dsfont}

\usepackage{tikz}
\usepackage{tikz-cd}
\usetikzlibrary{babel}
\usetikzlibrary{matrix}
\usetikzlibrary{shapes}
\usetikzlibrary{arrows}
\usetikzlibrary{calc,3d}
\usetikzlibrary{decorations,decorations.pathmorphing}
\usetikzlibrary{through}
\tikzset{ext/.style={circle, draw,inner sep=1pt},int/.style={circle,draw,fill,inner sep=1pt},nil/.style={inner sep=1pt}}
\tikzset{exte/.style={circle, draw,inner sep=3pt},inte/.style={circle,draw,fill,inner sep=3pt}}
\tikzset{diagram/.style={matrix of math nodes, row sep=3em, column sep=2.5em, text height=1.5ex, text depth=0.25ex}}
\tikzset{diagram2/.style={matrix of math nodes, row sep=0.5em, column sep=0.5em, text height=1.5ex, text depth=0.25ex}}
\tikzset{every picture/.append style={baseline=-.65ex}}

\usepackage{todonotes}

\usepackage{hyperref}
\newcommand{\nc}{\newcommand}

\providecommand{\eprint}[2][]{\href{http://arxiv.org/abs/#2}{arXiv:#2}}

\numberwithin{equation}{section}

\newtheorem{thm}[equation]{Theorem}
\newtheorem{prop}[equation]{Proposition}
\newtheorem{prop_def}[equation]{Proposition-Definition}
\newtheorem{lem}[equation]{Lemma}
\newtheorem{cor}[equation]{Corollary}
\newtheorem{rem}[equation]{Remark}
\newtheorem{definition}[equation]{Definition}
\newtheorem{example}[equation]{Example}

\newtheorem{notation}[equation]{Notation}

\newtheorem*{theorem*}{Theorem}

\newtheorem{theoremA}{Theorem}

\newtheorem{conjectureA}[theoremA]{Conjecture}

\nc{\gl}{\mathfrak{gl}}
\nc{\GL}{\mathfrak{GL}}
\nc{\g}{\mathfrak{g}}
\nc{\gh}{\widehat\g}
\nc{\h}{\mathfrak{h}}
\nc{\la}{\lambda}
\nc{\al}{\alpha }
\nc{\be}{\beta }
\nc{\ve}{\varepsilon }
\nc{\om}{\omega }

\nc{\af}{\mathrm{af}}

\nc{\ta}{\theta}
\nc{\ch}{{\mathop {\rm ch}}}
\nc{\gch}{{\mathop {\rm gch}}}

\nc{\Tr}{{\mathop {\rm Tr}\,}}
\nc{\Id}{{\mathop {\mathbf{Id}}}}
\nc{\ad}{{\mathop {\rm ad}}}
\nc{\bra}{\langle}
\nc{\ket}{\rangle}
\nc{\pa}{\partial}
\nc{\ld}{\ldots}
\nc{\cd}{\cdots}
\nc{\hk}{\hookrightarrow}
\nc{\T}{\otimes}
\nc{\gr}{\mathrm{gr}}
\nc{\ov}{\overline}

\nc{\cO}{\mathcal O}
\nc{\cC}{\mathcal C}
\nc{\msl}{\mathfrak{sl}}
\nc{\mgl}{\mathfrak{gl}}
\nc{\mSL}{\mathsf{SL}}
\nc{\U}{\mathrm U}
\nc{\V}{\EuScript V}
\nc{\cL}{\mathcal{L}}

\nc{\tI}{\mathtt I}
\nc{\tJ}{\mathtt J}

\nc{\Res}{\mathrm{Res}}
\nc{\Ind}{\mathrm{Ind}}
\nc{\Coind}{\mathrm{Coind}}
\nc{\LInd}{\mathbf{L}\mathrm{Ind}}
\nc{\RCoind}{\mathbf{R}\mathrm{Coind}}
\nc{\cone}{\mathsf{cone}}
\nc{\End}{\mathrm{End}}

\nc{\hocolim}{\mathsf{hocolim}}

\nc{\sgn}{\mathsf{sgn}}
\nc{\ttt}{\text{-}}

\nc{\dual}{{\checkmark}}

\nc{\sT}{\mathbf{T}\hspace*{-.7em}\text{\bf{--}}\hspace*{0.3em}}
\nc{\sTd}{\widetilde{\sT}}
\nc{\sD}{{{\mathbf{D}\hspace*{-.9em}\text{\bf{--}}\hspace*{0.3em}}}}
\nc{\sF}{\mathbf{F}}
\nc{\sX}{\mathbf{X}}
\nc{\sv}{\mathbf{t}}
\nc{\sG}{\mathbf{G}}

\nc{\sq}{\mathbf{q}}
\nc{\sY}{\mathbf{Y}}

\nc{\sEM}{{\mathbb{EM}}}
\nc{\sPM}{{\mathbb{PM}}}

\nc{\sEMloc}{{\mathds{EM}}}
\nc{\sDM}{{\mathds{DM}}}

\nc{\BGG}{\mathbf{BGG}}

\newcommand{\bC}{{\mathbb C}}
\newcommand{\bD}{{\mathbb D}}

\newcommand{\bQ}{{\mathbb Q}}
\newcommand{\bZ}{{\mathbb Z}}

\newcommand{\bP}{{\mathbb P}}
\newcommand{\bU}{{\mathbb U}}
\newcommand{\bG}{{\mathbb G}}

\newcommand{\bB}{{\mathbb B}}
\newcommand{\bWD}{{{\mathbb W}\makebox[2pt]{D}\, }}

\newcommand{\fp}{{\mathfrak p}}
\newcommand{\fh}{{\mathfrak h}}

\newcommand{\fg}{{\mathfrak g}}
\newcommand{\fgh}{{\widehat{\mathfrak g}}}
\newcommand{\fb}{{\mathfrak b}}

\newcommand{\fn}{{\mathfrak n}}
\newcommand{\fI}{{\mathfrak I}}

\newcommand{\fP}{\mathsf{P}}

\newcommand{\sfP}{\mathsf{P}}
\newcommand{\sfQ}{\mathsf{Q}}

\newcommand{\hT}{\mathpzc{T}}
\newcommand{\hP}{\mathpzc{P}}
\newcommand{\hX}{\mathpzc{X}}
\newcommand{\hY}{\mathpzc{Y}}

\newcommand{\DAHA}{\mathcal{H}\kern -.4em\mathcal{H}}
\newcommand{\hH}{\mathcal{H}}

\newcommand{\Hom}{\mathrm{Hom}}

\newcommand{\ldot}{{\:\raisebox{1.5pt}{\selectfont\text{\circle*{1.5}}}}}
\newcommand{\udot}{{\:\raisebox{4pt}{\selectfont\text{\circle*{1.5}}}}}

\usepackage{ulem}
\usepackage{color}

\begin{document}
	
	
	\title[Categorification of DAHA and Macdonald polynomials]
	{Categorification of DAHA and Macdonald polynomials}

	\author{Syu Kato}
	\address{Syu Kato:\newline
		Department of Mathematics, Kyoto University, Oiwake Kita-Shirakawa Sakyo Kyoto 606-8502 JAPAN} \email{syuchan@math.kyoto-u.ac.jp}
	
	\author[Khoroshkin]{Anton Khoroshkin}
	\address{Anton Khoroshkin: \newline
		Department of Mathematics, University of Haifa, Mount Carmel, 3498838, Haifa, Israel
	}
	\email{khoroshkin@gmail.com}
	
	\author[Makedonskyi]{Ievgen Makedonskyi}
	\address{Ievgen Makedonskyi:\newline
		Yanqi Lake Beijing Institute of Mathematical Sciences And Applications (BIMSA), 
		No. 544, Hefangkou Village, Huaibei Town, Huairou District, Beijing 101408.}
	\email{makedonskii\_e@mail.ru}

	\begin{abstract}
		We describe a categorification of the Double Affine Hecke Algebra ($\DAHA$) associated with an affine Lie algebra $\widehat{\fg}$, including a categorification of the polynomial representation and Macdonald polynomials.
		
		Our categorification results are presented in the derived setting, focusing on the derived category of graded modules over the Lie superalgebra $\fI[\xi]$, where $\fI \subset \widehat{\fg}$ is the Iwahori subalgebra of the affine Lie algebra and $\xi$ is a formal odd variable.
		
		First, we show that the compositions of induction and restriction functors associated with minimal parabolic subalgebras $\fp_{i}$ categorify the Demazure operators $T_i + 1 \in \DAHA$, ensuring that all algebraic relations of $T_i$ have categorical interpretations. Second, for each dominant weight $\lambda$ we introduce a complex $\sEM_{\lambda}$ of $\fI[\xi]$-modules and a complex $\sPM_{\lambda}$ of $\fg[z,\xi]$-modules, whose Euler characteristics are equal to nonsymmetric $E_{\lambda}$ and symmetric $P_{\lambda}$ Macdonald polynomials respectively.
		
		We illustrate our theory with the example   $\fg=\msl_2$ where we construct the cyclic representations of Lie superalgebra $\fI[\xi]$ such that their supercharacters coincide with certain normalizations of nonsymmetric Macdonald polynomials.
	\end{abstract}

	\maketitle
	
	\setcounter{tocdepth}{1}
	\tableofcontents

	\setcounter{section}{-1}
	\section{Introduction}
	\subsection{Categorification results}
	The main goal of this paper is to present an algebraic categorification of the Affine Hecke Algebra (AHA) associated with a simple Lie algebra $\fg$ and polynomial representation of AHA. The categorification is provided by the (bounded) derived category $\bD = \bD^{b}(\fb[\xi])$ of complexes of graded finite-dimensional modules over the Lie superalgebra $\fb[\xi]$, where $\fb$ is a Borel subalgebra of a semisimple Lie algebra $\fg$, and $\xi$ is a formal odd variable. The even part of $\fb[\xi]$ consists of $\fb$, and the odd elements $\xi \fb \subset \fb[\xi]$ form an s-Lie ideal. Let $\sfP$ be the weight lattice of $\g$, and let $\sfP_+$ denote its subset of dominant weights.
	
	This construction generalizes to Double Affine Hecke Algebra, i.e. AHA associated with affine Lie algebra $\widehat{\fg}$ obtained as the untwisted affinization of $\g$, replacing the Borel subalgebra with the Iwahori Lie algebra $\fI \subset \fg[z]$. The alternating sum of graded characters provides a decategorification map: $\bD \to \bZ(\!(q,t)\!)[\sfP]$, where $q$ counts the $z$-grading and $(-t)$ the $\xi$-grading, with the minus sign reflecting $\xi$ as an odd variable. Thus, $\bD$ categorifies a polynomial ring and our first main result in the direction of categorification is the following:
	\begin{theoremA}
		\label{thm::quadr::rel::inttro}
		For each simple root $\alpha$ and the corresponding minimal parabolic subalgebra $\fp_{\alpha}\supset\fb$ the restriction functor $\Res_{\alpha}: \fp_{\alpha}[\xi]\text{-mod} \to \fb[\xi]\text{-mod}$
		between the categories of graded finite-dimensional modules satisfies the following properties:
		\begin{itemize}
			\item {\rm(Theorem~\ref{cor::H::Ind}).}
			$\Res_{\alpha}$ admits the left adjoint (called induction $\Ind_{\alpha}$) and the right adjoint (named coinduction $\Coind_{\alpha}$) such that the corresponding derived functors are well-defined and are isomorphic up to a shift $\sv$ of $\xi$-grading\footnote{Note that on the level of charaters the functor $\sv[1]$ is given by multiplication by $t$ because of the superconvention for the $\xi$-grading.}
			and the homological shift:
			$$
			\RCoind_{\alpha} \simeq \LInd_{\alpha}\circ\sv[1].
			$$
			\item 
			{\rm(Proposition~\ref{prp::spherical}).} 
			The functor $\Res_{\alpha}$ is spherical as per~\cite{Anno,Kuznetsov_Spherical}. Specifically:
			the following morphisms, involving unit and counit transformations, are isomorphisms:
			$$
			\begin{tikzcd}
				\RCoind_{\alpha}\circ\Res_{\alpha} \circ\LInd_{\alpha} \ar[r,"\sim"] & \RCoind_{\alpha}\oplus \LInd_{\alpha} \ar[r,"\sim"] &
				\LInd_{\alpha}\circ\Res_{\alpha}\circ\RCoind_{\alpha}.	
			\end{tikzcd}
			$$		
		\end{itemize}
	\end{theoremA}
	
	The main protagonists of this paper are the adjoint derived autoequivalences $\sT_{\alpha}, \sT_\alpha' \in \mathrm{End}(\bD)$ for each simple positive root $\alpha$, defined by the distinguished triangles:
	\begin{equation}
		\label{eq::T::triangle}
		\begin{array}{c}
			\Id_{\fb[\xi]} \stackrel{\text{unit}}{\to} \Res_\alpha \circ \LInd_{\alpha} \to \sT_{\alpha} \stackrel{+1}{\to} \Id_{\fb[\xi]}[1], \
			\sT_\alpha' \to \Res_{\alpha} \circ \RCoind_{\alpha} \stackrel{\text{counit}}{\to} \Id_{\fb[\xi]} \stackrel{+1}{\to} \sT'_\alpha[1].
		\end{array}
	\end{equation}
	
	We verify (Corollary~\ref{cor::T_i::char}) that on the level of charaters $\sT_\alpha$ is acting by the generator $\hT_\alpha$ of the Hecke algebra. Moreover, from Theorem~\ref{thm::quadr::rel::inttro}, we derive the following isomorphism of derived endofunctors:
	$$
	\sT_{\alpha}' \simeq \cone(\sv^{-1}\circ \Res_\alpha\circ\LInd_{\alpha} \to \Id[-1]) \simeq
	\cone\left(\sv^{-1}\circ \left(\cone(\sT_{\alpha}[-1]\stackrel{\phantom{}^{\phantom{\frac{1}{2}}}}{\to}  \Id)\right) \to \Id[-1]\right) 
	$$
	which categorifies the quadratic {\it "Hecke"} relation known for the generator $\hT_{\alpha}$:	
	\begin{equation}
		\label{eq::Hecke::quadr}
		\hT_{\alpha}^{-1} = t^{-1}(\hT_{\alpha}+1) - 1 \ \Leftrightarrow \ (\hT_\alpha+1)(\hT_\alpha-t) = 0.
	\end{equation}

	Our second result, verified for all root systems except $G_2$, is:
	\begin{theoremA}[Theorem~\ref{thm::DAHA}]
		\label{thm::B}	
		Let $\tI = \{\alpha_1, \ldots, \alpha_r\}$ be the set of simple roots for the root system $\Delta$ of a semisimple (or affine) Lie algebra $\fg$. The autoequivalences $\sT_i := \sT_{\alpha_i}$, their inverses $\sT_i'$, the weight shift autoequivalences $\sX_i := \sX^{\alpha_i}$, and the $(-\xi)$-grading shift $\sv$ satisfy the relations that categorify the one known for the braid group. In other words, we have the following isomorphisms of triangulated endofunctors:
		\begin{itemize}
			\item Braid relations:
			$
			\underbrace{\sT_i \circ \sT_j \circ \ldots }_{m_{{i,j}} \text{ factors}} \simeq \underbrace{\sT_j \circ \sT_i \circ \ldots }_{m_{{i,j}}\text{ factors}} \text{ with }m_{{i,j}} = ord_{W}(s_is_j).$
			\item Weight shifts commute with each other and for each $\mu=\sum k_i \alpha_i$ with $k_i\in \bZ$ we have a well-defined weight shift functor
			$\sX^{\mu}: =(\sX_1)^{k_1}\circ \ldots \circ (\sX_r)^{k_r}$
			\item 
			$\begin{cases}
				{\sT_i\circ \sX^{\mu} \simeq \sX^{\mu}\circ \sT_i,} {~{\text if}~ \langle \mu,\alpha_i^{\vee} \rangle=0; }
				\\
				{ \sT_i\circ\sX^{\mu} \simeq \sX^{\mu-\alpha_i}\circ\sT'_i\sv[1],}{~{\text if}~ \langle \mu,\alpha_i^{\vee} \rangle=1.}
			\end{cases}$
		\end{itemize}	
	\end{theoremA}
	
	Theorems~\ref{thm::quadr::rel::inttro} and \ref{thm::B} yields a categorification of the affine Hecke algebra, a quotient of the affine braid group's group ring by quadratic Hecke relations~\eqref{eq::Hecke::quadr}. The categorification of the double affine Hecke algebra $\DAHA$ follows naturally, with operators $\hY^{\mu} \in \DAHA$ indexed by the coroot lattice $\sfQ^{\vee}$ having well-defined categorifications as endofunctors $\sY^{\mu}$ of the derived category of supercurrent modules over the Iwahori subalgebra $\fI[\xi]$.
	Our third result states:
	\begin{theoremA}[Theorem~\ref{thm::Y_eigen}]
		For each dominant weight $\lambda$,
		there exists a well defined object $\sEM_{\lambda}\in \bD^{-}(\fI[\xi]\text{-gmod})$ such that the components of weight $\mu$ vanish in its cohomology whenever $\mu\not\preceq\lambda$. In addition,
		\begin{itemize}
			\setlength{\itemsep}{0em}	
			\item The subspace of weight $\lambda$ in $H^{s}(\sEM_{\lambda})$ is equal to zero for $s\neq 0$, and is one-dimensional for $s=0$;
			\item For all  $\mu\in Q^{\vee}_{+}$ there exists an isomorphism $\sY^{\mu}(\sEM_{\lambda}) {\cong} \sq^{-\langle\mu,\lambda\rangle}(\sEM_{\lambda}).$ Here $\sq$ -- denotes the weight shift with respect to the $z$-grading. 
		\end{itemize}	
	\end{theoremA}	
	
	Since nonsymmetric Macdonald polynomials are eigenfunctions for $\hY^{\mu} \in \DAHA$, the graded supercharacter of $\sEM_{\lambda}$ equals $E_{\lambda}$. 
	Our fourth categorification result involves the Cherednik symmetrization operator.
	For each subset $\tJ$ of the set of simple roots $\tI$ and the corresponding parabolic subagelbra $\fp_\tJ$ the restriction functor 
	$
	\Res_{\tJ}:=\Res_{\fb[\xi]}^{\fp_{\tJ}[\xi]}: \fp_{\tJ}[\xi]\ttt\text{mod} \rightarrow \fb[\xi]\ttt\text{mod}
	$ 
	admits left and right adjoint functors whose derived functors are denoted by $\LInd_{\tJ}$ and $\RCoind_{\tJ}$, respectively.
	\begin{theoremA}[Theorem~\ref{thm::PJ::categorify}]
		The triangulated endofunctor $\sD_{\tJ}:=\Res_{\tJ}\circ\LInd_{\tJ}$ "categorifies" the Cherednik symmetrization operator 
		$\hP_{\tJ}:=\sum_{w\in W_{\tJ}}\hT_{w}\in \DAHA$
		and, in particular, satisfy the following relations:
		$$
		\forall i\in \tJ \text{ we have an isomorphism } 
		\sT_{i} \circ \sD_{\tJ} \simeq \sv \sD_{\tJ}[1].
		$$
	\end{theoremA}	
	As a corollary, we get that the complex $\sPM_{\lambda}:=\LInd_{\fg}(\sEM_{\lambda})\in\bD(\fg[z,\xi]\text{-gmod})$ ($\lambda \in \sfP_+$) categorifies the corresponding Macdonald polynomials $P_{\lambda}$ (Corollary~\ref{cor::MM}.

	\subsection{Conjectures verified for $\fg=\msl_2$}
	The following list of conjectures should persuade the reader to dwell  on the categorifications we suggest:
	
	\begin{conjectureA}
		\label{conj::EM}	
		For each $\lambda\in\sfP$ with $\lambda=u_\lambda \lambda_{-}$ (where $\lambda_{-}$ is antidominant and $u_{\lambda}\in W$ is the shortest element in the finite Weyl group)  there exists a complex of graded $\fI[\xi]$-modules $\sEM_{\lambda}$ yielding the following properties:
		\begin{itemize}
			\item The weight support of the cohomology of $\sEM_{\lambda}$ is less or equal to $\lambda_{-}$;
			\item For all $\mu\in \sfQ^{\vee}\cap\sfP_{+}$ there exists an isomorphism 
			$\sY^{\mu}(\sEM_{\lambda}) \simeq \sv^{\frac{l(\mu)+\langle u_{\lambda}^{-1}(2 \rho),\mu \rangle}{2}}\sq^{-\langle \lambda,\mu\rangle} \sEM_{\lambda}[s],$
			where $2\rho=\sum_{\beta \in \Delta_+}\beta$ and $l(\mu)$ is the length of the coweight $\mu$ considered as an element of the affine Weyl group $W^{\af}$ and $[s]$ is an appropriate homological shift.
		\end{itemize}
	\end{conjectureA}
	We explain the existance of $\sEM_{\lambda}$ for dominant $\lambda$ in~\S\ref{sec::Y::eigen}, the similar arguments can be used to construct $\sEM_{\lambda}$ for antidominant $\lambda$, however, these arguments does not work for general $\lambda$.
	The supercharacter of the conjectured complex $\sEM_{\lambda}$  should coincide with Macdonald polynomial $E_{\lambda}$ up to a common factor depending on $q$ and $t$ and this is why we can call them Macdonald modules.
	
	We also believe that the complexes $\sPM_{\lambda}$ defined above as $\LInd_{\fg}(\sEM_{\lambda})$ (for dominant $\lambda$) are eigen objects with respect to the action of the endofunctor $\sD_{\tI}\sY^{\mu}\sD_{\tI}$, just as the Macdonald polynomials $P_{\lambda}$ are eigenfunctions for the operators $\hP_{\tI}\hY^{\mu}\hP_{\tI}\in\DAHA$. An essential difference to the nonsymmetric Macdonald polynomials is that the corresponding eigenvalue is not a monomial in $q$ and $t$ and the precise meaning of the categorical eigenfunction requires a delicate description which we want to omit here.
	
	Moreover, any bounds of the vanishing of hom-spaces between different Macdonald complexes $\sEM_{\lambda}$ (whose existence we conjecture) will imply  the categorification of the orthogonality of nonsymmetric polynomials as outlined in the following 
	\begin{conjectureA}
		\label{conj::Ext::1}	\label{conj::Ext::2}
		For all pairs of integral weights $\lambda\neq \mu$  the following vanishing property holds:
		\[\hom^{\udot}_{\fI[\xi]}(\sEM_{\lambda},\sEM_{\mu})= 0.\]
		\\
		Respectively, we have vanishing of the derived hom-spaces 
		\(\hom^{\udot}_{\fg[z,\xi]}(\sPM_{\lambda},\sPM_{\mu})= 0\)		
		for all pairs of different integral dominant weights $\lambda$ and $\mu$.
	\end{conjectureA}
	
	The good news is that all aforementioned conjectures are verified for  $\fg=\msl_2$
	moreover, everything can be categorified on the level of modules rather than complexes.
	For each $k\in\bZ$ the $\fI[\xi]$-module $\sEMloc_{k\omega}$ is said to be a cyclic module generated by a cyclic vector $w_{k\omega}$ of the weight $k\omega$ subject to the following list of relations:
	\[
	\begin{cases}
		e z^a \xi^b w_{k \omega}=0, a \geq 0, b ={0,1};~
		(fz)^{k}w_{k \omega}=0;~ h\xi w_{k \omega}=0, \text{ if } k\geq 0,
		\\
		f z^a \xi^b w_{k \omega}=0, a >0, b ={0,1};~
		e^{-k+1}w_{k \omega}=0;~ h\xi w_{k \omega}=0.	 \text{ if } k<0.
	\end{cases}	
	\]
	Let us outline the properties of the modules $\sEMloc_{k\omega}$ we proved:
	\begin{enumerate}
		\item {\rm({\it Theorem~\ref{character}}).} The supercharacter of the cyclic $\fI[\xi]$-module $\sEMloc_{k\omega}$ is equal to the integral form of the nonsymmetric Macdonald polynomial $E_{k\omega}$. In particular, we have $\dim\sEMloc_{k\omega} = \dim \sEMloc_{-(k+1)\omega} = 4^k$ for $k\in \bZ_{\geq 0}$.
		\item {\rm ({\it Theorem~\ref{thm::Y_eigen::sl2}}).} There are isomorphisms for all $k\in\bZ_{\geq 0}$:
		\[
		\sY^\omega (\sEMloc_{k\omega})\simeq
		\sq^{-k}\sEMloc_{k\omega};  \quad
		\sY^{-\omega} (\sEMloc_{-k\omega})\simeq	\sv^{-1}\sq^{k}  \sEMloc_{k\omega}[-1].
		\]  
		\item {\rm({\it Corollary~\ref{extwanishing}})} The $Ext$-vanishing property
		$\hom^{\udot}(\sEMloc_{k\omega},\sEMloc_{m\omega})= 0$
		holds for all $k\neq m\in \bZ$.
	\end{enumerate}

	\subsection{Relation to the geometry of flag manifolds}
	
	Let us discuss the relation between our construction and the geometry of flag manifolds. Let $G$ be the simple algebraic group such that $\fg = \mathrm{Lie} \, G$, and let $B \subset G$ be the Borel subgroup such that $\fb = \mathrm{Lie} \, B$. We consider the truncation Lie algebra
	$$\fb [\epsilon] := \frac{\fb \otimes \bC [z]}{\fb \otimes \bC [z]z^2} \cong \fb \ltimes \fb,$$
	where the last Lie algebra contains (the second, degree one copy of) $\fb$ as an ideal with a trivial bracket. We similarly define $\fb \ltimes \fb^*$ by the action $\fb \circlearrowright \fb^*$. Let $\mathbf I$ be the direct sum of injective objects in $\fb\text{-gmod}$ prolonged to $\fb[\xi]\text{-gmod}$ trivially. We have an equivalence of categories
	$$\bD^+ ( \fb[\xi]\text{-gmod} ) \longrightarrow \bD^- ( \fb \ltimes \fb^* \text{-gmod} )^{op}$$
	sending $M^{\bullet} \mapsto Ext^{\bullet} _{\fb[\xi]\text{-gmod}} ( M^{\bullet}, \mathbf I )$. The abelian category $\mathsf{Coh}_{B \times \mathbb G_m} \, \fb$ of $(B \times \mathbb G_m)$-equivariant coherent sheaves on an affine space $\fb$ has a fully faithful embedding
	$$\mathsf{Coh}_{B \times \mathbb G_m} \, \fb \hookrightarrow \fb \ltimes \fb^* \text{-gmod}$$
	by differentiating the $B$-action and regarding the $\mathbb G_m$-action as a grading. We can rewrite the LHS as
	$$\mathsf{Coh}_{B \times \mathbb G_m} \, \fb \stackrel{\cong}{\longleftarrow} \mathsf{Coh}_{G \times \mathbb G_m} \, ( G \times_B \fb )$$
	by the $G$-action, where $( G \times_B \fb )$ is the $G$-equivariant vector bundle on $G/B$ whose fiber at the point $B/B \in G/B$ is $\fb$, and the functor is the restriction to the fiber at $B/B$.
	
	Note that the fiber of $T^* (G/B)$ at $B/B$ is $( \g / \fb )^* \cong \fn$, that induces an embedding
	$$\imath : T^* (G/B) \subset G \times_B \fb$$
	through a $( B \times \mathbb G_m )$-equivariant embedding of affine spaces $\fn \subset \fb$. The categorical affine Hecke algebra action (\cite{BR13}) on the LHS of
	\begin{equation}
		\bD^b ( \mathsf{Coh}_{G \times \mathbb G_m} \, T^* (G/B) ) \stackrel{\imath_*}{\longrightarrow} \bD^b ( \mathsf{Coh}_{G \times \mathbb G_m} \, ( G \times_B \fb ) )\label{g-cat-emb}
	\end{equation}
	consists of two portions: one is the twist by $(G \times \mathbb G_m)$-equivariant line bundle on $G/B$ pulled back to $T^* (G/B)$, while the other (corresponding to our $\sT_i$) is induced by the composition
	\begin{equation}
		( s_i )_* \circ p_i^* \circ ( p_i )_* \circ s_i^*\label{g-comp}
	\end{equation}
	of push-pull functors using the natural $(G \times \mathbb G_m)$-equivariant maps:
	$$T^* ( G/ B ) \stackrel{s_i}{\longleftarrow} G \times_B \mathfrak n_i \stackrel{p_i}{\longrightarrow} G \times _{P_i} \mathfrak n_i,$$
	where $P_i \supset B$ is a minimal parabolic subgroup of $G$ and $\mathfrak n_i \subset \mathfrak n$ is the nilradical of the Lie algebra of $P_i$. These two actions naturally prolong to the RHS of (\ref{g-cat-emb}). We can check that our $\sT_i$-action is also related to an incarnation of (\ref{g-comp}), and the line bundle twists on $T^* (G/B)$ and the character twists on $\mathsf{Coh}_{B \times \mathbb G_m} \, \fn$ $(\subset \mathsf{Coh}_{B \times \mathbb G_m} \, \fb)$ correspond to each other.
	
	The geometric construction in \cite{BR13} itself extends to the case of affine flag variety $X$ (or more general Kac-Moody flag variety). In the case of affine flag variety, we have a categorical action of $\DAHA$ on the derived category of the equivariant coherent sheaves on $T^*X$. This triggers at least two problems in our considerations. The first problem is that one needs to find some class of category that cannot be the category of coherent sheaves if one wants to obtain some numerical consequence (see e.g. \cite{VV}). In particular, the tie between the algebraic and geometric constructions explained above (that was {\it not} an equivalence from the beginning) becomes even less transparent. The second problem is that one needs to be more careful about objects corresponding to the Macdonald polynomials. We have operators corresponding to $\sX$ and $\sY$ in the categorical $\DAHA$-action. The joint eigenvectors of the action of $\sX$, coming from the character twists, correspond to the torus fixed points of $X$. However, this is only via the localization theorem, and it does not have a straight-forward meaning in $\fg[z,\xi]\text{-gmod}$ (as its objects correspond to equivariant objects in $T^* X$, that cannot be supported on a point). The action of $\sY$, coming from the composition of $\sT_i$s, has some meaning in $\fg[z,\xi]\text{-gmod}$, but its geometric meaning in terms of $T^*X$ is rather unclear and not useful to analyze its eigenobjects.
	
	These motivate us to provide an algebraic framework that is morally equivalent to the derived category of $T^*X$ that captures the object corresponding to the Macdonald polynomials with respect to the action of $\sY$ (as in the treatment of Macdonald polynomials in \cite{MacDonald}).

	\subsection{Structure of the Paper}
	
	The paper is organized as follows:
	
	Section~\ref{sec::Notations} covers the foundational results and notations from Lie theory (\S\ref{sec::Lie::setup}) and (Double) Affine Hecke Algebras (\S\ref{sec::AHA::def}, \S\ref{sec::DAHA::def}). It also includes definitions of symmetric (\S\ref{sec::Macdonald::def}) and nonsymmetric (\S\ref{sec::order}) Macdonald polynomials. Additionally, we explain (\S\ref{sec::decategorify}) how the ring of characters of the Lie superalgebra of currents $\fb[\xi]$ (resp. $\fI[\xi]$) coincides with the polynomial ring $\bZ_t[\sfP]$ (resp. $\bZ_{q,t}[\sfP]$), which motivates our subsequent categorifications.
	
	The primary triangulated functors—restriction $\Res_i$, derived induction $\LInd_i$, and their composition $\sD_i:=\Res_i\circ\LInd_i$—are defined in~\S\ref{sec::Demazure::all}. We demonstrate that these functors are spherical (\S\ref{sec::sherical}) and categorify the Hecke quadratic relation. All computations in~\S\ref{sec::Demazure::all} are derived from rank~$1$ system ($\msl_2$).
	
	A categorification of the Affine Hecke Algebra is detailed in~\S\ref{sec::proof}. We begin with an explanation of the common methods in~\S\ref{sec::strategy} and prove the categorification of each equality in separate subsections.
	
	A categorification of the Cherednik symmetrization operator $\hP_{\tJ}$ is proposed in~\S\ref{sec::symm::categorify}, where we also derive different properties of the derived induction $\LInd_{\tJ}$.
	
	Section \ref{sec::DAHA::Macdonald} is dedicated to a categorification of (nonsymmetric) Macdonald polynomials. We adapt the categories and triangulated functors discussed in~\S\ref{sec::Ind} to the affine setting in~\S\ref{sec::affine::O} and~\S\ref{sec::bD_c^-}.  This section concludes with a description of the desired categorification of the nonsymmetric Macdonald polynomials for dominant weights, and together with the symmetrization functor. Therefore, we achieve a categorification of symmetric Macdonald polynomials here.
	
	The final Section~\ref{sec::sl2::MM} provides a detailed description of all objects from~\S\ref{sec::DAHA::Macdonald}, specifically for $\fg=\msl_2$. Unexpectedly, for $\msl_2$, one can categorify both symmetric and nonsymmetric Macdonald polynomials using cyclic modules, whose description is presented in~\S\ref{sec::EM::sl2}, the derived setting for $\msl_2$ case is not necessary.

	\subsection*{Acknowledgements}
	We would like to thank I.Anno, A.Bondal, B.Feigin, A.Kuznetsov for stimulating discussions.
	We thank I.Marshall and E.Feigin for useful comments on the exposition of the text.
	Ie.\,M. was partially supported by the grant RSF 19-11-00056 and JSPS KAKENHI Grant Number 18F18014. S.\,K was partially supported by JSPS KAKENHI Grant Number JP19H01782.

	\section{Notations, recollection and a categorification of $\bZ_t[\sfP]$}
	\label{sec::Notations}
	
	We work over an algebraically closed field $\Bbbk$ of characteristic zero. It is worth mentioning that many of our constructions are well-defined for positive characteristics and will be discussed elsewhere.
	
	\subsection{Notation conventions}
	\label{sec::conventions}
	It is worth mentioning to say that there are many different objects, that are denoted in the literature by the same letter.
	For example, the standard notations for the weight lattice, the Macdonald polynomial, and the Cherednik symmetrization functor use the capital letter $P$.
	We use different fonts to separate objects of different nature:
	\begin{itemize}
		\item the \emph{mathsf} font is used for root and weight lattices $\sfQ\subset \sfP$; 
		\item the font \emph{mathpzc} is used for elements of Hecke algebra: e.g. $\hT_i$, $\hP$, $\hX$, $\hY$;
		\item the \emph{ordinary} font is used for Macdonald polynomials: e.g. $P_\lambda$, $E_{\lambda}$;
		\item the \emph{bold} letters are used for the (derived) functors: e.g. $\sD_{i}$, $\sT_i$, $\sX^{\lambda}$, $\sv$;
		\item the \emph{mathbb} letters are used for modules and complexes. E.g. $\sPM_{\lambda}$, $\sEM_{\lambda}$ are symmetric and nonsymmetric Macdonald complexes.
	\end{itemize}
	
	\subsection{Setup from Lie theory}
	\label{sec::Lie::setup}
	\subsubsection{Finite-dimensional semi-simple Lie algebras}
	\label{sec::Setup::finite}
	Suppose that $\fg$ is a simple Lie algebra that is not of type $\mathsf{G}_2$. We denote its set of roots, together with a chosen set of positive roots by $\Delta=\Delta_+\sqcup \Delta_-$ ($\Delta_- = - \Delta_+$). Let $\Pi:=\{\alpha_1,\ldots,\alpha_r\}\subset \Delta_{+}$ be the set of simple roots and $\tI:=\{1,\ldots,r\}$ be the indexing set of simple roots. These data uniquely define the Cartan decomposition:
	$\fg=\fn_{-}\oplus\fh\oplus\fn_{+} = \fn_{-}\oplus\fb_{+}$ (we will mostly omit the index for the positive Borel subalgebra by identifying $\fb$ and $\fb_{+}$).
	The $\msl_2$-triple assigned to a positive root $\alpha_i\in\Delta_{+}$ is denoted by $\langle e_i=e_{\alpha_i},h_i,f_i:=e_{-\alpha_i}\rangle$, where $e_{\beta}$ is a fixed root vector for every $\beta \in \Delta$.
	
	Let $\sfP \subset \fh^*$ be the integral weight lattice of $\fg$. Let $\sfQ \subset \sfP$ be the sublattice spanned by $\Delta$ (root lattice). 	We define $\Pi^{\vee} \subset \sfQ ^{\vee}$ to be the set of positive simple coroots, and let $\sfQ_+^{\vee} \subset \sfQ ^{\vee}$ be the set of non-negative integer span of $\Pi^{\vee}$.
	Let $W$ be the Weyl group of $\fg$. We fix bijections $\tI \cong \Pi \cong \Pi^{\vee}$ so that $i \in \tI$ corresponds to $\alpha_i \in \Pi$, its coroot $\alpha_i^{\vee} \in \Pi ^{\vee}$, and a simple reflection $s_i \in W$ corresponding to $\alpha_i$. Let $\sfQ^{\vee}$ be the dual lattice of $\sfP$ with a natural pairing $\left< -, - \right> : \sfQ^{\vee} \times \sfP \rightarrow \bZ$. 
	Denote by $\sfP_{+}$ -- the set of dominant weights, i.e. $\lambda\in \sfP_{+}$ iff $\forall\alpha_i\in\Pi$ we have $\langle \lambda,\alpha_i^{\vee}\rangle\geq 0$. Let $\{\varpi_i\}_{i \in \tI} \subset \sfP_+$ be the set of fundamental weights (i.e. $\left< \al_i^{\vee}, \varpi_j \right> = \delta_{ij}$). We denote by $\bG$ the simply connected algebraic group over $\Bbbk$ whose Lie algebra is $\fg$. 
	For each $\tJ \subset \tI$, we have the root subsystem $\Delta_{\tJ}=\Delta_{\tJ}^{-}\sqcup\Delta_{\tJ}^{+}\subset\Delta$ generated by the subset of simple roots $\{\alpha_i\colon i\in\tJ\}$. Respectively, the parabolic Lie subalgebra $\fp_{\tJ} \subset \fg$ is generated by $\fb_{+}$ and the root elements $\{f_{i}\}_{i \in \tJ}$. In case $\tJ = \{ i \}$, we might write $\fp_{i}$ or $\fp_{\al_i}$ instead of $\fp_{\{i\}}$. Let $\msl_2^{i}$ denote the span of $e_i,f_i,h_i$ in $\fp_{\al_i}$, that forms a Lie subalgebra isomorphic to $\msl_2$. We have a subalgebra $\fb_i \subset \msl_2^{i}$ spanned by $e_i,h_i$. We set $\g_\tJ$ to be the maximal semi-simple Lie subalgebra of $[\fp_{\tJ},\fp_{\tJ}]$ normalized by the action of $\fh$. 
	The corresponding algebraic groups are denoted by $\bB \subset \bP_{\tJ} \subset \bG$.
	Let $\fg_{\tJ}:=\fn_{\tJ}^{-}\oplus \fh \oplus \fn_{\tJ}^{+}$ be the Cartan decomposition of the reductive Lie algrebra $\fg_{\tJ}$.
	While working with representations of parabolic Lie subalgebras we will also need the following notation:
	$\rho_{\tJ}$ is the half-sum of the positive roots of $\fg_{\tJ}$.
	
	For a $\h$-semisimple module $M$ whose $\h$-weights belongs to $\sfP$, we define the weight support $\Psi ( M )$ of $M$ as the set of $\h$-weights with non-zero weight space in $M$.
	
	\subsubsection{Affine Lie algebras}	
	\label{sec::Affine::Lie}
	Denote by $\fgh$ the untwisted affine Lie algebra corresponding to $\fg$. We set $\fg [z^{\pm 1}] := \fg \otimes \Bbbk[z^{\pm 1}]$. This acquires the natural structure of a Lie algebra. The Lie algebra $\fg [z^{\pm 1}]$ admits a natural surjection from $[\fgh, \fgh]$ by taking the quotient by the center.
	
	Let $\Delta_{\af} := \Delta \times \bZ \delta \cup \{m \delta\}_{m \neq 0}$ be the untwisted affine root system of $\Delta$ with its positive part $\Delta_+ \subset \Delta_{\af, +}$. We set $\alpha_0 := - \vartheta + \delta$, $\Pi_{\af} := \Pi \cup \{ \alpha_0 \}$, and $\tI_{\af} := \tI \cup \{ 0 \}$, where $\vartheta$ is the highest root of $\Delta_+$.
	Let $\fI \subset \fgh$ be the (image of the) upper-triangular subalgebra of $[\fgh, \fgh]$ (that contains $\fh$; the Iwahori subalgebra). We set $\bU := U ( \fI )$ (the enveloping algebra of $\fI$). For each $\tJ \subsetneq \tI_\af$, we have a Lie subalgebra $\widehat{\fp}_{\tJ} \subset \fg [z^{\pm 1}]$ generated by $\fI$ and the root subspaces of $\{- \al_i \}_{i \in \tJ}$.
	
	We set $W _{\af}:= W \ltimes \sfQ^{\vee}$ and call it the affine Weyl group. It is a reflection group generated by $\{s_i \mid i \in \tI_{\af} \}$, where $s_0$ is the reflection with respect to $\alpha_0$. We also have a reflection $s_i \in W_\af$ corresponding to $\alpha \in \Delta \times \bZ \delta \subsetneq \Delta_\af$. Let $\ell: W_\af \rightarrow \bZ_{\ge 0}$ be the length function (based on $\tI_\af$) and let $w_0 \in W$ be the longest element in the finite Weyl group $W \subset W_\af$. Together with the normalization $t_{- \vartheta^{\vee}} := s_{\vartheta} s_0$ (for the coroot $\vartheta^{\vee}$ of $\vartheta$), we introduce the translation element $t_\beta \in W _{\af}$ for each $\beta \in \sfQ^{\vee}$. We set $\Omega$ to be the group of diagram automorphism of the affine Dynkin diagram of $\g$ such that $\Omega\ltimes W_\af$ gives the extended affine Weyl group $\widetilde{W}_\af$. 
	We have
	$$\widetilde{W} _\af \cong W \ltimes \sfP^{\vee} \supset W_\af = W \ltimes \sfQ^{\vee}.$$
	Every element of $\widetilde{W}_\af$ is written as $w \pi$ for $w \in W_\af$ and $\pi \in \Omega$. As $\Omega$ preserves $\{s_i \mid i \in \tI_{\af} \}$, we define
	$$\ell ( w \pi ) = \ell ( \pi w ) := \ell ( w ).$$
	
	Recall that $\Omega$ permutes the set of level one dominant weights of $\fgh$, and their projections to $\fh^*$ induces $\Omega$-permutations of a finite subset $\Lambda_0$ of $\sfP$.
	Below we consider the level one action of $\widetilde{W}_\af$ on $\sfP$, that extends the $\Omega$-action on $\Lambda_0$:
	\[
	s_i \la := \begin{cases}\la - \left< \al_i^{\vee}, \la \right> \al_i & (i \in \tI) \\ \la + ( \left< \vartheta^{\vee}, \la \right> + 1 ) \al_i & (i = 0)\end{cases}.
	\]

	\subsection{Affine Hecke Algebra (AHA)}
	\label{sec::AHA::def}	
	With each root system $\Delta$ one assigns a braid group $B_{\Delta}$ whose set of generators $\{\hT_i | i\in \tI\}$ coincides with the generators $\{s_i | i\in\tI\}$ of the Weyl group $W$ and satisfy only the braid relations:
	\begin{equation}
		\label{eq::rel::braid0}		
		\underbrace{\hT_i \circ \hT_j \circ \ldots }_{m_{{i,j}} \text{ factors}} \simeq \underbrace{\hT_j \circ \hT_i \circ \ldots }_{m_{{i,j}}\text{ factors}} \text{ with }m_{{i,j}} = ord_{W}(s_is_j). 		
	\end{equation}
	The Hecke algebra $\hH_{\Delta}$ is the quotient of the group ring of the braid group $B_{\Delta}$ with its coefficient ring $\bZ[t,t^{-1}]$ subject to the quadratic relation:
	\begin{equation}
		\label{eq::Hecke::alg}
		\hT_i^{-1} = t^{-1} \hT_i + t^{-1} -1 \Leftrightarrow (\hT_i+1)(\hT_i-t) =0. 			
	\end{equation}
	The Affine braid group is an extension of the braid group $B_{\Delta}$ with the weight lattice $\sfP$, considered as a free abelian group\footnote{Sometimes the affine braid group of $\g$ refers to the braid group corresponding to $W \ltimes \sfQ^{\vee}$. We use the dual conventions which are more convenient for our purposes.}. We denote these translation elements $\{\hX^{\mu} |\mu\in\sfP\}$.
	Respectively, the Affine Hecke algebra is the semidirect product of the Hecke algebra $\hH_{\Delta}$ and the group ring of the weight lattice defined through the following relations:
	\begin{gather}	
		\label{eq::rel::braid1}		
		\hT_i \hX^{\mu} = \hX^{\mu} \hT_i , \text{ if } \langle\mu,\alpha_i^{\vee}\rangle =0; \\
		\label{eq::rel::braid2}	 \hT_i \hX^{\mu} = t \hX^{\mu-\alpha_i} \hT_i^{-1}, \text{ if } 	\langle\mu,\alpha_i^{\vee}\rangle =1.
	\end{gather}
	
	\begin{rem}
		There are many different renormalizations of parameters in Hecke algebras (see e.g.~\cite{Lusztig_AHA}, \cite{MacDonald}, \cite{Ch1}).
		First, in general, one uses several parameters $t_{\alpha}$ with $\alpha\in\Delta$ and $t_\alpha=t_{\sigma(\alpha)}$ for all $\sigma\in W$, that might depend on the length of a simple root. It is important for our purposes, that all $t_i$ are equal to a single variable $t$.
		Second, to have a more symmetric presentation of Hecke Relation~\eqref{eq::Hecke::alg} one  frequently uses the generators $\widetilde{\hT}_i:=t^{-1/2}\hT_i$ instead of $\hT_i$:
		\[(\widetilde{\hT}_i + t^{-1/2}) (\widetilde{\hT}_i - t^{1/2}) = 0 \Leftrightarrow (\hT_i+1)(\hT_i-t)=0\]
	\end{rem}
	The AHA has a standard (\emph{basic}) polynomial representation $\bZ_t[\sfP]:=\bZ[\sfP]\otimes\bZ[t,t^{-1}]$ (see e.g.~\cite[\S4.3]{MacDonald},~\cite[p.310]{Ch}) with the following (normalized) action of generators $\hT_i$:
	\begin{equation}
		\label{eq::T::action::Hecke}
		t s_i+\frac{1-t}{1-\hX^{\alpha_i}}(s_i-1).
	\end{equation}
	The meaning of this formula is that $(s_i-1)$ acts prior to $\frac{1-t}{1-\hX^{\alpha_i}}$ (what is important because $\hX^{\mu}$ and $s_i$ does not commute). In particular, if $\langle\mu,\alpha_i^{\vee}\rangle > 0$ we have
	\begin{multline}
		\label{eq::T::X::action::basic}	
		\hT_i (\hX^{\mu}) = t \hX^{s_i(\mu)} + (1-t)\frac{\hX^{s_i(\mu)} - \hX^{\mu}}{1-\hX^{\alpha_i}} = \\
		= t \hX^{s_i(\mu)} + (1-t)\left(\hX^{s_i(\mu)} + \hX^{s_i(\mu)+\alpha_i}+\ldots + \hX^{\mu-\alpha_i}\right).
	\end{multline}
	In particular, the action of $\hT_i+1$ in the basic representation commutes with the multiplication by the $s_i$-symmetric polynomial in $\hX$ (see e.g.~\cite{Ch}).
	
	One of the main purposes of this paper is to describe the appropriate derived category $\bD$ together with a collection of endofunctors $\{\sT_i, i\in\tI\}$ (resp. in $\tI_{\af}$), $\sX^{\mu}$ (with $\mu\in\sfP$), $\pi\in\Omega$ and $\sv$ such that the \emph{decategorification} of $\bD$ is the polynomial representation $\bZ_t[\sfP]$, and the \emph{decategorification} of endofunctors coincides with $\hT_i$, $\hX^{\mu}$ respectively.
	By decategorification, we mean taking the character of a module or the Euler characteristic of modules while dealing with complexes (see the details below).
	
	\subsection{Lie algebras of supercurrents and decatecorification}
	\label{sec::decategorify}	
	Let $\Bbbk[\xi]$ be the polynomial superalgebra in one odd variable $\xi$. We use the shorter notation $\fg[\xi]$ for the corresponding current Lie superalgebra	$\fg \otimes \Bbbk[\xi]$
	and, moreover,  for $x\in\fg$ and $p\in \Bbbk[\xi]$ we write $xp$ instead of $x \otimes p$ and $x$ stands for $x\otimes 1$. The parity of $\xi$ predicts the following decomposition of $\fg$-modules:
	$$\fg[\xi]=\fg \oplus \fg \xi,$$
	that respects the $\xi$-grading on $\fg[\xi]$ with $\deg_{\xi}(\fg)=0$ and $\deg_{\xi}(\fg\xi)=1$. A $\fg[\xi]$-module $M$ is said to be $(\fh,\xi)$-graded if $M$ admits a semi-simple $( \fh \oplus \bC \xi )$-eigendecomposition whose eigenvalues belong to $\sfP \oplus \bZ$.
	We denote by $\mathcal{O}=\mathcal{O}(\fb[\xi])$ the category of $(\fh,\xi)$-graded finitely generated $\fn_{+}$-{finite} $\fb[\xi]$--modules.
	That is, a $\bZ$-graded finitely generated $\fb[\xi]$-module $M$ belongs to $\mathcal{O}$ if it is a $\fh$-semisimple $\fb$-module with locally nilpotent actions of elements $e_i$ for all $\alpha\in\Delta_{+}$. Moreover, $M\simeq \oplus M_k$, where $M_k$ is the $k$'th graded component with respect to the $\xi$-grading and each $M_k \simeq \oplus_{\lambda\in\sfP} M_k^{(\lambda)}$ is the decomposition with respect to the $\fh$-action.
	In particular, taking the $(\fh,\xi)$-character of a module is the  \emph{decategorification} of our constructions:
	$$
	\chi: M \mapsto \sum_{\lambda,k} \dim M_{k}^{(\lambda)} \hX^{\lambda} (-t)^{k} \in \bZ_{t}[\sfP],
	$$	
	where $\bZ_{t}[\sfP]:=\bZ[\sfP]\otimes_{\bZ}\bZ[t,t^{-1}]$.
	The main object of study for us is the appropriate derived category associated with the aforementioned abelian one. If $\fg$ is finite we consider the ordinary bounded derived category $\bD^{b}(\cO(\fb[\xi]))$. 
	If $\fg$ is replaced with an affine Lie algebra, then unfortunately, the bounded derived category does not contain all objects we are interested in, however, we may consider the certain subcategory of $\bD^{-}(\cO(\fb[\xi]))$ for which we can take the character: complexes that are bounded in each $\fh$-graded component and have the total bounds on the region of $\fh$-weights (see \S\ref{sec::bD_c^-} for details). 
	\begin{prop_def}
		\label{prp::def::decategorify}	
		The category $\bD^{b}(\mathcal{O}(\fb[\xi]))$ is considered as a categorification of the polynomial ring $\bZ_{t}[\sfP]$. Where 
		\begin{itemize}
			\item 
			the Euler characteristic of characters defines a \emph{decategorification} map:
			$$
			\chi: M^{\udot} \mapsto \sum_n (-1)^{n} \chi(M^{n}) \in \bZ_{t}[\sfP].
			$$	
			\item The tensor product with a one-dimensional module $\Bbbk_{\lambda}$ defines a shift of $\fh$-grading and categorifies the monomial $\hX^{\lambda}$. The corresponding functor is denoted by $\sX^{\lambda}$ and is well defined for both abelian and derived categories;
			\item The functor of the shift of a $\xi$-grading  $M \mapsto M\{1\}$ (where $M\{1\}_k:=M_{k-1}$) is denoted by $\sv$ and categorifies the $(-t)$ element of the polynomial ring $\bZ_{t}$;
			\item The homological shift $M^{\udot}\to M[1]^{\udot}$ gives the minus sign for the decategorifications.
		\end{itemize}
	\end{prop_def}		
	\begin{proof}
		The functors $\sX^{\lambda}\sv^m$ defines the action of the free abelian group $\sfP\oplus \mathbb{Z}$ on $\cO(\fb[\xi])$:
		\[\sX^{\lambda}\sv^m \circ \sX^{\mu}\sv^n = \sX^{\lambda+\mu}\sv^{m+n} \hskip 5mm \lambda, \mu \in \sfP, m, n, \in \mathbb{Z}.
		\]
		The minus sign concerning the character of the $\xi$-grading reflects the fact, that $\xi$ is an odd variable.
	\end{proof}		
	In the subsequent section~\ref{sec::Demazure::all} we will define the endofunctors $\sT_{i}$ ($i \in \tI$) whose action categorifies the action of $\hT_i$ in the polynomial representation $\bZ_{t}[\sfP]$.
	The categorification of relation~\eqref{eq::Hecke::alg} is explained in~\S\ref{sec::sherical} while a categorification of Relations~\eqref{eq::rel::braid0},  ither on the level of characters and will explain the meaning of the categorifications of the braid relations~\eqref{eq::rel::braid0}, \eqref{eq::rel::braid1}, \eqref{eq::rel::braid2} are given in \S\ref{sec::AHA::categorify}. 
	
	\subsection{Macdonald polynomials $P_\lambda$}
	\label{sec::Macdonald::def}
	The ring of symmetric polynomials $\bZ[\sfP]^{W}$ is the ring of characters of finite-dimensional modules over a semi-simple Lie algebra $\fg$ whose Weyl group is $W$.
	The polynomials $\{ m_\lambda=\sum_{\mu\in W\lambda} \hX^{\mu}\}_{\la \in \sfP_+}$ is called the monomial basis of the ring of symmetric functions. Here $\lambda$ belongs to the set of integral dominant weights $\sfP_+$.
	I.\,G.\,Macdonald \cite{MacDonald::Hall} defined a two-parameter pairing $\langle\ttt,\ttt\rangle_{q,t}$ on $\bZ_{q,t}[\sfP]^{W}$ $(:= \bQ (q,t) \otimes_\bZ \bZ[\sfP]^{W})$ and a family of orthogonal
	polynomials $\{P_{\lambda}(q,t)\}_{\lambda\in \sfP_{+}}$ resulting from the Gram-Schmidt orthogonalization process applied to  the monomial basis concerning the standard partial ordering of the set of dominant weights:
	$$
	\langle P_\lambda, P_\mu\rangle_{q,t} = 0 \text{ if } \lambda > \mu \quad \& \quad P_{\lambda} = m_{\lambda} +\sum_{\mu<\lambda} c_{\lambda,\mu} m_{\mu}, \text{ with } c_{\lambda,\mu}\in \bQ(q,t).
	$$
	Though the ordering is partial:
	$$
	\lambda \geq \mu \Leftrightarrow \lambda- \mu = \text{ sum of positive roots},
	$$
	the set of all polynomials are still orthogonal to each other. I.e., we have $\langle P_\lambda, P_\mu\rangle_{q,t} = 0$ for $\lambda \neq \mu$ even if $\lambda$ and $\mu$ are not comparable with respect to $\ge$.
	
	One of the main purposes of this paper is to describe a possible categorification of Macdonald polynomials.
	Following the suggestion of a decategorification described in~\S\ref{sec::decategorify} we consider the category of bigraded modules over the Lie superalgebra of currents $\g\otimes\Bbbk[z,\xi]$. 
	Here the first $q$-grading is assigned to the even variable $z$
	and the second $t$-grading is assigned to the odd variable $\xi$.
	As shown in~\cite{Khor::Weyl} (see also~\cite{FGT}, \cite{Fe1}) the Euler characteristic of extension groups categorifies the Macdonald pairing $\langle\ttt,\ttt\rangle_{q,t}$: 
	\begin{equation}
		\label{eq::hom::pair::intro}
		\langle \gch(M), \gch(N)\rangle_{q,t}=
		\sum_{\begin{smallmatrix}
				i\geq 0, \\ 
				k,l\in\bZ
			\end{smallmatrix}
		} (-1)^{i+l} q^{k} t^{l} \dim Ext^{i}_{\fg\otimes\Bbbk[z,\xi]\ttt gmod}(M\{k,l\}, N^{\dual})
	\end{equation}
	whenever the right-hand side is a well-defined power series.
	Here $\gch(M)$ denotes the graded super character of a bigraded $\fh$-semisimple $\fg[z,\xi]$-module $M$, $M\{k,l\}$ denotes the graded shift of the module $M$, and $N^\dual$ denotes the restricted dual module twisted by the Cartan antiinvolution.
	
	Unfortunately, it is impossible to find a family of bigraded modules whose graded characters are equal to Macdonald polynomials $P_{\lambda}$, because it is known that $P_{\lambda}$ are not Schur-positive for $\lambda$ sufficiently large.
	Therefore, one can expect a categorification in terms of a derived category instead of an abelian one.
	For each dominant $\lambda$ 
	we construct a complex $\sPM_\lambda$ of bigraded $\fg[z,\xi]$-modules whose Euler characteristic of characters is proportional to the Macdonald polynomials $P_{\lambda}$ and whose cohomology does not contain $\fg$-submodules with weights not less than or equal to $\lambda$ (Corollary~\ref{cor::MM}).
	
	\subsection{Convex orders on $\sfP$ and nonsymmetric Macdonald polynomials $E_\lambda$}
	\label{sec::order}
	
	For $\la \in \sfP$, we set $\la_+ \in \sfP$ to be the unique dominant weight in $W \la$ and set $\la_- \in \sfP$ to be the unique anti-dominant weight. We define three partial preorders on $\sfP$ as:
	\begin{equation}
		\label{eq::order}
		\begin{array}{rl}
			\la \le \mu & \Leftrightarrow \la \in \mu - \sfQ_+\\
			\la \lhd \mu & \Leftrightarrow \la_- \neq \mu_- \text{ and } \la_- \in \mu _- + \sfQ_+\\
			\la \preceq \mu & \Leftrightarrow \la \lhd \mu, \text{ or } \la_- = \mu_- \text{ and } \la \ge \mu.
		\end{array}
	\end{equation}
	We call $\le$ the dominance order, $\prec$ the Cherednik order, and $\lhd$ the Macdonald order. Note that we have $\la \lhd \mu$ if and only if $\la \prec \mu_-$. For each $\la \in \sfP$, we consider
	\begin{align*}
		\sfP [\prec \la] & := \{ \mu \in \sfP \mid \mu \prec \la \} \subset \sfP [\preceq \la] :=\{ \mu \in \sfP \mid \mu \preceq \la \}\\
		\sfP [\lhd \la] & := \{ \mu \in \sfP \mid \mu \lhd \la\}.
	\end{align*}
	By the definition of $\lhd$, the set $\sfP [\lhd \la]$ is $W$-stable.
	
	\begin{thm}[Cherednik \cite{Ch1} \S 1]\label{convex}
		Let $i \in \tI_\af$ and $\la, \mu \in \sfP$. We have:
		\begin{enumerate}
			\item if $\mu$ and $\mu - m \al_i$ $(m \ge 0)$ belongs to $\sfP [\prec \la]$, then we have $\mu - c \al_i$ for each $0 \le c \le m$;
			\item if $\la \preceq s_i \la$, then we have $s_i \sfP [\preceq \la] \cup \sfP [\preceq \la] \subset \sfP [\preceq s_i \la]$;
			\item if $\la \succeq s_i \la$, then we have $s_i \sfP [\preceq \la] \subset \sfP [\preceq \la]$.
		\end{enumerate}
	\end{thm}
	
	For $\la \in \sfP$, let $u_\la \in W$ be the minimal length element such that $u_\la \la \in \sfP_-$.

	I.Cherednik introduced a two-parameter nonsymmetric pairing $(\ttt,\ttt)_{q,t}$ on the ring $\bZ_{q,t}[\sfP]$ and a family of orthogonal polynomials $\{E_\lambda(q,t), \lambda\in\sfP\}$ resulting the Gram-Schmidt orthogonalization process applied to the monomials $\hX^{\lambda}$ with respect to the aformentioned partial order $\prec$.
	Despite the order being partial we have
	$$
	\left( E_\lambda, E_\mu\right)_{q,t} = 0 \text{ if } \lambda\neq \mu \quad \& \quad E_{\lambda} = \hX^{\lambda} +\sum_{\mu\prec\lambda} d_{\lambda,\mu} \hX^{\mu}, 
	\text{ with }d_{\lambda,\mu}\in \bQ(q,t).
	$$
	The categorification of the pairing repeats Categorification~\eqref{eq::hom::pair::intro}. One has to consider the graded Ext-pairing of the representations of the Lie super-algebra $\fI[\xi]$ of supercurrents over the Iwahori subalgebra $\fI$ of the affine Lie algebra $\widehat\fg$:
	$$
	\left( \gch(M), \gch(N)\right)_{q,t}=
	\sum_{\begin{smallmatrix}
			i\geq 0, \\ 
			k,l\in\bZ
		\end{smallmatrix}
	} (-1)^{i+l} q^{k} t^{l} \dim Ext^{i}_{\fI[\xi]\ttt gmod}(M\{k,l\}, N)
	$$
	Here $M,N$ are finite-dimensional $\fI[\xi]$-modules and $\gch(M),\gch(N)$ are their graded characters that belongs to the ring $\bZ_{q,t}[\sfP]$.
	Having this in mind we suggest a partial answer to a categorification problem.
	In~\S\ref{sec::Y::eigen} we define complexes $\sEM_\lambda$ of $\fI[\xi]$-modules whose graded character is equal to $E_\lambda$ (for dominant weight $\lambda$).
	However, we use another definition of (nonsymmetric) Macdonald polynomials using the Double Affine Hecke Algebra recalled in the subsequent~\S\ref{sec::DAHA::def}.

	\subsection{Double Affine Hecke Algebras and $\hY$-eigen-functions}
	\label{sec::DAHA::def}
	The Double Affine Hecke Algebra ($\DAHA$) introduced by I.Cherednik is the Affine Hecke algebra associated with an affine root system extended by the group diagram automorphism $\Omega=\sfP/\sfQ$.
	This algebra acts on the so-called basic representation $\bZ_{q,t}[\sfP]$ with $\hT_i,i\in\tI_{\af}$ acting as prescribed in~\eqref{eq::T::action::Hecke}, $\hX^{\lambda}$ by multiplication, while for the imaginary root $\delta$ we have $q:=\hX^{\delta}$.
	
	One of the improtant features of $\DAHA$ that apart for $\hX^{\lambda}$ it has another family of commuting operators $\{\hY^{\mu},\mu\in\sfQ^{\vee}\}$.
	Namely, for each fundamental coweight  $\omega_i \in \Pi^{\vee}$ with $i\in\tI$  we  have the translation element $t_{\omega_i}\in W_{\af}$ with a reduced decomposition:
	\begin{equation}
		\label{eq::t=ps}
		t_{\omega_i} = \pi s_{i_1} s_{i_2} \cdots s_{i_{\ell}} \hskip 5mm i_1,\ldots,i_{\ell} \in \tI_\af, \pi \in \Omega.
	\end{equation}
	One defines the fundamental $\hY$-elements by the following assignment:
	$$\hY^{\omega_i} = \pi \circ \hT_{i_1} \circ \hT_{i_2} \circ \cdots \circ \hT_{i_\ell}.
	$$
	Cherednik showed that $\hY^{\omega_i}$ is a collection of well-defined commuting elements of $\DAHA$ and the nonsymmetric Macdonald polynomials $E_\lambda$ are the eigenfunctions for them. I.e. for any $\mu=\sum_{i\in\tI} k_i \omega_i$ with $k_i\geq 0$ we consider $\hY^{\mu}:=\prod_{i\in\tI}(\hY^{\omega_i})^{k_i}$ and we have
	$$\hY^{\mu}(E_{\lambda}) = t^{\frac{l(\mu)+\langle u_{\lambda}^{-1}(2 \rho),\mu \rangle}{2}} q^{-\langle \lambda,\mu\rangle} E_{\lambda},$$
	where $u_{\lambda}\lambda=\lambda_{-}\in\sfP_{-}$,  $2\rho=\sum_{\beta \in \Delta_+}\beta$ and $l(\mu)$ is the length of the coweight $\mu$ considered as an element of the affine Weyl group $W_{\af}$.
	Note, that all eigenvalues are monomials in $q,t$ and are well suited for a categorification given by appropriate shifts of the gradings.

	\subsection{Symmetrization operator}
	\label{sec::Sym::AHA}
	With each subset $\tJ\subset\tI$ Cherednik assigns the symmetrization operator $\hP_{\tJ}\in\DAHA$ from \cite{Ch}, page 324. In our normalization/notations it is equal to:
	\[{\hP_{\tJ}}:=\sum_{w\in W_{\tJ}} {\hT}_{w}.\]
	It is shown in \cite{Ch} that for all $i\in\tJ$ the element $\hP_{\tJ}$ satisfy the following list of properties:
	\begin{gather}
		\tag{P1}
		\label{itm::P:prop2}
		\forall i\in\tJ \text{ one has }\ \hT_i{\hP_{\tJ}}={\hP_{\tJ}}\hT_i=t{\hP}; \\
		\tag{P2} 
		\label{itm::P:prop3}
		\text{ the image of the operator ${\hP_{\tJ}}$ is contained in $W_{\tJ}$-symmetric functions; }
		\\
		\tag{P3}
		\label{itm::P:prop5}
		\hP_{\tJ} f = f \hP_{\tJ} \text{ whenever } f \text{ is a $W_{\tJ}$-symmetric function.}
	\end{gather}
	Note that any operator that satisfy properties \ref{itm::P:prop2}--\ref{itm::P:prop5} is proportional to a projector on the ring of $W_{\tJ}$-symmetric functions. Therefore, it may differ from the 
	symmetrization operator $\hP_J$ by a multiplication with a $W_{\tJ}$-invariant function.
	To normalize it we recall more property of the symmetrization operator that kill this freedom:
	\begin{gather} 
		\tag{P4} \label{itm::P:prop4} \begin{array}{l}
			\text{If the weight $\mu$ is contained in the convex hall of the $W_{\tJ}$-orbit of the $\tJ$-dominant }
			\\
			\text{ weight $\lambda$, then the support of ${\hP_{\tJ}} \hX^{\mu}$ is contained
				in this convex hall;} 
		\end{array} 
		\\
		\tag{P5}	
		\label{itm::P:prop6}
		\hP_{\tJ}(\hX^{\lambda}) = \hX^{\lambda} + \sum_{\mu\prec\lambda} \hX^{\mu} c_{\lambda,\mu}(t), \text{ for $\lambda$ sufficiently big and dominant};
		\\
		\tag{P6}
		\label{itm::P:prop::0}
		\hP_{\tJ}(1) = \sum_{\omega\in W_{\tJ}} t^{l(\omega)}.
	\end{gather}	
	Note, that in order to normalize the symmetrization operator Property~\eqref{itm::P:prop6} is sufficient as well as Property~\eqref{itm::P:prop::0} is sufficient.

	Finally, one of the main applications of the symmetrization operator is that for a maximal proper subset $\tJ=\tI\subset\tI_{\af}$ it maps the nonsymmetric Macdonald polynomials to the symmetric one:
	$$
	\forall \lambda\in\sfP \ \ 
	\hP_{\tI}(E_{\lambda}) \text{ is proportional to } P_{\lambda_{+}}.
	$$
	Moreover, Macdonald polynomial $P_{\lambda_{+}}$ is an eigen-function for the action of any $W$-symmetric polynomial in $\hY$'s, however, the eigenvalue will be no more a monomial in $q,t$ and we do not know a reasonable categorical meaning of this statement at the moment.

	\section{Demazure endofunctors $\sD_i$ -- Derived superinduction}
	\label{sec::Demazure::all}
	\subsection{The main heroes of a categorification}
	\label{sec::Demazure::functor::def}	
	Recall that $\fp_{\tJ}=\fn_{\tJ}\oplus\fb$ denotes the parabolic subalgebra generated by the set $\{f_i\colon i\in\tJ\}$ and Borel subalgebra $\fb$.
	Denote by $\fp_{\tJ}[\xi]$ the corresponding $\xi$-graded Lie superalgebra of currents and let us denote by $\cO(\fp_{\tJ}[\xi])$ the category of $(\fh,\xi)$-graded, $\fp_{\tJ}$-integrable finitely generated $\fp_{\tJ}[\xi]$-modules.
	The $\xi$-graded embedding of Lie superalgebras:
	$$
	\fb[\xi] \subset \fp_{\tJ}[\xi]
	$$
	defines the restriction functor 
	\begin{equation}\label{Resfunctor}
		\Res_{\tJ}: \mathcal{O}(\fp_{\tJ}[\xi]) \rightarrow \mathcal{O}({\fb}[\xi])
	\end{equation}
	which is always exact, because it does not do anything with the underlying vector space.
	Functor $\Res_{\tJ}$ admits left and right adjoints which we denote by $\Ind_{\tJ}$ and $\Coind_{\tJ}$ following the classical Frobenius reciprocity principle:
	\begin{equation}
		\label{eq::res_ind_adjunction}
		\begin{tikzcd}
			\cO({\fb}[\xi])
			\arrow[rr, shift left = 3, "\Ind_{\tJ}"]
			\arrow[rr, shift right = 3, "\Coind_{\tJ}"', "\perp" near end]
			&&
			\cO(\fp_{\tJ}[\xi])
			\arrow[ll, "\Res_{\tJ}" description,"\perp"' near start]
		\end{tikzcd}
	\end{equation}
	Note, however, that integrability affects the story in the following way: the left adjoint functor $\Ind_{\tJ}$ maps $\fb[\xi]$-module $M$ to the maximal $\fg_{\tJ}$-integrable $\fp_{\tJ}$-quotient of the induced module):
	\begin{equation}\label{Indfunctor}
		\Ind_{\tJ}: M \mapsto U(\fp_{\tJ}[\xi]) \otimes_{U(\fb[\xi])} M/\sim,
	\end{equation}
	respectively the right adjoint functor $\Coind_{\tJ}(N)$ is the maximal $\fp_{\tJ}$-integrable submodule in the coinduced  module $\Hom_{\fb[\xi]}(U(\fp_{\tJ}[\xi]),N)$.
	This section is devoted to the particular case $\tJ=\{i\}$ which restrict all computations to the case of $\fg=\msl_2$.
	and we simplify the notations to $\Res_{i}$, $\Ind_{i}$ and $\Coind_i$ correspondingly. 
	The rank two case $\tJ=\{i,j\}$ is used in~\S\ref{sec::braid}.
	The case of generic $\tJ$ is described in details in~\S\ref{sec::symm::categorify}.
	
	For algebraic groups the corresponding induction functor is known under the name \emph{Induction} (see e.g. \cite{Yantzen}), one also uses the name \emph{Zuckermann functor} (see e.g. \cite{Knapp_Vogan}).
	Below we explain (Corollary~\ref{cor::H::Ind}) that the corresponding left (right) derived functors $\LInd_i$ and $\RCoind_i$ are well defined in the bounded derived category.
	We show that the functor $\Res_i$ is spherical in the sense of~\cite{Anno} (Proposition~\ref{prp::spherical}) and use the standard terminology of \emph{twist} and \emph{cotwist} functors assigned to spherical functors:
	\begin{notation}
		\label{not::demazure::twist}
		\begin{itemize}	
			\setlength{\itemsep}{0em}	
			\item
			The endofunctors $\Res_i\circ\LInd_i, \Res_i\circ\RCoind_i\in End(\bD^b(\cO(\fb[\xi])))$ are denoted by $\sD_i$ (resp. $\sD_i'$) and are  called  \emph{Demazure} (resp. \emph{dual Demazure}) functors;
			\item
			The mapping cone of the unit morphism $\Id \stackrel{\eta}{\rightarrow} \Res_i\circ\LInd_i$ assigned with adjunction~\eqref{eq::res_ind_adjunction} between restriction and induction functors is called the \emph{Demazure twist} and is denoted by $\sT_i$;
			\item The mapping cocone (=shifted cone) of the counit morphism $\Res_i\circ \RCoind_i \stackrel{\varepsilon}{\rightarrow} \Id$ is called the \emph{Demazure cotwist} and is denoted by $\sT_i'$.		
		\end{itemize}
	\end{notation}
	In other words, we have a pair of adjoint exact triangles of endofunctors of the triangulated category $\bD^b(\cO(\fb[\xi]))$:
	\[
	\Id_{\fb[\xi]} \stackrel{\eta}{\to} \sD_i \to \sT_i \stackrel{+1}{\to}  \Id_{\fb[\xi]}[1], \qquad
	\sT_i' \to \sD'_i \stackrel{\varepsilon}{\to} \Id_{\fb}  \stackrel{+1}{\to} \sT'_i[1]
	\]
	
	In the remaining subsections of this section, we explain the following details.
	\S\ref{sec::sl_2} describes induction without supercurrents. \S\ref{sec::super::Ind} highlights the differences between even and odd inductions. The primary property of being a spherical functor is discussed in~\S\ref{sec::sherical}. We conclude with subsection~\ref{sec::Demazure::modules}, where we describe the superposition of (nonderived) superinductions applied to a sufficiently dominant weight vector. The corresponding modules generalize the Demazure modules.

	\subsection{The derived induction for $\msl_2$}
	\label{sec::sl_2}
	From now on till the end of this section we fix a simple root $\alpha_i\in\Pi$, $i\in\tI$ and assign to the $\msl_2^{i}$-triple $\langle e,h,f\rangle$, its Borel subalgebra  $\fb_{i}:=\langle e_{i},h_{i}\rangle$ and the corresponding minimal parabolic subalgebra: 
	$$\fp_{i}:=\msl_2^{i} + \fb \simeq \langle f_i \rangle \oplus \fb.$$
	The category of $\msl_2^{i}$-integrable modules is equivalent to the category of rational $\mSL_2$-modules. 
	Following the ideas announced in the preceding~\S\ref{sec::Demazure::functor::def} the restriction functor $\Res_{i}^{\bar{0}}:\cO(\msl_2^{i}\to \cO(\fb_{i})$ admits the left adjoint induction functor $\Ind_{i}^{\bar{0}}$ and the right adjoint coinduction functor $\Coind_{i}^{\bar{0}}$. 
	The upper index $\bar{0}$ underlines that this is the even (co)induction, while the odd (co)induction will be considered in the subsequent~\S\ref{sec::super::Ind}.
	Recall, that $\Ind_i^{\bar{0}}$ assigns to a $\fb_i$-module $M$ the maximal $\msl_2^{i}$-integrable quotient of the corresponding induced module $U(\msl_2^{i})\otimes_{U(\fb_i)}M$, respectively, $\Coind_i^{\bar{0}}(M)$ is the maximal $\msl_2^{i}$-integrable submodule of $\Hom_{U(\fb_i)}(U(\msl_2^{i}),M)$.
	The functor $\Ind_i^{\bar{0}}$ has a presentation via the induction for groups (see~\cite{Yantzen}).
	Note that we have the full description of the abelian category $\cO(\fb_i)$. Namely, each indecomposable $h_i$-semisimple $e_i$-finite $\fb_i$-module admits a cyclic vector $v_n$ and is uniquely defined by the lowest and the highest weights:
	\begin{equation}
		\label{eq::V_ab}
		\begin{array}{c}
			V_{n,m}:=\mathbb{C}[e_i]/(e_i^{\frac{m-n}{2}+1}) v_{n} = \langle v_{n}, e_i v_{n}, e_i^2 v_{n}, \ldots e_i^k v_{n} = v_{m} \rangle, \\
			\ h_iv_{n} = n v_{n}, \ n \in \bZ \ \text{with} \ k=\frac{m-n}{2}\in \mathbb{N}.
		\end{array}
	\end{equation}
	Each indecomposable $\msl_2^{i}$-integrable module is irreducible and is indexed by its highest weight:
	\[ L(n):= \langle v_{n}, f_i v_{n}, \ldots, f_i^{n}v_{n} =v_{-n}\rangle
	= \langle v_{-n},e_i v_{-n}, \ldots, e_i^n v_{-n}\rangle, \ n \in \bZ_{\ge 0}. \]
	\begin{lem}
		\label{lm::LInd::sl2}
		The left derived functor $\mathbf{L}^{\udot}\Ind_{i}^{\bar{0}}$ of the induction functor has the following description on the indecomposable $\fb$-modules:
		\begin{gather*}
			{			\mathbf{L}^0\Ind_{i}^{\bar{0}}(V_{n,m}) = \Ind_{i}^{\bar{0}}(V_{n,m}) =
				\begin{cases} \oplus_{k=0}^{\frac{m-|n|}{2}} L(m- 2k), \text{ if } m \geq |n|  \geq 0; \\
					0, \text{ otherwise. }
			\end{cases} } 
			\\
			{			\mathbf{L}^{-1}\Ind_{i}^{\bar{0}}(V_{n,m}) =
				\begin{cases} 
					\oplus_{k=1}^{-\frac{n+m}{2} - 1} L(-n -2k) , \text{ if } m\geq -2\ \& \ n+m \leq 0; \\
					\oplus_{k=1}^{\frac{m-n}{2} +1}  L(-n - 2k), \text{ if } m < -2, \\
					0, \text{ otherwise. }
			\end{cases} } \\
			\mathbf{L}^{s}\Ind_{i}^{\bar{0}}(V_{n,m})=0 \text{ for }s\neq 0,-1.
		\end{gather*}
	\end{lem}
	\begin{proof}
		The proof is a straightforward computation.
		First, we compute the induction functor for one-dimensional modules whose description follows from the known description of $\msl_2$ Verma modules:
		$$
		\Ind_{i}^{\bar{0}}(V_{n,n}) = \begin{cases}
			L(n), \text{ if } n\in \bZ_{\geq 0}; \\
			0, \text{ otherwise. }
		\end{cases}
		$$  
		Second, we compute recursively the (derived) induction for arbitrary indecombosable module $V_{n,m}$ using the short exact sequences for different integer numbers $n,m$ and $k$ of the same parity:
		\begin{equation}\label{eq::sl_2::indecomp} 0 \to V_{m+2,k} \to V_{n,k} \to V_{n,m} \to 0 \text{ with } k > m \geq n. 
		\end{equation}
	\end{proof}	
	Note, that the category $\cO(\fb_{i})$ does not contain enough projectives, however,  this category contains enough $\Ind_{i}^{\bar{0}}$-acyclic objects since $\mathbf{L}^{-1}\Ind_{i}^{\bar{0}} (V_{n,m})=0$ whenever $n+m \geq 0$. Moreover, for $k>-n$  the short exact sequence~\eqref{eq::sl_2::indecomp} is an $\Ind_{i}^{\bar{0}}$-acyclic resolution of an indecomoposable module $V_{n,m}$.

	The following technical Lemma is very useful for further computations with induction functors
	\begin{lem}
		\label{lem::tensor::ind}	
		If $M$ is a finite-dimensional $\msl_2$-module then there is a natural isomorphism of additive functors from $\cO(\fb)$ to $\cO(\msl_2)$:
		\begin{equation}
			\label{eq::tensor::ind}
			\Ind(\Res(M)\otimes -) \simeq M\otimes\Ind(-),	
		\end{equation}
		what induces the isomorphism of derived functors $\LInd(M\otimes -) \stackrel{\simeq}{\longrightarrow} M\otimes\LInd(-).$	
	\end{lem}	
	\begin{proof}
		Note that for any pair of finite-dimensional $\msl_2$ modules $M$ and $K$ and an $\fb$-module $N$ we have an natural isomorphism of vector spaces based on the adjunction between induction and restriction:
		\begin{multline*}
			\Hom_{\cO(\msl_2)}(\Ind(\Res(M)\otimes N), K) \simeq 
			\Hom_{\cO(\fb)}(\Res(M)\otimes N, \Res(K)) \simeq 
			\\ \simeq
			\Hom_{\cO(\fb)}(N, (\Res(M))^{*}\otimes \Res(K)) \simeq 
			\Hom_{\cO(\fb)}(N, (\Res((M)^{*}\otimes  K)) \simeq
			\\
			\simeq
			\Hom_{\cO(\msl_2)}(\Ind(N), M^{*}\otimes K) \simeq 
			\Hom_{\cO(\msl_2)}(M\otimes\Ind(N), K).
		\end{multline*}
		Here $M^*$ denotes the linear dual module $\Hom(M,\Bbbk)$.
		In other words both functors in~\eqref{eq::tensor::ind} are the left adjoint to the isomorphic exact functors 
		$\Res(M^{*})\otimes\Res(-) \simeq \Res(M^{*}\otimes-)$ and, consequently, these functors are isomorphic, because the left adjoint if exists is unique.
		
		The same arguments can be used to show the isomorphism of the corresponding left derived functors.
	\end{proof}	
	
	Let us also outline the description of the coinduction functor:
	\begin{lem}
		\label{lem::RCoind::sl_2}
		The right derived functor $\RCoind_{i}^{\bar{0}}: \cO(\fb_{i})\to\bD^b( \cO(\msl_2^{i}) )$ has the following description on indecomposable $\fb_i$-modules:
		\begin{gather*}
			\mathbf{R}^{0}\Coind_{i}^{\bar{0}}(V_{n,m}) = \Coind_{i}^{\bar{0}}(V_{n,m}) =
			\begin{cases}
				\oplus_{k=0}^{-\frac{n+|m|}{2}} L(-|m|+2k), \text{ if }m\leq 0 \text{ or }m\geq 0\geq m+n;\\
				0, \text{ if } m>0\ \& \ m+n>0
			\end{cases}
			\\
			\mathbf{R}^{1}\Coind_{i}^{\bar{0}}(V_{n,m}) = 
			\begin{cases}
				\oplus_{k=0}^{\frac{m-n}{2}} L(m-2k-2), \text{ if } n\geq 2; \\
				\oplus_{k=1}^{\frac{m+n}{2}-1}L(m-2k), \text{ if } m+n\geq 4 \ \& \ n<2; \\
				0, \text{ otherwise. }
			\end{cases}
		\end{gather*}    
	\end{lem}
	\begin{proof}
		A straightforward computation similar to the one we suggested for Lemma~\ref{lm::LInd::sl2}.
	\end{proof}
	
	\begin{cor}
		We have an equivalence of functors:
		\begin{equation}
			\label{eq::coind=ind::0}
			\RCoind_{i}^{\bar{0}} \simeq \LInd_{i}^{\bar{0}}\circ \sX^{-\alpha_i}[-1]. \end{equation}    
	\end{cor}
	\begin{proof}
		The computations we made in Lemma~\ref{lm::LInd::sl2} and Lemma~\ref{lem::RCoind::sl_2} shows that for any $\fb_i$-indecomposable module $V_{m,n}$ we have an isomorphism:
		\begin{equation}
			\label{eq::ind::coind::indec}
			\mathbf{R}^{s}\Coind_{i}^{\bar{0}}(V_{n,m}) \simeq \mathbf{L}^{s-1}\Ind_{i}^{\bar{0}}(V_{n-2,m-2}).
		\end{equation}
		It is a straightforward computation that shows that this Isomorphism~\eqref{eq::ind::coind::indec} is functorial and, moreover, is compatible with the tensor product with finite-dimensional $\msl_2^i$-modules.    
	\end{proof}

	It is worthmentioning that one can repeat all the story for the embedding of Lie algebras $\fb\hookrightarrow\fp_i$ instead of the one we used $\fb_i\hookrightarrow \msl_2^{i}$. Finally, we end up with the 
	restriction functor $\Res_{i}^{\bar{0}}:\cO(\fp_i)\to\cO(\fb)$ and its left and right (derived) adjoint functors which we denote by the same letters $\LInd_{i}^{\bar{0}}$ and $\RCoind_{i}^{\bar{0}}$ respectively.

	\begin{lem}
		For all $\mu\in \sfP$ such that $\langle\mu,\alpha_i\rangle =1$ 
		there exists a natural transformation $\varphi$ of the abelian and (consequently) the corresponding derived functors:   
		\begin{equation}
			\label{eq::shift::Ind:0}
			\varphi:\Res_{i}^{\bar{0}}\circ\Ind_{i}^{\bar{0}}\circ \sX^{\mu} \rightarrow \sX^{\mu-\alpha_i} \circ \Res_{i}^{\bar{0}}\circ\Ind_{i}^{\bar{0}}. 
		\end{equation}
	\end{lem}
	\begin{proof}
		The proof of this Lemma also reduces to the computations for $\fb_i$ and $\msl_2^{i}$-modules.
		Instead of working with the shift endofunctor $\sX^{\mu}$ of the bigger category $\cO(\fb)$ we will deal with the weight shift endofunctor of $\sX\in\End(\cO(\fb_{i}))$ that maps the indecomposable $\fb_{i}$-module $V_{n,m}$ to the indecomposable module with the shifted weight $V_{n+1,m+1}$.
		Then for all integer $m\geq 0$ we have a surjective map of $\fb_i$-modules with one-dimensional kernel $V_{m+1,m+1}$:
		$$
		\varphi:\sX (L(m+1))=V_{-m,m+2} \twoheadrightarrow V_{-m,m} = L(m) \ \Leftrightarrow \  \varphi: L(m+1) \twoheadrightarrow \sX^{-1} (L(m))$$
		Consequently, for all pairs of integer numbers $m, n$ of the same parity such that $m\geq |n|$ we have a surjective map of $\fb_i$-modules:
		\begin{multline*}
			\Res_{i}^{\bar{0}}(\Ind_{i}^{\bar{0}}(\sX(V_{n,m})) = 
			\Res_{i}^{\bar{0}}(\Ind_{i}^{\bar{0}}(V_{n+1,m+1})) = 
			\oplus_{k=0}^{\frac{m+1-|n+1|}{2}} L(m+1-2k) 
			\stackrel{\varphi}{\twoheadrightarrow}
			\\
			\twoheadrightarrow 
			\oplus_{k=0}^{\frac{m+1-|n+1|}{2}} \sX^{-1} (L(m-2k))
			\twoheadrightarrow \sX^{-1} \left(\oplus_{k=0}^{\frac{m-|n|}{2}} L(m-2k)\right) \simeq
			\sX^{-1}\circ\Res_{i}^{\bar{0}}\circ\Ind_{i}^{\bar{0}} V_{n,m}.  
		\end{multline*}
		There is nothing to prove for indecomposable $\fb_i$-modules $V_{n,m}$ such that $\Ind_{i}(V_{n,m})=0$. Thus, we end up with a map $\varphi$ defined for indecomposable functors and it remains to check that it is compatible with morphisms in order to claim that this is a functor. This is a straightforward computation whose details we leave for the reader.
	\end{proof}
	
	\begin{cor}
		\label{cor::Ind:0::X}
		For all $\mu\in\sfP$ with $\langle\mu,\alpha_i^{\vee}\rangle=1$ there is a natural transformation of triangulated endofunctors:
		$$\varphi:\Res_{i}^{\bar{0}}\circ\LInd_{i}^{\bar{0}}\circ \sX^{\mu} \rightarrow \sX^{\mu-\alpha_i} \circ \Res_{i}^{\bar{0}}\circ\LInd_{i}^{\bar{0}}.  
		$$    
	\end{cor}

	\subsection{The derived superinduction $\LInd_i$}
	\label{sec::Ind} \label{sec::super::Ind}
	Let us fix a simple positive root $\alpha_i\in \Pi$ and assign to it a minimal parabolic subalgebra $\fp_i$. Denote by $\fp_i^{\bar{1}}$ the intermediate Lie subalgebra spanned by $\fb$ and $\fp_i\xi$:
	\[\fb[\xi]\subset \fp_i^{\bar{1}} \subset \fp_i[\xi].\]
	The quotient $\fp_i^{\bar{1}}/\fb[\xi]$ is a one-dimensional odd space spanned by $f_i\xi$. Respectively $\fp_i[\xi]/\fp_i^{\bar{1}}$ is even and is spanned by $f_i$.
	The (co)induction functors admit a decomposition into consecutive (co)inductions:
	\begin{equation}
		\label{Indfunctor1}
		\begin{tikzcd}
			\cO(\fb[\xi])
			\arrow[rr,shift left = 1,"\Ind_i^{\bar{1}}"]
			\arrow[rrrr, bend left, shift left = 1, "\Ind_i"]
			&& \cO(\fp_i^{\bar{1}})
			\arrow[ll,shift left = 1,"\Res_i^{\bar{1}}"]
			\arrow[rr,shift left = 1,"\Ind_i^{\bar{0}}"] && \cO(\fp_i[\xi])
			\arrow[ll,shift left = 1,"\Res_i^{\bar{0}}"]
			\arrow[llll, bend right, shift left = 1, "\Res_i"]
		\end{tikzcd}
	\end{equation}
	The key advantage of this decomposition is that the odd induction $\Ind_i^{\bar{1}}$ is an exact functor and, therefore, we have an isomorphism of derived and nonderived one:
	$$
	\LInd_i^{\bar{1}}(M) = \Ind_i^{\bar{1}}(M).
	$$
	This happens because the universal enveloping algebra of an odd Lie one-dimensional superalgebra $\langle f_i\xi\rangle$ is spanned by two elements $1,f_i\xi$. Consequently, $\dim (\Ind_i^{\bar{1}}M) = 2 \dim M$ and we have a short exact sequence of $\fb[\xi]$-modules:
	\begin{equation}
		\label{eq::Ind_odd::h-grad}
		0\to M \to \Ind_i^{\bar{1}}M = U(\fp_i^{\bar{1}})\otimes_{U(\fb[\xi])} M \to (f_i\xi) M \to 0.
	\end{equation}
	Here \eqref{eq::Ind_odd::h-grad} is interpreted as a sequence of morphisms of functors that induces a distinguished triangle:
	\begin{equation}
		\label{eq::Ind::od::proj}
		\Id_{\fb[\xi]} \stackrel{\eta^{\bar{1}}}\longrightarrow \Res_i^{\bar{1}}\circ\Ind_i^{\bar{1}} \longrightarrow \sX^{-\alpha_i}\sv.
	\end{equation}
	The same arguments shows that the odd restriction functor $\Res_i^{\bar{1}}$ admits the right adjoint coinduction functor which is also exact and fits in a similar distinguished triangle:
	\begin{equation}
		\label{eq::CoInd::od::proj}
		\RCoind_{i}^{\bar{1}}(M)=\Coind_i^{\bar{1}}(M) = \Hom_{U(\fb[\xi])}(U(\fp_i^{\bar{1}}),M) \ \Rightarrow \ 
		\sX^{\alpha_i} \sv^{-1} \rightarrow \Res_i^{\bar{1}}\circ\Coind_i^{\bar{1}} \stackrel{\varepsilon^{\bar{1}}}{\rightarrow} \Id_{\fb[\xi]}.
	\end{equation}
	The odd induction and coinduction functors commutes with the shift grading functors $\sX^{\mu}$ and thanks to distinguished triangles~\eqref{eq::Ind::od::proj} and~\eqref{eq::CoInd::od::proj} we end up with isomorphisms: 
	\begin{equation}
		\label{eq::Ind::Coind::odd}
		\Ind_i^{\bar{1}} \simeq 
		\Coind_i^{\bar{1}}\circ \sX^{-\alpha_i}\sv \simeq \sX^{-\alpha_i}\sv \circ \Coind_i^{\bar{1}}. 
	\end{equation}
	
	\begin{rem}
		The short exact sequence~\eqref{eq::Ind_odd::h-grad} may not split in general. For example, consider $\fg=\msl_2$ and two-dimensional cyclic module  $M=U(\fb[\xi])v/(e v , h v, (e\xi)v)$ whose basis is  $\langle v, (h\xi)v\rangle$. The module $\Ind_i^{\bar{1}}M$ is the cyclic module generated by the same cyclic vector $v$ subject to the same set of relations $ev=hv=(e\xi)v = 0$ and has the basis $\langle v,(f\xi)v, (h\xi)v, (f\xi)(h\xi)v\rangle$. We have  $e(f\xi)v=(h\xi)v\neq 0$ and on the other hand, the action of $e$ on $M$ is trivial which implies that $\Ind_i^{\bar{1}}M$ is not isomorphic to the direct sum of $M$ and $M$ shifted.
	\end{rem}	
	Let us outline in one theorem the properties of the derived superinduction and coinduction functors:
	\begin{thm}
		\label{cor::H::Ind}			\label{cor::sl2::tensor}
		\begin{enumerate}
			\item	
			\label{thm::item::LInd::1}
			The right exact induction functors $\Ind_i:\cO(\fb[\xi])\to \cO(\fp_i[\xi])$, as well as the left exact coinduction functor $\Coind_i$ admit derived functors 
			$$\LInd_i,\RCoind_i:\bD^{b}(\mathcal{O}(\fb[\xi])) \to \bD^{b}(\mathcal{O}(\fp_i[\xi]));$$
			\item 
			\label{thm::item::LInd::2}
			Each graded $\fb[\xi]$-module $N$ admits $\Ind_i$-acyclic and $\Coind_i$-acyclic resolutions of length at most $1$, what follows that 		
			$$
			\mathbf{L}^{<-1}\Ind_i(N) = \mathbf{R}^{>1}\Coind_i(N) =0;
			$$	 
			\item 
			\label{thm::item::LInd::4}
			The (derived) (co)induction $\Ind_i$ (resp. $\LInd_{i}$) commute with the tensor product with a finite-dimensional $\fp_i[\xi]$-module $M$:
			\begin{equation}
				\label{eq::tensor::compat}
				\begin{array}{c}			
					(\mathbf{L})\Ind_i( \Res_i(M)\otimes - ) \simeq M \otimes (\mathbf{L})\Ind_i(-). 
					\\
					(\mathbf{R})\Coind_i( \Res_i(M)\otimes - ) \simeq M \otimes (\mathbf{R})\Coind_i(-); 
				\end{array}			
			\end{equation}
			\item 
			\label{thm::item::LInd::5}
			The derived induction and coinduction are isomorphic modulo certain shift functors:
			\begin{equation}
				\label{eq::ind::coind::iso}
				\LInd_{i} \simeq  \RCoind_{i}\circ \sv[1]. 
			\end{equation}
		\end{enumerate}
	\end{thm}	
	\begin{proof}		
		The proof of the theorem is based on the computation done for the (derived) induction $\fg=\msl_2$ in the preceding Subsection~\S\ref{sec::sl_2}. We just have to take into account the odd induction $\Ind_i^{\bar{1}}$ (which does not affect the derived story because it is an exact functor) and the extra action of the Borel subalgebra $\fb$ on all modules.

		Note that functors $\Ind_i$ (respectively $\Coind_i$) are left (resp. right) adjoint to the restriction functor $\Res_i$. Therefore, the corresponding derived functors are well defined whenever the underlying abelian category contains enough projective (respectively enough injective) objects. Unfortunately, this is not the case and this category does not contain either projective or injective objects. However, let us explain that this category contains enough $\Ind_i$-acyclic and $\Coind_i$-acyclic objects and describe an acyclic resolution for any $M\in\cO(\fb[\xi])$.
		
		Denote by $L^{\alpha_i}({r})$ the following $\fp_i^{\bar{1}}$-module generated by a cyclic vector $v$, subject to the following list of relations
		\[
		\fh v = (\fg\xi) v = e_j v = e_i^{r+1}v = 0, \text{ with } j\neq i.
		\]
		The module $L^{\alpha_i}({r})$ is of dimension $r+1$ and has a basis $\{(e_i)^{m}v | 0\leq m \leq r\}$.
		The submodule generated by $e_iv$ is isomorphic to the $\alpha$-shifted module
		$\sX^{\alpha_i}( L^{\alpha_i}({r-1})) = L^{\alpha_i}(r-1)\otimes_{\Bbbk} \Bbbk_{\alpha_i}$.
		Thus we have a short exact sequence of  $\fp_i^{\bar{1}}$-modules:
		\begin{equation}
			\label{eq::sl2::resol}
			0 \to \sX^{\alpha_i}( L^{\alpha_i}({r-1})) \to L^{\alpha_i}(r) \to \Bbbk \to 0.
		\end{equation}
		We claim that the tensor product over the base field $\Bbbk$ of any given {finite-dimensional} $\fp_i^{\bar{1}}$-module (resp. $\fb[\xi]$-module) $M$ with the short exact sequence~\eqref{eq::sl2::resol} defines a two-term $\Ind_i^{\bar{0}}$-acyclic resolution of $M$ for $r \gg 0$:
		\[
		0\to M\otimes_{\Bbbk} \sX^{\alpha_i}( L^{\alpha_i}({r})) \to M\otimes_{\Bbbk} L^{\alpha_i}({r+1}) \to M \to 0.
		\]
		This follows from the $\msl_2$-computation, because the tensor product of $\fb_i$-module $V_{n,m}$ and $L^{\alpha_i}(r)$ is isomorphic to
		$$\bigoplus_{k=0}^{m-n} V_{n+k,r+m-k} \hskip 5mm \text{whenever} \hskip 5mm r>\max \{ 2 (m-n), m+n \}$$
		as $\mathfrak{sl}_2^i$-module, that is 
		$\LInd_i$-acyclic by Lemma~\ref{lm::LInd::sl2}.
		
		The same arguments show that each module admits a two-term $\Coind_i$-acyclic resolution what follows items \eqref{thm::item::LInd::1}, \eqref{thm::item::LInd::2}.
		
		The proof of Item~\eqref{thm::item::LInd::4} repeats the proof of Proposition~\ref{lem::tensor::ind}. 
		Indeed, for a finite-dimensional module $M$, its linear dual module $M^*$ is well defined and the functors in~\eqref{eq::tensor::compat} are left (resp. right) adjoints to the isomorphic exact functors 
		$$
		\Res_{i}(M^*)\otimes \Res_{i}(-) \simeq \Res_{i}(M^*\otimes -).
		$$
		
		The last item follows from the similar isomorphisms we proved for even and odd inductions:
		\begin{multline*}
			\LInd_{i}=\LInd_{i}^{\bar{0}}\circ\LInd_i^{\bar{1}} \stackrel{\eqref{eq::Ind::Coind::odd}}{\simeq} \LInd_i^{\bar{0}}\circ(\sX^{-\alpha_i}\sv \circ\RCoind_i^{\bar{1}})
			\simeq \\
			\stackrel{\eqref{eq::coind=ind::0}}{\simeq}(\RCoind_{i}^{\bar{0}}\circ\sX^{\alpha_i}[1])\circ(\sX^{-\alpha_i}\sv\circ\RCoind_i^{\bar{1}}) \simeq 
			\RCoind_i^{\bar{0}}\circ\RCoind_i^{\bar{1}}\circ\sv[1] \simeq \RCoind_i\circ\sv[1].
		\end{multline*}
		These complete the proofs.
	\end{proof}

	We finish this subsection with the detailed description of the Demazure endofunctors $\sD_i=\Res_i\circ\LInd_i$, $\sD_i':=\Res_i\circ\RCoind_i$ and the autoequivalences 
	$$\sT_i:=\cone(\Id\to\Res_i \circ\LInd_i), \quad\sT_i':=\cone(\Res_i\circ\RCoind_i\to\Id)[-1]$$ in particular cases:
	\begin{example}	
		\label{ex::sl_2::T_i}	
		Let $\Bbbk_{\lambda}$ be an irreducible one-dimensional $\fb$-module with $\fh$-action by the given weight  $\lambda\in\sfP$.
		
		Suppose that $\langle\lambda,\alpha_i^{\vee}\rangle = n\geq 2$  then $\sD_i(\Bbbk_{\lambda})$ is a module concentrated in $0$'th homological degree:
		\begin{equation}
			\sD_i(\Bbbk_{\lambda}) = \quad
			\begin{tikzpicture}[scale=0.5]
				\node[ext] (v0) at (-8,0) {\tiny{$f_i^{n} v_{\lambda}$}};
				\node (a0) at (-6,0) {\small{$e_i$}};
				\node[int] (v1) at (-4,0) {};
				\node (a1) at (-2,0) {\tiny{\ldots}};
				\node (a2) at (2,0) {\tiny{\ldots}};
				\node[int] (v2) at (4,0) {};
				\node (a3) at (6,0) {\small{$e_i$}};
				\node[ext] (v3) at (8,0) {\tiny{$v_{\lambda}$}};
				\draw (v0) edge (a0);
				\draw (a0) edge[->] (v1) edge (a1);
				\draw (a2) edge[->] (v2) edge (a3);
				\draw (a3) edge[->] (v3);
				\node[int] (u0) at (-4,2) {};
				\node (b0) at (-2,2) {\tiny{$\ldots$}};
				\node (b1) at (2,2) {\tiny{$\ldots$}};
				\node[ext] (u1) at (4,2) {\tiny{$f_i\xi v_{\lambda}$}};
				\node (p0) at (-6,1) {\tiny{$e_i\xi$}};
				\node (p1) at (0,2) {\tiny{$\ldots$}};
				\node (p2) at (0,0) {\tiny{$\ldots$}};
				\draw (v0) edge (p0);
				\draw (p0) edge[->] (u0);
				\draw (u0) edge (b0);
				\draw (b1) edge[->] (u1);
				\draw (v1) edge[->] (p1);
				\draw (p2) edge[->] (u1);
			\end{tikzpicture}
			.		\end{equation}
		We denote by $v_\lambda$ the generator of the one-dimensional module $\Bbbk_{\lambda}$.
		Vertices correspond to the basis vectors of the module $\sD_i(\Bbbk_\lambda)$, and edges are responsible for the action of elements $e_i$ and $e_i\xi$.
		The horizontal grading is the $h_i$-eigenvalues and the vertical one corresponds to the $\xi$-grading.  The complex $\sT_i(\Bbbk_{\lambda})=\cone(\Bbbk_{\lambda} \to \sD_i(\Bbbk_{\lambda}))$  consists of the quotient  module placed in $0$'th homological degree. (We denote by $\times$ the place where the vector $v_\lambda$ is supposed to stay in order to visualize the $\xi$ and $\fh$-gradings):
		\begin{equation}
			\label{eq::T::k}	
			\sT_i(\Bbbk_{\lambda}) =
			\quad 
			\begin{tikzpicture}[scale=0.5]
				\node[ext] (v0) at (-8,0) {\tiny{$f_i^{n} v_{\lambda}$}};
				\node (a0) at (-6,0) {\small{$e_i$}};
				\node[int] (v1) at (-4,0) {};
				\node (a1) at (-2,0) {\tiny{\ldots}};
				\node (a2) at (2,0) {\tiny{\ldots}};
				\node[int] (v2) at (4,0) {};
				\node (v3) at (8,0) {$\times$};
				\draw (v0) edge (a0);
				\draw (a0) edge[->] (v1) edge (a1);
				\draw (a2) edge[->] (v2);
				\node[int] (u0) at (-4,2) {};
				\node (b0) at (-2,2) {\tiny{$\ldots$}};
				\node (b1) at (2,2) {\tiny{$\ldots$}};
				\node[ext] (u1) at (4,2) {\tiny{$f_i\xi v_{\lambda}$}};
				\node (p0) at (-6,1) {\tiny{$e_i\xi$}};
				\node (p1) at (0,2) {\tiny{$\ldots$}};
				\node (p2) at (0,0) {\tiny{$\ldots$}};
				\draw (v0) edge (p0);
				\draw (p0) edge[->] (u0);
				\draw (u0) edge (b0);
				\draw (b1) edge[->] (u1);
				\draw (v1) edge[->] (p1);
				\draw (p2) edge[->] (u1);
			\end{tikzpicture}
			.		\end{equation}
		On the other hand the coinduction functors $\sD_i'(\Bbbk_{\lambda})$
		is concentrated in $1$'st homological degree, whenever  $\langle\lambda,\alpha_i^{\vee}\rangle = n\geq 2$: \[
		\sD_i'(\Bbbk_{\lambda})[1] =
		\quad 
		\begin{tikzpicture}[scale=0.5]
			\node[int] (v0) at (-8,-2) {};
			\node (a0) at (-6,-2) {\small{$e_i$}};
			\node[int] (v1) at (-4,-2) {};
			\node (a1) at (-2,-2) {\tiny{\ldots}};
			\node (a2) at (2,-2) {\tiny{\ldots}};
			\node[int] (v2) at (4,-2) {};
			\node (a3) at (6,-2) {\small{$e_i$}};
			\node[int] (v3) at (8,-2) {};
			\node (v4) at (8,0) {$\times$};
			\draw (v0) edge (a0);
			\draw (a0) edge[->] (v1) edge (a1);
			\draw (a2) edge[->] (v2);
			\draw (v2) edge (a3);
			\draw (a3) edge[->] (v3);
			\node[ext] (u0) at (-4,0) {\tiny{$f_i^{n-1}v_{\lambda}$}};
			\node (b0) at (-2,0) {\tiny{$\ldots$}};
			\node (b1) at (2,0) {\tiny{$\ldots$}};
			\node[int] (u1) at (4,0) {};
			\node (u2) at (8,0) {$\phantom{v_i}$};
			\node (p0) at (-6,-1) {\tiny{$e_i\xi$}};
			\node (p1) at (0,0) {\tiny{$\ldots$}};
			\node (p2) at (0,-2) {\tiny{$\ldots$}};
			\draw (v0) edge (p0);
			\draw (p0) edge[->] (u0);
			\draw (u0) edge (b0);
			\draw (b1) edge[->] (u1);
			\draw (v1) edge[->] (p1);
			\draw (p2) edge[->] (u1);
		\end{tikzpicture}
		.\]
		The complex $\sT_i'$ is also concentrated in one homological degree and consists of the extension of the aforementioned module $\sD_i'(\Bbbk_{\lambda})[1]$ with a vector $v_{\lambda}$:
		\begin{equation}
			\label{eq::pict::T'}
			\sT_i'(\Bbbk_{\lambda})[1] =
			\quad 
			\begin{tikzpicture}[scale=0.5]
				\node[int] (v0) at (-8,-2) {};
				\node (a0) at (-6,-2) {\small{$e_i$}};
				\node[int] (v1) at (-4,-2) {};
				\node (a1) at (-2,-2) {\tiny{\ldots}};
				\node (a2) at (2,-2) {\tiny{\ldots}};
				\node[int] (v2) at (4,-2) {};
				\node (a3) at (6,-2) {\small{$e_i$}};
				\node[int] (v3) at (8,-2) {};
				\draw (v0) edge (a0);
				\draw (a0) edge[->] (v1) edge (a1);
				\draw (a2) edge[->] (v2);
				\draw (v2) edge (a3);
				\draw (a3) edge[->] (v3);
				\node[ext] (u0) at (-4,0) {\tiny{$f_i^{n-1}v_{\lambda}$}};
				\node (b0) at (-2,0) {\tiny{$\ldots$}};
				\node (b1) at (2,0) {\tiny{$\ldots$}};
				\node[int] (u1) at (4,0) {};
				\node[ext] (u2) at (8,0) {\tiny{${v_\lambda}$}};
				\node (p0) at (-6,-1) {\tiny{$e_i\xi$}};
				\node (p1) at (0,0) {\tiny{$\ldots$}};
				\node (p2) at (0,-2) {\tiny{$\ldots$}};
				\node (p3) at (6,-1) {\tiny{$e_i\xi$}};
				\draw (v0) edge (p0);
				\draw (p0) edge[->] (u0);
				\draw (u0) edge (b0);
				\draw (b1) edge[->] (u1);
				\draw (v1) edge[->] (p1);
				\draw (p2) edge[->] (u1);
				\draw (v2) edge (p3);
				\draw (p3) edge[->] (u2);
				\draw (u1) edge[->] (u2);
			\end{tikzpicture}
			.
		\end{equation}
		Note that all modules we described above admit a cyclic vector and we see directly an isomorphism \begin{equation}
			\sv\sT_i'(\Bbbk_{\lambda})[1] \simeq \sT_i(\Bbbk_{\lambda+\alpha_i}) \text{ for } \langle \lambda,\alpha_i^{\vee}\rangle \geq 2.
		\end{equation}
		Let us also provide a pictorial description for $\langle\lambda,\alpha_i^{\vee}\rangle = -n \leq -2$:
		\begin{equation*}
			\sD_i'(\Bbbk_{\lambda})   = 
			\begin{tikzpicture}[scale=0.5]
				\node[ext] (v) at (-12,0) {\tiny{$ v_{\lambda}$}};
				\node[int] (v0) at (-8,0) {};
				\node (a) at (-10,0) {\small{$e_i$}};
				\node (a0) at (-6,0) {\small{$e_i$}};
				\node[int] (v1) at (-4,0) {};
				\node (a1) at (-2,0) {\tiny{\ldots}};
				\node (a2) at (2,0) {\tiny{\ldots}};
				\node[int] (v2) at (4,0) {};
				\node (a3) at (6,0) {\small{$e_i$}};
				\node[ext] (v3) at (8,0) {\tiny{$e_i^n v_{\lambda}$}};
				\draw (v) edge (a);
				\draw (a) edge[->] (v0);
				\draw (v0) edge (a0);
				\draw (a0) edge[->] (v1) edge (a1);
				\draw (a2) edge[->] (v2) edge (a3);
				\draw (a3) edge[->] (v3);
				\node[int] (u0) at (-4,-2) {};
				\node (b0) at (-2,-2) {\tiny{$\ldots$}};
				\node (b1) at (2,-2) {\tiny{$\ldots$}};
				\node[int] (u1) at (4,-2) {};
				\node (p0) at (-2,-1) {\tiny{$e_i\xi$}};
				\node (p1) at (0,0) {\tiny{$\ldots$}};
				\node (p2) at (0,-2) {\tiny{$\ldots$}};
				\node (p3) at (6,-1) {\tiny{$e_i\xi$}};
				\draw (u0) edge (p0);
				\draw (p0) edge[->] (p1);
				\draw (p2) edge[->] (v2);
				\draw (u1) edge (p3);
				\draw (p3) edge[->] (v3);
				\draw (b1) edge[->] (u1);
				\draw (u0) edge (b0);
				\node[int] (u) at (-8,-2) {};
				\node (p) at (-6,-2) {\tiny{$e_i$}};
				\draw (u) edge (p);
				\draw (p) edge[->] (u0);
				\node (p4) at (-6,-1) {\tiny{$e_i\xi$}};
				\draw (u) edge (p4);
				\draw (p4) edge[->] (v1);
			\end{tikzpicture}
			.
		\end{equation*}
	\end{example}

	\begin{cor}
		\label{cor::T_i::char}		
		On the level of characters, $\sT_i$ is acting by the Hecke generator $\hT_i$:
		\begin{equation}
			\label{eq::decategorify::T}
			\forall M^{\udot}\in\bD^{b}(\mathcal{O}(\fb[\xi])) \text{  we have  } \chi(\sT_i(M^{\udot})) = \hT_i(\chi(M^{\udot})) \in \bZ_{t}[\sfP],
		\end{equation}
		where $\hT_i$ acts in polynomial representation by operators~\eqref{eq::T::action::Hecke}.
	\end{cor}
	\begin{proof}
		Since the character computes the Euler characteristic of complexes and the latter is additive it is enough to verify Identity~\eqref{eq::decategorify::T} only for one-dimensional $M=\Bbbk_\lambda$ with zero $\xi$-grading.
		For $\langle\lambda,\alpha_i^{\vee}\rangle\geq 2$ we see from~\eqref{eq::T::k} that 
		\begin{multline*}
			\chi (\sT_i(\Bbbk_\lambda)) =\left( \hX^{s_i(\lambda)} + \hX^{s_i(\lambda)+\alpha_i}+\ldots+\hX^{\lambda-\alpha_i}\right) + 
			(-t)\left(
			\hX^{s_i(\lambda)+\alpha_i}+\ldots+\hX^{\lambda-\alpha_i}
			\right) =  \\
			= t\hX^{s_i(\lambda)} + (1-t)\left(\hX^{s_i(\lambda)} + \hX^{s_i(\lambda)+\alpha_i}+\ldots+\hX^{\lambda-\alpha_i}\right) = 
			t\hX^{s_i(\lambda)}  + (1-t) \frac{\hX^{s_i(\lambda)} -\hX^{\lambda}}{1-\hX^{\alpha_i}} \stackrel{\eqref{eq::T::X::action::basic}}{=} \hT_i \hX^{\lambda}. 
		\end{multline*}
		One can also work out the same computation for $\langle\lambda,\alpha_i^{\vee}\rangle < 2$. However, it is enough to notice that $\hT_i+1$ commutes with multiplication by $s_i$-symmetric functions and $\sD_i$ commutes with tensor product with $\msl_2$-two-dimensional module $L^{\alpha_i}({1})$. What means that whenever we verify Identity~\eqref{eq::decategorify::T} for $M=\Bbbk_{\lambda}$ with $\langle\lambda,\alpha_i^{\vee}\rangle \geq k$ we can use the induction argument:
		\begin{multline*}
			\chi(\sT_i(\Bbbk_{\lambda-\varpi_i}))+\hX^{\lambda-\varpi_i} = 
			\chi(\sD_i(\Bbbk_{\lambda-\varpi_i})) = \chi(\sD_i(L^{\alpha_i}({1})\otimes\Bbbk_{\lambda}) - \chi(\sD_i(\Bbbk_{\lambda-\varpi_i+\al_i})) = 
			\\ =
			\chi(L^{\alpha_i}({1})\otimes\sD_i(\Bbbk_{\lambda})) - (\hT_i+1)(\hX^{\lambda-\varpi_i+\al_i}) = 
			(\hX^{\al_i-\varpi_i}+\hX^{-\varpi_i})(\hT_i+1)(\hX^{\lambda})- (\hT_i+1)(\hX^{\lambda-\varpi_i+\al_i}) =
			\\ =
			(\hT_i+1)(\hX^{\al_i-\varpi_i}+\hX^{-\varpi_i})\hX^{\lambda} - (\hT_i+1)(\hX^{\lambda-\varpi_i+\al_i}) =
			(\hT_i+1)(\hX^{\lambda-\varpi_i}).
		\end{multline*}
		and prove Identity~\eqref{eq::decategorify::T} for $M=\Bbbk_{\lambda-\varpi_i}$.
	\end{proof}	
	
	\subsection{$\Res_i$, $\LInd_i$ are  spherical functors}
	\label{sec::sherical}
	
	In the preceding \S\ref{sec::super::Ind} we showed the existence of the left and right adjoints $\LInd_i$ and $\RCoind_{i}$ to the restriction functor $\Res_i$, collect different properties of these functors in Theorem~\ref{cor::H::Ind}.  
	In this section, we prove the extra properties of the functors $\Res_i$ that categorifies the Hecke relations known for $\hT_i$.
	
	Let us briefly recall the formalism of spherical functors introduced by I.\,Anno. We follow the definitions suggested in~\cite{Kuznetsov_Spherical} and refer to the detailed description to~\cite{Anno}.\footnote{ We will consider cones of functors like $\sD_{i},\sT_{i},\BGG_{{ij}}$, whose existence do not follow from the structure of triangulated categories~\cite[Introduction]{Anno}. Examining the following diagrams
		$$(P_{\tJ}/B) \stackrel{\mathsf{pr}_2}{\longleftarrow} P_{\tJ} \times ( P_{\tJ}/B ) \stackrel{/B}{\longrightarrow} P_{\tJ} \times^B ( P_{\tJ}/B ) \stackrel{\cong}{\longrightarrow} ( P_{\tJ}/B )^2 \stackrel{\mathsf{pr}_2}{\longrightarrow} (P_{\tJ}/B) \hskip 5mm \text{for each} \hskip 3mm |\tJ| \le 2,$$
		that appear in their definitions, we find that they are common except for sheaves (Fourier-Mukai type kernel) on $( P_{\tJ}/B )^2$ (that are objects in a triangulated category). Thus, we can replace cones in the functor category with cones in a triangulated category to avoid enhancement in the discussions below (as in \cite[\S 2.1]{Kuznetsov_Spherical}).}
	\begin{definition}
		A functor $\Phi:\bD_1\rightarrow\bD_2$ between triangulated categories is called spherical if it admits the left adjoint $\Phi^*$ and the right adjoint $\Phi^{!}$ such that the following natural transformations are isomorphisms of functors:
		\begin{equation*}
			\begin{array}{c}
				\begin{tikzcd}
					\Phi^{!}\oplus \Phi^{*} \arrow[rr,"\eta\circ\Phi^!+\Phi^*\circ\eta"] & & \Phi^{*}\circ\Phi\circ\Phi^{!},
				\end{tikzcd} \\
				\begin{tikzcd}
					\Phi^{!}\oplus\Phi^{*} & &
					\arrow[ll,"\Phi^!\circ \varepsilon + \varepsilon\circ \Phi"']
					\Phi^{!}\circ\Phi\circ\Phi^{*},
				\end{tikzcd}
			\end{array}
		\end{equation*}
		where $\eta$ and $\varepsilon$ denote the unit and the counit morphisms associated with adjunctions.
	\end{definition}
	The following necessary and sufficient property of spherical functors is quite important for us:
	\begin{prop}(\cite[Proposition 2.9]{Kuznetsov_Spherical})
		\label{prp::spherical::general}
		If the functor $\Phi:\bD_1\to\bD_2$ is spherical, then the endofunctor $T_{\Phi}:=\cone(\Id\rightarrow\Phi\circ\Phi^*)$ and the endofunctor $T_{\Phi}':=\mathsf{cocone}(\Phi\circ\Phi^!\rightarrow\Id)$ are mutualy inverse autoequivalences of $\bD_1$.
	\end{prop}
	The endofunctor $T_{\Phi}$ is called \emph{twist} and $T_{\Phi}'$ is called \emph{dual twist} (we refer to~\cite{Anno} for the detailed exposition of the formalism of spherical functors). In particular, one can easily show from the definition that if $\Phi$ is a spherical functor, then both $\Phi^*$ and $\Phi^{!}$ are also spherical functors.

	\begin{thm}
		\label{prp::spherical}	
		The derived functor $\Res_i$ is spherical.
		That is, the following two natural transformations of endofunctors are isomorphisms:
		\begin{eqnarray}
			\label{eq::Ind::spherical}
			{
				\begin{tikzcd}
					\RCoind_i \oplus \LInd_i \arrow[rrrrr,"\eta\circ \RCoind_i + \LInd_i\circ\eta"] &&&&&  \LInd_i \circ \Res_i \circ \RCoind_i,
				\end{tikzcd}
			}
			\\
			\label{eq::CoInd::spherical}
			{
				\begin{tikzcd}
					\RCoind_i \oplus \LInd_i &&&&& \arrow[lllll," \RCoind_i\circ\varepsilon + \varepsilon\circ \LInd_i"']  \RCoind_i \circ \Res_i \circ \LInd_i.
				\end{tikzcd}
			}
		\end{eqnarray}
	\end{thm}
	\begin{proof}	
		If $d:=\langle\lambda,\alpha^{\vee}\rangle \gg 0$, then the $\msl_2$-theory implies the following isomorphisms of graded $\fb$-modules:
		\begin{gather*}
			\LInd_i(\Bbbk_{\lambda}) = \LInd_i^{\bar{0}}(Span(v_\lambda, f_i\xi v_{\lambda}) ) = L^{\alpha_i}(d) \oplus L^{\alpha_i}(d-2)\xi,
			\\
			\RCoind_i(\Bbbk_{\lambda}) = \LInd_i(\sv^{-1}\Bbbk_{\lambda}[-1]) = (L^{\alpha_i}(d)\xi^{-1}\oplus L^{\alpha_i}(d-2))[-1],
			\\
			\begin{array}{c}
				\LInd_i \circ \Res_i \circ \RCoind_i (\Bbbk_{\lambda}) \cong
				\LInd_i\circ \Res_i \left((L^{\alpha_i}(d)\xi^{-1}\oplus L^{\alpha_i}(d-2))[-1] \right) 
				\\ \cong
				\LInd_i^{\bar{0}}\left(
				V^{\alpha_i}_{-d,d}\xi^{-1}\oplus V^{\alpha_i}_{-d+2,d-2} \oplus V^{\alpha_i}_{-d-2,d-2} \oplus V^{\alpha_i}_{-d,d-4}\xi  \right)[-1] \\
				\cong
				(L^{\alpha_i}({d})\xi^{-1}\oplus L^{\alpha_i}({d}-2))[-1] \oplus (L^{\alpha_i}({d})\oplus  L^{\alpha_i}({d}-2)\xi )[-1+1]\\
				\cong
				\RCoind_i(\Bbbk_{{d}}) \oplus \LInd_i(\Bbbk_{{d}}),
			\end{array}
		\end{gather*}
		where we denote by $\Bbbk_{\lambda}$ a one-dimensional $\fb$-module with trivial action of $\fn_+$ and $\fh$ acts by weight $\lambda$, and $L^{\alpha_i}(d)$ denotes an $\fp_i$-module whose restriction to $\msl_2^{\alpha_i}$ is an irreducible module whose highest weight is equal to $d$ and $V^{\alpha_i}_{a,b}$ denotes a $\fb$-module whose restriction to $\fb_i$ is isomorphic to the indecomposable module $V_{a,b}$ from~\eqref{eq::V_ab}.

		Both functors $\eta\circ\RCoind_i$ and $\LInd_i\circ\eta$ are different from zero, and the images of $\RCoind_i(\Bbbk_{\lambda})$ and $\LInd_i(\Bbbk_{\lambda})$ contain $L^{\alpha_i} (d) \xi^{-1}$ and $L^{\alpha_i}( d )$ by inspection. In view of Example \ref{ex::sl_2::T_i},  the natural transformation $( \eta\circ\RCoind_i + \LInd_i\circ\eta )$ in  (\ref{eq::Ind::spherical}) induces an isomorphism after the application to $\Bbbk_\la$ (with $\langle\lambda,\alpha^{\vee}\rangle \gg 0$).
		
		As mentioned in Corollary~\ref{cor::sl2::tensor} all functors involved in (\ref{eq::Ind::spherical}) commute with tensoring with $\msl_2^{i}$-integrable modules.
		From the above, we know the equivalence~\eqref{eq::Ind::spherical} holds for $\langle\lambda,\alpha_i\rangle \ge d_0$ for some $d_0 \in \mathbb{Z}$. By tensoring with the standard irreducible two-dimensional $\msl_2$-module and applying the long exact sequences, we deduce the isomorphism between the lefthand side and the righthand side of~\eqref{eq::Ind::spherical} as a $\fb$-modules for $\langle\lambda,\alpha_i\rangle \ge d \in \mathbb{Z}$ provided if we know the isomorphism for $\langle\lambda,\alpha_i\rangle \ge d + 1$. Thus, \eqref{eq::Ind::spherical} yields an isomorphism for every $\la$. Again by applying the long exact sequence coming from the extension of finite-dimenaional $\fb[\xi]$-modules, we conclude that the natural transformation of functors~\eqref{eq::Ind::spherical} yields an isomorphism after applying to every object in $\mathcal O (\fb [\xi] )$. Therefore, the natural transformation~\eqref{eq::Ind::spherical} is an isomorphism of functors. The proof of the isomorphism~\eqref{eq::CoInd::spherical} is completely analogous.
	\end{proof}
	From general formalism of spherical functors described in Proposition~\ref{prp::spherical::general}, we have
	\begin{cor}
		\label{cor::Ind::spherical}	
		The endofunctors $\sT_i$ and  $\sT'_i$ defined by distinguished triangles
		\begin{equation}
			\label{eq::triangle::T}    
			\Id_{\fb} \stackrel{\eta}{\to}  \Res_i\circ\LInd_i \to \sT_i \stackrel{+1}{\to} \Id_{\fb}[1], \quad
			\sT_i' \to \Res_i\circ\RCoind_i \stackrel{\varepsilon}{\to} \Id_{\fb}  \stackrel{+1}{\to} \sT'_i[1]
		\end{equation}	
		are mutually inverse autoequivalences of  the derived category $\bD^b(\cO(\fb[\xi]))$. 
	\end{cor}	
	Together with Isomorphism~\eqref{eq::ind::coind::iso}, we get the following collections of equivalences of derived endofunctors of $\bD^b(\cO(\fb[\xi])))$:
	\begin{equation}
		\label{eq::Hecke::T::T'}
		\sT_{i}' \simeq \cone(\sv^{-1}\circ \Res_i\circ\LInd_{i} \to \Id[-1]) \simeq
		\cone\left(\sv^{-1}\circ \left(\cone(\sT_{i}[-1]\stackrel{\phantom{}^{\phantom{\frac{1}{2}}}}{\to}  \Id)\right) \to \Id[-1]\right) 
	\end{equation}
	which categorifies the quadratic relations:	
	\begin{equation*}
		\hT_{i}^{-1} = t^{-1}(\hT_{i}+1) - 1 \ \Leftrightarrow \ (\hT_i+1)(\hT_i-t) = 0
	\end{equation*}
	known for the generators $\hT_{i}$ in the Hecke algebras.
	
	Since the left and right adjoints to a spherical functor are spherical functors, we have:
	\begin{cor}
		The induction functor $\LInd_i$ is a spherical functor whose right adjoint is $\Res_{i}$ and whose left adjoint is isomorphic to $\Res_i\sv^{-1}[-1]$ (the restriction functor with the shifted $\xi$ and homological gradings). 
		\[
		\begin{tikzcd}
			\bD^{b}(\cO({\fp_{i}}[\xi]))
			\arrow[rr, shift left = 3, "\Res_i\sv^{-1}{[-1]}"]
			\arrow[rr, shift right = 3, "\Res_{i}"', "\perp" near end]
			&&
			\bD^{b}(\cO(\fb[\xi]))
			\arrow[ll, "\LInd_{i}" description,"\perp"' near start]
		\end{tikzcd}.
		\]
		In particular, we have equivalences of endofunctors of $\bD^b(\cO(\fp_i[\xi])))$:
		\begin{equation}
			\label{eq::Res::spherical}
			\begin{tikzcd}
				\Res_i\sv^{-1}[-1] \oplus \Res_i \ar[r,"\eta\circ\Res+\Res\circ\eta"] & 
				\Res_i\sv^{-1}[-1] \circ \LInd_i\circ \Res_i \ar[r,"\Res\circ\varepsilon +\varepsilon\circ\Res"] & 
				\Res_i\sv^{-1}[-1] \oplus \Res_i.
			\end{tikzcd}	
		\end{equation}
	\end{cor} 
	By shifting the Isomorphism~\eqref{eq::Res::spherical}, we get one more isomorphism of triangulated functors that can be verified directly for $\msl_2[\xi]$-modules:
	\begin{equation}
		\label{eq::Ind::sperical}
		\Res_i \oplus \Res_i\sv[1] \stackrel{\simeq}\rightarrow \Res_i \circ \LInd_i\circ\Res_i.
	\end{equation}
	Applying one more $\LInd_i$ from the right to Isomorphism~\eqref{eq::Ind::sperical}, we end up with another categorification of the quadratic relation~\eqref{eq::Hecke::alg} known for generators $\hT_i$ in the Hecke algebra:
	\begin{multline}
		\label{eq::DD=D+D}
		\sD_i \circ \sD_i \equiv (\Res_i\circ\LInd_i)\circ(\Res_i\circ\LInd_i) \simeq 
		\\ \simeq
		(\Res_i\circ\LInd_i)\oplus (\Res_i\circ\LInd_i)\sv[1] \equiv \sD_i \oplus \sD_i\sv[1].
	\end{multline}
	The Demazure endofunctors $\sD_i:=\Res_i\circ\LInd_i$ categorify the element $\hT_i+1$ in the Hecke algebra.
	Hence, the decategorification of Isomorphism~(\ref{eq::DD=D+D}) looks as follows:
	$$
	(\hT_i+1)(\hT_i+1) = (\hT_i+1)+t(\hT_i+1) \ \Leftrightarrow (\hT_i+1)(\hT_i-t) = 0.
	$$

	\subsection{Iterating inductions for sufficiently dominant weights}
	\label{sec::Demazure::modules}
	We finish this section with a description of the family of modules $\bWD_{\lambda}$ which we call super-Demazure modules.
	Each module can be realized as an iterated composition of appropriate induction functors applied to one-dimensional modules. 
	The $\xi$-degree zero parts of these modules coincide with the classical Demazure modules.

	\begin{definition}
		With each weight $\lambda\in \sfP$ we assign a cyclic $\fb[\xi]$-module $\bWD_{\lambda}$ generated by a cyclic vector $v_{\lambda}$ of weight $\lambda$ subject to the integrability relations:
		\begin{equation}
			\label{eq::Demazure::relations}
			\begin{array}{c}
				\forall \alpha\in\Delta_{+} \text{ such that } \langle\alpha^{\vee},\lambda\rangle<0  \quad e_{\alpha}^{-\langle\alpha^{\vee},\lambda\rangle+1} v_{\lambda} = 0, \\
				\forall \alpha\in\Delta_{+} \text{ such that } \langle\alpha^{\vee},\lambda\rangle\geq 0  \quad e_{\alpha} v_{\lambda} = 0, \\
				\forall h\in \fh \ (h\xi)v_{\lambda} = 0.
			\end{array}
		\end{equation}	 	
	\end{definition}
	
	\begin{thm}
		\label{thm::Ind::Demazure}	
		Suppose $\lambda$ is a dominant integral weight and $s_{i_1}\ldots s_{i_k}$ is a reduced decomposition of a given element $\sigma$ of the Weyl group associated with the root system $\Delta$.
		Then we have an isomorphism of $\fb[\xi]$-modules:
		$$(\Res_{\alpha_{i_1}}\circ\Ind_{\alpha_{i_1}})\circ \ldots \circ(\Res_{\alpha_{i_k}} \circ \Ind_{\alpha_{i_k}}) (\Bbbk_{\lambda}) \simeq \bWD_{\sigma\lambda}$$ and, in particular, the iterated induction does not depend on a reduced decomposition of $\sigma$.
	\end{thm}
	\begin{proof}
		The proof is by induction on the length of $\sigma$ and repeats the standard arguments known for the Demazure modules in the classical case.
		The base of induction corresponds to the case $\sigma=1$. For dominant $\lambda$ we have $\langle \lambda, \alpha_i^{\vee} \rangle \geq 0$ for all $ {i}\in\tI$ and, consequently, $\bWD_{\lambda}=\Bbbk_{\lambda}$ since the elements $\{e_i|{i}\in\tI\}\cup \fh[\xi]$  generates the Lie superalgebra $\fb[\xi]$.
		
		For the induction step, we suppose that a simple root ${i}\in\tI$ is chosen in such a way that
		$l(s_i\sigma)=l(\sigma)+1$.
		Note that the Zukerman induction functor $\Ind_i$ is the quotient of the usual induction. Therefore, the $\fp_i[\xi]$-module $\Ind_i(\bWD_{\sigma(\lambda)})$ admits a cyclic vector $v_{\sigma\lambda}$ of weight $\sigma(\lambda)$ (that equals the image of the cyclic vector $v'$ of the $\fb[\xi]$-module $\bWD_{\sigma(\lambda)}$ of weight $\sigma(\lambda)$). However, the condition $l(s_i\sigma)=l(\sigma)+1$ implies that $d:=\langle\sigma(\lambda),\alpha_i^{\vee}\rangle >0$ and the $\msl_2^{\alpha_i}$-submodule generated by $v'$ is of dimension $d+1$.
		Consequently, the following list of relations are satisfied for the action of $\fp_i[\xi]$ on
		$v'$ in the cyclic module $\Ind_i(\bWD_{\sigma(\lambda)})$:
		\begin{eqnarray}
			\label{eq::Demazure::relations::1a}
			e_i v_{\sigma\lambda} = 0,\  f_i^{d+1} v_{\sigma\lambda} = 0, \\
			\label{eq::Demazure::relations::1b}
			\forall \alpha\in\Delta_{+} \text{ such that } \langle\alpha^{\vee},\sigma(\lambda)\rangle<0  \quad e_{\alpha}^{-\langle\alpha^{\vee},\sigma(\lambda)\rangle+1} v_{\sigma\lambda} = 0, \\
			\label{eq::Demazure::relations::1c}
			\forall \alpha\in\Delta_{+}\setminus\{\alpha_i\} \text{ such that } \langle\alpha^{\vee},\sigma(\lambda)\rangle\geq 0  \quad e_\alpha v_{\sigma\lambda} = 0, \\
			\label{eq::Demazure::relations::1d}	
			\forall h\in \fh \ (h\xi)v_{\sigma\lambda} = 0.
		\end{eqnarray}
		Moreover, this is the full list of relations on the cyclic $\fp_i[\xi]$-module $\Ind_i(\bWD_{\sigma(\lambda)})$ since it is a maximal quotient of the module $U(\fp_i[\xi])\otimes_{U(\fb[\xi])}\bWD_{\sigma(\lambda)}\simeq \Bbbk[f_i]\otimes(1,f_i\xi)\otimes\bWD_{\sigma(\lambda)}.$
		The module 	$\Ind_i(\bWD_{\sigma(\lambda)})$  is $\msl_2^{\alpha_i}$-integrable and, in particular, the weight decomposition is stable under the reflection $s_i$. The automorphism $s_i$ acts on the parabolic subalgebra $\fp_i$ and consequently on the $\fp_i[\xi]$-integrable modules. The image under $s_i$ of the cyclic vector $v_{\sigma\lambda}$ is the vector $v_{s_i\sigma\lambda}:=e_i^{d} v_{\sigma\lambda}$ of weight $s_i\sigma(\lambda)$. 
		The relations on the $\fp_i[\xi]$-cyclic vector $v_{s_i\sigma\lambda}$ coincide with the image of $s_i$ of Relations~\eqref{eq::Demazure::relations::1a}-\eqref{eq::Demazure::relations::1d}. Note that the condition $l(s_i\sigma)=l(\sigma)+1$ implies the coincidence of sets 
		\begin{multline*}
			\{\alpha\in\Delta : \langle\alpha^{\vee},\sigma(\lambda)\rangle<0 \} 
			= \Delta_{+}\cap \sigma^{-1}(\Delta_{-}\setminus\{-\alpha_i\}) = 
			\\
			= \Delta_{+}\cap \sigma^{-1}s_i(\Delta_{-}\setminus\{-\alpha_i\}) =
			\{\alpha\in\Delta_{+}\setminus\{\alpha_i\} : \langle\alpha_j^{\vee},s_i\sigma(\lambda)\rangle<0 \},
		\end{multline*} 
		and the complement of these sets in $\Delta_{+}\setminus\{\alpha_i\}$ that index Relations~\eqref{eq::Demazure::relations::1c}.
		The image under $s_i$ of Relations~\eqref{eq::Demazure::relations::1a} interchanges $e_i$ and $f_i$:
		$$
		f_iv_{s_i\sigma\lambda}=0\ \& \ e_i^{d+1}v_{s_i\sigma\lambda} = 0.
		$$
		Therefore, the $\fp_i[\xi]$-module $\Ind_i(\bWD_{\sigma(\lambda)})$ contains a cyclic vector $v_{s_i\sigma\lambda}$ of weight $s_i\sigma\lambda$ and the relations for it are Relations~\eqref{eq::Demazure::relations} and the relation $f_i v_{s_i\sigma\lambda}=0$ implies that
		$$ f_i\xi v_{s_i\sigma\lambda} = -\frac{1}{2}[h\xi,f_i] v_{s_i\sigma\lambda}  =0$$ 
		and, consequently, $v_{s_i\sigma\lambda}$ is a cyclic vector of the $\fb[\xi]$-module $\Res_i\Ind_i(\bWD_{\sigma(\lambda)})$. The comparison of relations for the cyclic vector finishes the comparison of $\Res_i\Ind_i(\bWD_{\sigma(\lambda)})$ and $\bWD_{s_i\sigma(\lambda)}$.
	\end{proof}

	\section{Categorification of AHA and its polynomial representation}
	\label{sec::proof}
	
	\subsection{Statement of a categorification}
	\label{sec::AHA::categorify}
	
	In this section we will explain the meaning of the categorification of the Affine Hecke Algebra (AHA) and its action in the faithful polynomial representation $\bZ_{t}[\sfP]$ via the endofunctors $\sD_i$,$\sT_i$ of the derived category $\bD^{b}(\cO(\fb[\xi]))$ defined in \S\ref{sec::Demazure::all} and outlined in Notation~\ref{not::demazure::twist}.
	
	\begin{thm}
		\label{thm::DAHA}	
		The endofunctors $\{\sT_i |i\in \tI\}$, $\{\sX^{\mu}|\mu\in \sfP\}$, and $\{\sv^m|m\in \bZ \}$,	
		categorify the action of the affine Hecke algebra in the basic representation $\bZ_t[\sfP]$
		under the following correspondence:
		\[ \hT_i \leftrightarrow \sT_i, \ \hX^{\mu} \leftrightarrow \sX^{\mu}, \ t \leftrightarrow \sv[1].
		\]	
		More precisely,	on the level of characters the action of endofunctors $\sT_i$, $\sX^{\mu}$ and $\sv[1]$ on $\bZ_{t}[\sfP]$ coincides with the action of the generators of AHA: $\hT_i$, $\hX^{\mu}$ and $t$ correspondingly. Moreover, there exists the following list of isomorphisms of triangulated endofunctors
		\begin{gather}
			\label{eq::Hecke}
			\begin{array}{c}
				{	\sT_i \circ \sT'_i \simeq \sT_i'\circ \sT_i \simeq \Id_{\fb {[\xi]}}; } \\
				{ \sT'_i \simeq  \cone((\sv\circ\sD_i) \to \Id_{\fb {[\xi]}}[-1]) \simeq \cone((\sv\circ\cone(\sT_i[-1] \to \Id_{\fb {[\xi]}}))\to \Id_{\fb {[\xi]}}[-1])   }
			\end{array}
			\\
			\label{eq::shift1}
			{  \sT_i \circ \sX^{\mu} \simeq \sX^{\mu} \circ \sT_i,} {~{\text if}~ \langle \mu,\alpha_i^{\vee} \rangle=0; }\\
			\label{eq::shift2}
			{ \sT_i \circ \sX^{\mu}  \simeq \sv\sX^{\mu-\alpha_i}\circ\sT'_i[1],}{~{\text if}~ \langle \mu, \alpha_i^{\vee} \rangle=1.}
			\\
			\label{eq::braid} \underbrace{\sT_i \circ \sT_j \circ \ldots }_{m_{{i,j}} \text{ factors}} \simeq \underbrace{\sT_j \circ \sT_i \circ \ldots }_{m_{{i,j}}\text{ factors}} \text{ here } m_{{i,j}} = ord_{W}(s_is_j)
		\end{gather}	
		that categorify the corresponding defining Relations \eqref{eq::Hecke::alg}, \eqref{eq::rel::braid1}, \eqref{eq::rel::braid2}  and \eqref{eq::rel::braid0} of the Affine Hecke algebra.
	\end{thm}

	Part of this theorem was already proved in the preceding section~\S\ref{sec::Demazure::all} and the remaining equivalences of functors will be explained later in this section. More concretely, our verification of Theorem \ref{thm::DAHA} is organized as follows:
	\begin{itemize}
		\item We showed the coincidence of characters in Proposition~\ref{prp::def::decategorify} and in Corollary~\ref{cor::T_i::char}. 
		\item The proof of the \emph{Hecke} quadratic relation~\eqref{eq::Hecke} was shown in~\S\ref{sec::sherical}. (See Isomorphism~\eqref{eq::Hecke::T::T'} followed from Corollary~\ref{cor::Ind::spherical}).	 \item Relations~\eqref{eq::shift1}, \eqref{eq::shift2} are also based on a computation for $\fg=\msl_2$ and are explained in Section~\ref{sec::shift}  (Corollary~\ref{cor::AHA::rel::categorify}).
		\item The  case-by-case proof (except $G_2$) of the  \emph{Braid Relations}~\eqref{eq::braid}.
		is contained in Section~\ref{sec::braid}.
		The strategies in all cases are the same:
		\noindent
		First, we expand all arrows in the iterated composition, representing $\sT_i$ as a cone of a morphism $\Id_{\fb}\to \sD_i$ and get a complex of functors.
		Second, we define a complex of functors (called $\BGG_{{ij}}$) whose components are numbered by the vertices of the Bruhat graph associated with the corresponding rank $2$ subsystem generated by a couple of simple roots $\alpha_i$ and $\alpha_j$. Finally, we show that initial expansion of $\sT_i$'s is quasi isomorphic to $\BGG_{ij}$ and it is not difficult to see that $\BGG_{ij}\simeq \BGG_{ji}$
		what implies Braid relation~\eqref{eq::braid}.
	\end{itemize}
	
	\subsection{General strategy}
	\label{sec::strategy}
	Theorem~\ref{thm::DAHA} consists of several isomorphisms of different endofunctors of the bounded derived category of the abelian category $\cO(\fb[\xi])$.
	The proof of each of these isomorphisms comply with the following scheme:
	
	First, we present a natural transformation $\epsilon:\sF\Rightarrow \sG$ connecting the pair of endofunctors $\sF,\sG\in End(\bD^{b}(\cO(\fb[\xi])))$ that appear in the left hand and the right-hand sides of the stated isomorphism. This is the most conceptual part of each isomorphism and is explained in detail in each particular case.
	Note that all functors under consideration interact with the finite subsystem $\Pi_0\subset \Pi$  of ranks $1$ or $2$. Therefore, it is enough to explain the corresponding equivalence $\varepsilon:\sF\Rightarrow \sG$  for the corresponding rank $\leq 2$ finite-dimensional Lie algebra $\fg_0$.
	
	For the second step, we check that for all sufficiently dominant integral weights $\lambda$ the map $\epsilon_{\Bbbk_{\lambda}}:\sF(\Bbbk_{\lambda}) \rightarrow \sG(\Bbbk_{\lambda})$ is a quasi isomorphism.
	Here $\Bbbk_{\lambda}$ ($\lambda\in \fP$) exhausts the set of isomorphism classes of irreducible one-dimensional $\fb[\xi]$ modules up to grading shifts, and sufficiently dominant means that $\langle\lambda,\alpha_i\rangle \gg 0$ for all $\alpha_i\in\Pi_0$.
	This step is also checked directly in each particular case. The dominance condition implies vanishing of almost all  cohomologies of $\sF(\Bbbk_{\lambda})$ and $\sG(\Bbbk_{\lambda})$.
	
	The remaining third step is the same in all cases and we will not repeat it many times.
	Thanks to Corollary~\ref{cor::sl2::tensor} we notify that both endofunctors $\sF$ and $\sG$ commutes with tensor products with $\fp[\xi]$-finite-dimensional modules $M$ for the parabolic subalgebra $\fp\supset\fb$ associated with the subset $\Pi_0$:
	$$ 
	\sF(\Res(M)\otimes - ) \simeq M\otimes \sF(-)\quad \& \quad \sG(\Res(M)\otimes -) \simeq  M\otimes \sG(-)$$
	The abelian category of finite-dimensional $\fb[\xi]$ modules is generated by the $\Bbbk_{\lambda}$ with $\lambda$ sufficiently dominant together with their tensor products with $\fp[\xi]$-finite-dimensional modules.
	The category $\cO(\fb[\xi])$ is the full subcategory of the abelian category generated by pro-objects constructed from finite-dimensional ones.	
	This implies that $\epsilon:\sF\to \sG$ is an equivalence of triangulated functors.
	
	\subsection{Relations \eqref{eq::shift2}}
	\label{sec::shift}
	The commutativity of $\sX^{\lambda}$ and $\sT_i$ for $\langle\lambda,\alpha_i^{\vee}\rangle =0$ is obvious since they work with noninteracting weights, and hence $\sX^{\lambda}$ commutes with $\sD_i$. Relation~\eqref{eq::shift2} is more involved and requires the map of functors.
	For $\lambda\in\sfP$ such that $\langle\lambda,\alpha_i^{\vee}\rangle=1$, we have a natural transformation $\varphi$ between even induction functors:
	$$\varphi:\Res_{i}^{\bar{0}}\circ\LInd_{i}^{\bar{0}}\circ \sX^{\mu} \rightarrow \sX^{\mu-\alpha_i} \circ \Res_{i}^{\bar{0}}\circ\LInd_{i}^{\bar{0}} 
	$$ 
	thanks to Corollary~\ref{cor::Ind:0::X}.
	
	We define a natural transformation $\bar{\varphi}:=\Res_i^{\bar{1}}\circ\varphi\circ\Ind_{i}^{\bar{1}}:\sD_i\circ\sX^{\mu} \to \sX^{\mu-\alpha_i} \circ \sD_i$ as:
	\begin{multline}
		\bar{\varphi}:\sD_i\circ\sX^{\mu}\simeq 
		\Res_i\circ\Ind_i\circ\sX^{\mu} \simeq 
		\Res_i^{\bar{1}}\circ\Res_{i}^{\bar{0}}\circ\LInd_{i}^{\bar{0}}\circ\LInd_i^{\bar{1}}\circ \sX^{\mu}
		\simeq \\
		\simeq
		\Res_i^{\bar{1}}\circ\Res_{i}^{\bar{0}}\circ(\LInd_{i}^{\bar{0}}\circ \sX^{\mu}) \circ\LInd_i^{\bar{1}} 
		\stackrel{\Res_i^{\bar{1}}\circ\varphi\circ\Ind_{i}^{\bar{1}}}{\longrightarrow} 
		\Res_i^{\bar{1}}\circ(\sX^{\mu-\alpha_i} \circ \Res_{i}^{\bar{0}}\circ\LInd_{i}^{\bar{0}})\circ\LInd_i^{\bar{1}} \simeq \\
		\simeq \sX^{\mu-\alpha_i} \circ \Res_i \circ \LInd_i
		\simeq \sX^{\mu-\alpha_i} \circ \sD_i.
	\end{multline}
	\begin{lem}
		The following natural transformations of triangulated endofunctors of $\bD^{b}(\cO(\fb))$ assemble into a distinguished triangle
		
		\begin{equation}
			\label{eq::DX::triangle}
			\begin{tikzcd}
				\sX^{\mu} \arrow[r,"\eta"] & \Res_i^{\bar{0}}\circ\LInd_i^{\bar{0}}\circ\sX^{\mu} \arrow[r,"\varphi"] &
				\sX^{\mu-\alpha_i}\circ \Res_i^{\bar{0}}\circ\LInd_i^{\bar{0}}
				\arrow[r,"+1"] & \sX^{\mu}[1].
			\end{tikzcd}
		\end{equation}
		The composition with exact functors $\Res_i^{\bar{1}}$ from the left and $\Ind_i^{\bar{1}}$ from the right to the distinguished triangle~\eqref{eq::DX::triangle} leads to the following distinguished triangle of endofunctors of $\bD^{b}(\cO(\fb[\xi]))$:
		\begin{equation}
			\label{eq::DX=XD::triangle}
			\begin{tikzcd}
				\Res_{i}^{\bar{1}}\circ\Ind_{i}^{\bar{1}}\circ\sX^{\mu} 
				\arrow[r,"\eta"] & 
				\sD_i\circ\sX^{\mu} \arrow[r,"\bar\varphi"] & \sX^{\mu-\alpha_i} \circ \sD_i \arrow[r,"+1"] & \Res_{i}^{\bar{1}}\circ\Ind_{i}^{\bar{1}}\circ\sX^{\mu}[1].
			\end{tikzcd}
		\end{equation}
	\end{lem}
	\begin{proof}
		The proof uses our standard strategy. First one verifies that for sufficiently big dominant weight $\lambda$ we do not have derived component of the induction and have a short exact sequence of $\fb$-modules:
		$$
		\begin{tikzcd}
			0\arrow[r] & \Bbbk_{\lambda+\mu} \arrow[r] &  L^{i}(\lambda+\mu) \arrow[r]   & \sX^{\mu-\alpha}(L^{i}(\lambda)) \arrow[r] & 0 
		\end{tikzcd}
		$$
		Here $\Bbbk_{\lambda+\mu}=\sX^{\mu}(\Bbbk_{\lambda})$ and $\Res_{i}^{\bar{0}}(\Ind_i^{\bar{0}}(\Bbbk_{\lambda+\mu}))$ equals the irreducible $\msl_2^{i}$-module $L^{i}(\lambda+\mu)$ with the highest weight $\lambda+\mu$.
		Moreover, all functors commutes with taking tensor product with $\msl_2^{i}$-integrable modules what follows that the triangle~\eqref{eq::DX::triangle} is distinguished. Triangle~\eqref{eq::DX=XD::triangle} is distinguished because functors $\Res_i^{\bar{1}}$ and $\Ind_i^{\bar{1}}$ are exact and commutes with $\sX^{\bullet}$.
	\end{proof}
	\begin{prop}
		\label{prp::dist::triangle}
		There exist a distinguished triangle of endofunctors of $\bD^{b}(\cO(\fb[\xi]))$:
		\begin{equation*}
			\begin{tikzcd}
				\sv\sX^{\mu+\alpha_i} \ar[r] & \sT_i\circ \sX^{\mu} \ar[r] & \sX^{\mu-\alpha_i}\circ\sD_i \arrow[r,"+1"] & \sv\sX^{\mu+\alpha_i}[1]
			\end{tikzcd}. 
		\end{equation*}
	\end{prop}
	\begin{proof}
		Consider the following diagram of natural transformations between endofunctors of the triangulated category $\bD^{b}(\cO(\fb[\xi]))$:
		\begin{equation}
			\label{eq::octahedron}
			\begin{tikzcd}
				& & \sT_i\circ\sX^{\mu} \arrow[ld,"+1"'] \arrow[rrdd,dotted,bend left]
				\arrow[dd,phantom,"\circlearrowleft" near start] & & \\
				& \sX^{\mu} \arrow[rr,"\eta"description] \arrow[rd,"\eta^{\bar{1}}"description] & & \sD_i\circ \sX^{\mu} \arrow[rd,"\bar\varphi"] \arrow[ul] & \\
				\sv\sX^{\mu-\alpha_i} \arrow[ur,"+1"] \arrow[rruu,dotted,bend left] 
				\arrow[rrru,phantom,"\circlearrowright" near start]
				&& \Res_i^{\bar{1}}\circ\Ind_i^{\bar{1}} \circ \sX^{\mu} \arrow[ur,"\eta^{\bar{0}}" description] \arrow[ll] 
				&& \sX^{\mu-\alpha_i}\sD_i \arrow[ll,"+1"] \arrow[llll,"+1"',bend left, dotted]
				\arrow[lllu,phantom,"\circlearrowright" near start]    
			\end{tikzcd}.
		\end{equation}
		The natural transformation $\eta=\eta^{\bar{0}}\circ\eta^{\bar{1}}$ is the counit for adjoint functors which is a composition of even and odd counits. 
		The circled arrow symbols $\circlearrowleft$ underline the distinguished triangles~\eqref{eq::triangle::T}, \eqref{eq::Ind::od::proj}, and~\eqref{eq::DX=XD::triangle}, while $+1$ corresponds to the shifted arrow in a distinguished triangle. Therefore, the {\it octahedral axiom} of triangulated categories implies the existence of the distinguished triangle whose maps are drawn with dotted arrows.
		The latter distinguished triangle is the one we are looking for.
	\end{proof}
	\begin{cor}
		\label{cor::AHA::rel::categorify}
		The endofunctors $\sT_i\circ\sX^{\mu}$ and $\sv\circ\sX^{\mu-\alpha_i}\circ \sT_i'[1]$ are isomorphic.
	\end{cor}
	\begin{proof}
		We have an isomorphism of $2$ of $3$ endofunctors between two distinguished triangles:
		\begin{equation}
			\label{eq::2::of::3}
			\begin{tikzcd}
				\sv\sX^{\mu-\alpha_i} \arrow[r] \arrow[d,equal] & \sT_i \circ \sX^{\mu} \arrow[r]  & \sX^{\mu-\alpha_i}\circ\sD_i \arrow[r,"+1"] \arrow[d,equal,"\eqref{eq::ind::coind::iso}"] & \sv\sX^{\mu-\alpha_i}[1] \arrow[d,equal] \\
				\sv\sX^{\mu-\alpha_i} \arrow[r,"+1"] & \sv\sX^{\mu-\alpha_i}\circ\sT_i'[1] \arrow[r] & \sX^{\mu-\alpha_i} \circ\sD_i'\sv[1] \arrow[r,"\varepsilon"] & \sv\sX^{\mu-\alpha_i}[1] 
			\end{tikzcd}.    
		\end{equation}    
		Where the top distinguished triange was defined in Proposition~\ref{prp::dist::triangle} while the bottom one follows from the definition of the Demazure cotwist $\sT_i'$ (see e.g. Notation~\ref{not::demazure::twist}). As always, all endofunctors in Diagram~\eqref{eq::2::of::3} commute with 
		tensoring with $\msl_2^{i}$-integrable modules. Therefore, in order to check that this diagram is commutative it is enough to verify it for one-dimensional modules $\Bbbk_{\lambda}$ with $\lambda$ sufficiently antidominant. This follows from Example \ref{ex::sl_2::T_i} by inspection.
		The $2$-of-$3$ axiom of triangulated categories implies the isomorphism of endofunctors $\sT_i\sX^{\mu}$ and $\sv\sX^{\mu-\alpha_i} \sT_i'[1]$.
	\end{proof}

	\subsection{Braid relations for $\sT_i$,$\sT_{j}$}
	\label{sec::braid}
	
	Suppose that we have $\tI_0 =\{i,j\} \subset \tI$. We have a parabolic subalgebra $\fp_{\tI_0}$ and the rank two semi-simple Lie subalgebra $\fg_{\tI_0}$ of $\g$, that we denote by $\fp_{ij}$ and $\fg_{ij}$, respectively. Then, we have the restriction functor
	$$\Res_{{ij}} : \mathcal O(\fp_{{ij}}[\xi]) \longrightarrow \mathcal O(\fb[\xi])$$
	from the category $\mathcal O(\fp_{{ij}}[\xi])$ of $(\h,\xi)$-graded $\fp_{ij}$-integrable finitely generated $\fp_{{ij}}[\xi]$-modules, and the induction functor $\Ind_{{ij}} : \mathcal O(\fb[\xi]) \to \mathcal O(\fp_{ij}[\xi])$ that sends $M$ to the maximal $\g_{ij}$-integrable quotient of $U(\fp_{{ij}}[\xi])\otimes_{U(\fb[\xi])}M$.
	
	It is obvious that if  $\langle\alpha_i,\alpha_j\rangle =0$ then $\sD_i$ and $\sD_j$ commute and consequently $\sT_i\sT_j = \sT_j\sT_i$.
	
	\subsubsection{$A_2$ case}
	\label{sec::braid::A2}
	
	Suppose that the rank $2$ root subsystem generated by ${i,j} \in \tI$ is $A_2$.
	In particular, we assume that $\langle\alpha_i,\alpha_j\rangle =-1$ and we denote by $\msl_3^{{ij}}$ ( resp. $\fp_{{ij}}[\xi]$) the corresponding simple (resp. parabolic) Lie subalgebras of $\fg$ (resp. $\fg[\xi]$) associated to the subset $\{{i,j}\} \subset \tI$.
	
	Recall that the derived induction functor $\LInd_i$ as well as the composition $\sD_i=\Res_i\circ\LInd_i$ are the left derived functors associated with the right exact functors $\Ind_i$ and $\Res_i\circ\Ind_i$, respectively.
	First, let us describe a morphism of additive functors which implies a morphism of derived ones.
	\begin{lem}
		\label{lem::Ind_map}	
		There exists a natural equivalence of additive right exact endofunctors of the abelian category $\cO(\fb[\xi])$:  \begin{equation}
			\label{map::inductions}
			\pi_{{ij}}: \Res_i\circ \Ind_i\circ \Res_j\circ  \Ind_j \circ\Res_i\circ  \Ind_i \rightarrow \Res_{{ij}}\circ 	\Ind_{{ij}}
		\end{equation}
	\end{lem}	
	Lemma~\ref{lem::Ind_map} is very well known for even inductions and was first noticed by Demazure. See e.g.~\cite{Joseph}  and~\cite{Arkhipov} where the modern categorical setup of Demazure functors is used. Most of the proofs for the classical (nonsuper) version are based on the geometry of Bott-Samelson varieties. We suggest below the algebraic proof of this statement without referring to geometry.
	\begin{proof}
		Recall that the inductions we are dealing with are the maximal integrable quotients of the appropriate ordinary inductions from subalgebras.
		Consequently, the embedding of parabolic subalgebras $\fp_i[\xi] \subset \fp_{{ij}}[\xi]$ predicts the natural transformation between functors of  corresponding inductions:
		$$
		\pi_i:\Res_i \circ \Ind_i \to \Res_{{ij}} \circ \Ind_{{ij}}.
		$$
		Moreover, for each $\fb[\xi]$-module $M$ the module $\Res_{{ij}} \circ \Ind_{{ij}}(M)$ is $\msl_3^{{ij}}$-integrable and, in particular, $\msl_2^{j}$-integrable. The commutativity with the tensor product stated in Corollary~\ref{cor::sl2::tensor} explains the isomorphism of additive (underived) functors
		\begin{equation}
			\label{eq::res::integrab}
			\Res_j\circ\Ind_j (\Res_{{ij}} \circ \Ind_{{ij}}(M)) =
			(\Res_{{ij}} \circ \Ind_{{ij}}(M)) \otimes \Res_j\circ\Ind_j (\Bbbk_0) =
			\Res_{{ij}} \circ \Ind_{{ij}}(M)
		\end{equation}
		The last equality follows from the observation $\Ind_i(\Bbbk_0)=\Bbbk_0$ that can be either checked by hands or figured out from Theorem~\ref{thm::Ind::Demazure}.
		The same isomorphism also makes sense for $j$ replaced with $i$. Finally, we define the morphism of functors
		$\pi_{{ij}}:= (\Res_i\circ \Ind_i\circ \Res_j\circ  \Ind_j)\circ \pi_i$.
		Together with Isomorphisms~\eqref{eq::res::integrab} we have
		\begin{multline}
			\pi_{{ij}}\colon \
			( \Res_i\circ \Ind_i)\circ (\Res_j\circ  \Ind_j) \circ (\Res_i\circ  \Ind_i) \stackrel{\pi_i}{\longrightarrow}   \\
			( \Res_i\circ \Ind_i)\circ( \Res_j\circ  \Ind_j )\circ(\Res_{{ij}}\circ  \Ind_{{ij}})) \stackrel{\eqref{eq::res::integrab}}{\simeq}  \Res_i\circ \Ind_i\circ(\Res_{{ij}}\circ  \Ind_{{ij}}) \stackrel{\eqref{eq::res::integrab}}{\simeq} \Res_{{ij}}\circ  \Ind_{{ij}}.
		\end{multline}
		Thanks to the adjunction between restriction and induction functors we know that all induction functors under consideration are right exact, restriction functors are exact and a composition of right exact functors is right exact. We conclude that both functors in~\eqref{map::inductions} are right exact.
		
		We explain in Theorem~\ref{thm::Ind::Demazure}  that the superposition of consecutive inductions applied to the irreducible module $\Bbbk_{\lambda}$ leads to the super-Demazure module $\bWD_{s_is_js_i(\lambda)}$ that happens to be the cyclic $\fp_{{ij}}$-integrable module generated by the cyclic vector of weight ${s_is_js_i(\lambda)}$.
		On the other hand we know that $\Ind_{{ij}}(\Bbbk_{\lambda})$ is the universal cyclic $\fp_{{ij}}$-integrable module generated by the cyclic vector of weight $\lambda$.
		The $\fp_{{ij}}$-integrability implies that one can also choose as a generating cyclic vector any other extremal vector in the $\msl_3^{{ij}}$ irreducible representation. In particular the vector of the weight ${s_is_js_i(\lambda)}$.  Hence, $\pi_{{ij}}$ is an isomorphism for irreducible $\Bbbk_{\lambda}$ with $\lambda$ sufficiently dominant. On the other hand, all functors commute with tensor products with $\msl_3^{{ij}}$-integrable modules. Therefore, $\pi_{{ij}}$ is an isomorphism of additive functors.
	\end{proof}	
	
	Corollary~\ref{cor::H::Ind} says  that  $\sD_i$ has nonvanishing cohomology in the $-1$'st and in the $0$'th degrees.
	Consequently, the composition $\sD_i\circ\sD_j\circ\sD_i$ has only nonpositive cohomology and the natural transformation $\pi_{{ij}}$ of $0$'th cohomology described in Lemma~\ref{lem::Ind_map} extends by universal property to a natural transformation of derived functors:
	\[
	\mathbf{L}\pi_{{ij}}:  \sD_i\circ\sD_j\circ\sD_i \longrightarrow \sD_{{ij}}.
	\]
	This is not an equivalence of derived functors and the next paragraph explains the difference (the cone) between these functors.

	Consider the composition of two counit morphisms coming from adjunctions:
	\[\Id_{\fp_i} \stackrel{\eta_i}{\rightarrow} \RCoind_i\circ \Res_i
	\stackrel{\RCoind_i\circ\eta_j\circ \Res_i}{\longrightarrow} \RCoind_i\circ( \Res_j \circ\LInd_j)\circ\Res_i
	\]
	Together with the composition $\Res_i$ from the left and $\LInd_i$ from the right we get the following morphism of endofunctors:
	\[\sD_i = \Res_i\circ\LInd_i \longrightarrow \Res_i\circ\RCoind_i\circ \Res_j\circ\LInd_j\circ\Res_i\circ\LInd_i = \sD'_i\circ\sD_j\circ\sD_i \]
	Finally, thanks to Duality~\eqref{eq::Serre} between induction and coinduction functors we end up with the following morphism of endofunctors:
	\begin{equation}
		\label{eq::counit::braid}
		\sD_i \longrightarrow \sv^{-1}[-1]\sD_i\circ\sD_j\circ\sD_i \qquad \Leftrightarrow
		\qquad \eta:\sD_i\sv[1] \longrightarrow \sD_i\circ\sD_j\circ\sD_i
	\end{equation}
	
	\begin{prop}
		\label{prp::DDD}	
		There is a distinguish triangle of endofunctors of $\bD^{b}(\cO(\fb[\xi]))$:
		\begin{equation}
			\label{eq::triangle::DD}
			\sD_i\sv[1] \rightarrow \sD_i\circ\sD_j\circ\sD_i \stackrel{\mathbf{L}\pi_{{ij}}}{\longrightarrow} \sD_{{ij}}
		\end{equation}	
	\end{prop}
	\begin{proof}
		First, we check the acyclicity of the images of functors on irreducible (one-dimensional) $\fb[\xi]$-modules $\Bbbk_{\lambda}$ for sufficiently ${i,j}$-dominant weight $\lambda\gg 0$.
		We already know the coincidence of $0$'th cohomologies of  $\sD_i(\sD_j(\sD_i(\Bbbk_{\lambda})))$ and $\sD_{{ij}}(\Bbbk_{\lambda})$ with the super-Demazure module $\bWD_{s_is_js_i(\lambda)}$ thanks to Lemma~\ref{lem::Ind_map} and Theorem~\ref{thm::Ind::Demazure}.
		We claim that the direct inspection of weights shows that $\sD_i(\Bbbk_\lambda)$, $\sD_j\sD_i(\Bbbk_{\lambda})$ as well as $\sD_{{ij}}(\Bbbk_{\lambda})$ do not have nonzero cohomology. Thus, the only nonzero cohomology in Triangle~\eqref{eq::triangle::DD} is $-1$'st.
		The coincidence of $\sD^{-1}_i(\bWD_{s_js_i(\lambda)})$ and $\sD^{0}_i(\Bbbk_{\lambda}) = \Ind_i\Bbbk_{\lambda}$ can be either checked directly or verified from the equality of characters for the cohomology of functors. The compatibility with tensor products of the finite-dimensional $\fp_{ij}$-modules forces these isomorphisms to be functorial since the $\fp_{ij}$-action is responsible for extensions. Consequently, Triangle~\eqref{eq::triangle::DD} is distinguished for $\Bbbk_{\lambda}$ with $\lambda\gg 0$.
		
		Finally, we again use the fact that all functors under consideration commute with the tensor multiplication with $\msl_3$-integrable modules (Corollary~\ref{cor::sl2::tensor}).
		It remains to notice that iterated tensor products of $\{\Bbbk_{\lambda}\mid \lambda \gg 0\}$ with two fundamental representations of $\msl_3$ generate the Grothendieck group $K_0(\cO(\fb[\xi]))$.
		Therefore, Triangle~\eqref{eq::triangle::DD} is distinguished.	
	\end{proof}	
	
	Consider the commutative diagram of endofunctors:
	\begin{equation}
		\label{eq::Bruhat:A2}
		\begin{tikzcd}
			& \sD_j  \arrow[rr,"\eta_i"]  \arrow[rrd,"\eta_i" near start] && \sD_i\sD_j  \arrow[rrd,"(\mathbf{L}\pi_{{ij}})\circ\eta_i "description] & & \\
			\Id \arrow[ru,"\eta_j"] \arrow[r,"\eta_i"'] & \sD_i  \arrow[rru,"\eta_j"near start] \arrow[rr,"\eta_j"'] && \sD_j\sD_i  \arrow[rr,"(\mathbf{L}\pi_{{ij}})\circ\eta_i"'] &&
			\sD_{{ij}}
		\end{tikzcd}
	\end{equation}
	All morphisms are unit maps except the rightmost terms where one takes the composition of the unit and the left-derived functor $\pi_{{ij}}:\sD_i\sD_j\sD_i\to \sD_{{ij}}$.
	The commutativity of the diagram follows from the commutativity of unit maps.
	
	We denote by $\BGG_{{ij}}$ the total complex of the corresponding double complex:
	\[ \Id[3] \to (\sD_i\oplus\sD_j)[2] \to (\sD_i\sD_j \oplus \sD_j\sD_i)[1] \to \sD_{{ij}}.
	\]
	
	\begin{prop}\label{braidA2}
		Assume that $\langle \alpha_i, \alpha_j \rangle=-1$. Then there exists a subcomplex $C_{iji}$ and a pair of equivalences $\zeta_i$ and $\xi_i$
		\[ \sT_i\circ \sT_{j}\circ \sT_i \stackrel{\zeta_i}\longleftarrow C_{{iji}} \stackrel{\xi_i}{\longrightarrow} \BGG_{{ij}}.\]
		In particular, we have an isomorphism of triangulated functors $\sT_i\circ \sT_{j}\circ \sT_i$ and $\sT_{j}\circ \sT_i\circ \sT_{j}$ because of  the symmetry  $\BGG_{{ij}}=\BGG_{ji}$.
	\end{prop}
	\begin{proof}
		If we rephrase Proposition~\ref{prp::DDD} in terms of spherical functors we end up with sufficient properties of the braid relations formulated in~\cite{Anno}. However, the arguments are not too complicated and we want to have the self-contained proof to compare it with the analogous relations for other rank $2$ root subsystems that are not covered by~\cite{Anno}.
		
		The endofunctor $\sT_i\circ \sT_{j}\circ \sT_i$ is the successive cone that yields a triple complex, given by the commutative cube of units:
		\[
		\begin{tikzcd}[row sep=scriptsize, column sep=scriptsize]
			& \Id \arrow[dl,"\eta_j"description]
			\arrow[rrr,"\eta_i"description] \arrow[ddd,"\eta_i"description] & & & \sD_i  \arrow[dl,"\eta_{j}"description] \arrow[ddd,"\eta_i" description] \\
			\sD_j  \arrow[rrr, crossing over,"\eta_i"description] \arrow[ddd,"\eta_i"description] & & & \sD_j\circ \sD_i \\
			& & & &\\
			& \sD_i  \arrow[dl,"\eta_j"description] \arrow[rrr,"\eta_i"description] & & & \sD_i\circ \sD_i  \arrow[dl,"\eta_j"description] \\
			\sD_i\circ \sD_j  \arrow[rrr,"\eta_i"description] & & & \sD_i\circ \sD_j\circ \sD_i  \arrow[from=uuu, crossing over,"\eta_i"description]
		\end{tikzcd}
		\]
		Denote by $\widetilde{\sD_i}\subset \sD_i\oplus \sD_i$ the kernel of the map $\sD_i\oplus \sD_i\rightarrow \sD_i\circ\sD_i$ which  is clearly isomorphic to $\sD_i$. Denote by $\sD_i\sv[1]$ the direct summand of $\sD_i\circ\sD_i$ defined in~\eqref{eq::DD=D+D}.
		We include these summands into a subcomplex called $C_{{iji}}$:
		\[
		\begin{tikzcd}
			C_{{iji}} \arrow[d,"\sim",hookrightarrow] & := 
			\Id [3] \oplus 
			\left(\begin{array}{c}
				\sD_j\oplus \\
				\widetilde{\sD_i}
			\end{array}\right)[2] 
			\oplus 
			\left(
			\begin{array}{c}
				\sD_i\circ\sD_j\oplus \\
				\sD_j\circ\sD_i \oplus \\
				\sD_i\sv[1]
			\end{array}
			\right)[1] \oplus \sD_i\circ\sD_j\circ\sD_i \\
			\sT_i\circ\sT_j\circ\sT_i & := 
			\Id [3] \oplus 
			\left(\begin{array}{c}
				\sD_j\oplus \\
				{\sD_i}\oplus \\
				\sD_i
			\end{array}\right)[2] 
			\oplus 
			\left(
			\begin{array}{c}
				\sD_i\circ\sD_j\oplus \\
				\sD_j\circ\sD_i \oplus \\
				\sD_i\circ\sD_i[1]
			\end{array}
			\right)[1] \oplus \sD_i\circ\sD_j\circ\sD_i
		\end{tikzcd}
		\]
		We claim, that it is not difficult to show that the subspace $C_{{iji}}$ is indeed a subcomplex. Indeed, notice that the restriction of the counit map $\eta_j$ onto the kernel $\widetilde{\sD_i}$ coincides with $\eta_j$. Moreover, all elements in the direct summand $\sD_i\subset \sD_i\circ\sD_i$ belong to the image of the differential and hence the restriction of the {total} differential on it is zero by the commutativity of the diagram, and the sign convention of the double complex.
		The cone of the morphism $\zeta_i:C_{{iji}}\rightarrow \sT_i\circ\sT_j\circ\sT_i$ is isomorphic to the acyclic complex $\sD_i\rightarrow \sD_i$ and, therefore, $\zeta_i$ is a quasiisomorphism.
		
		We claim that there exists a straightforward map of complexes: $\xi_i: C_{{iji}} \to \BGG_{{ij}}$ which is the componentwise maps of induction functors. The cone of this morphism is isomorphic to the distinguished triangle~\eqref{eq::triangle::DD}.
		Consequently, $\xi_i$ is an isomorphism.
		
		By interchanging $i$ and $j$ we get another equivalences:
		$$\sT_{j}\circ \sT_i\circ \sT_{j}\stackrel{\zeta_j}{\longleftarrow} C_{{jij}} \stackrel{\xi_j}{\longrightarrow} \BGG_{{ji}}=\BGG_{{ij}}.$$
		Therefore, complexes $\sT_{j}\circ \sT_i\circ \sT_{j}$ and $\sT_i\circ \sT_{j}\circ \sT_i$ are isomorphic in derived category.
	\end{proof}	
	
	\subsubsection{Braid relations: $B_2$ case}	
	\label{sec::Braid::B2}
	
	In this section, we omit the superposition sign "$\circ$" to make formulas with functors more compact.
	
	Suppose that the root subsystem generated by ${i,j}\in \tI$ is $B_2$.
	In particular, we assume $\langle\alpha_i,\alpha_j\rangle =-2$.
	The goal of this subsection is to explain the corresponding isomorphism of triangulated functors:
	\begin{equation}
		\label{eq::B2::rel}
		\sT_i\sT_j\sT_i\sT_j \simeq \sT_j\sT_i\sT_j\sT_i.
	\end{equation}
	
	The strategy repeats the one suggested for $A_2$ case:\\
	First, we use Decomposition~\eqref{eq::DD=D+D} and replace $\sD_i\to \sD_i\sD_i$ by the subcomplex $0\to \sD_i\sv[1]$ (that is quasi-isomorphic to the original).
	Second, we use the properties of the spherical functors and make several more cancellations of this kind.
	Third, we use the universal properties of derived functors and relate the remaining complex with the BGG-complex whose zero term is the derived induction on the parabolic subalgebra $\fp_{{ij}}$. Finally, we notice that the BGG complex is invariant under the switch of $\alpha_i$ and $\alpha_j$ that implies the equivalence~\eqref{eq::B2::rel}.

	We start with the description of the key ingredient of the BGG complex the
	functor $\sD_{{ij}} =\Res_{{ij}}\circ\LInd_{{ij}}$ that is the derived superinduction to the parabolic subalgebra $\fp_{{ij}}[\xi]$:
	\begin{lem}
		There exists a distinguished triangle of endofunctors:
		\begin{equation}
			\label{dist-triangle-ZZZ}
			\left(
			\begin{array}{c}
				\sD_i\sD_j\sv[1] \oplus \\
				\sD_i\sD_j\sv[1]
			\end{array}\right)
			\stackrel{\eta\circ \sD_j+\sD_i\circ\eta}\longrightarrow \sD_i\sD_j\sD_i\sD_j
			\stackrel{\mathbf{L}\pi_{{ij}}}{\longrightarrow} \sD_{{ij}}.
		\end{equation}
	\end{lem}
	\begin{proof}
		First, we notice that the $0$'th cohomology of $\sD_{{ij}}$ coincides with $0$'th cohomology of $\sD_i\sD_j\sD_i\sD_j$. We illustrated this fact for $\Bbbk_{\lambda}$ with sufficiently dominant $\lambda$ in Section~\ref{sec::Demazure::modules} and the standard trick with the tensor product with integrable modules explains the coincidence for all $\lambda$. 
		Let us denote the corresponding isomorphism $\pi_{{ij}}$ and the corresponding derived functor  $\mathbf{L}\pi_{{ij}}:\sD_i\sD_j\sD_i\sD_j \rightarrow \sD_{{ij}}$ exists thanks to the universal property of the left derived functor $\sD_{{ij}}$.
		
		Let us compute the cohomology of the image of the triangle~\eqref{dist-triangle-ZZZ} for
		$\Bbbk_{\lambda}$ with sufficiently dominant $\lambda$. 
		As we mentioned the $0$-th cohomology vanishes thanks to Theorem~\ref{thm::Ind::Demazure}:
		$$
		H^{0}(\sD_i\sD_j\sD_i\sD_j(\Bbbk_{\lambda})) = H^{0}(\sD_{{ij}}(\Bbbk_{\lambda})) = \Ind_{{ij}}(\Bbbk_{\lambda}).
		$$
		Moreover,  for $\lambda$ sufficiently dominant $\sD_j(\Bbbk_{\lambda})$ has no higher cohomology and all weights in the $\msl_2^{\beta}$-integrable module $\Ind_j(\Bbbk_{\lambda})$ remain to be $\alpha$-dominant.  Therefore, $\sD_i\sD_j(\Bbbk_{\lambda})$ also does not have nonzero cohomology.
		As in the $A_2$-case the complex $\sD_{{ij}}(\Bbbk_{\lambda})$ also do not have higher cohomology for sufficiently dominant $\lambda$.
		Thus, the cohomology of all complexes in the triangle~\eqref{dist-triangle-ZZZ} differs from zero only for $H^{0}$ and $H^{-1}$.
		The coincidence for $H^{-1}$ follows from the coincidence of characters which was computed by I.Cherednik in~\cite{Ch}.
	\end{proof}	
	
	Next, let us rewrite the complex $\sT_i\sT_j\sT_i\sT_j$ in terms of $\sD$'s:
	\begin{multline*}
		(\Id\stackrel{\eta_i}{\longrightarrow} \sD_i)\circ(\Id\stackrel{\eta_j}{\longrightarrow} \sD_j) \circ
		(\Id\stackrel{\eta_i}{\longrightarrow} \sD_i)\circ(\Id\stackrel{\eta_j}{\longrightarrow} \sD_j) =
		\\
		=
		\Id \rightarrow \left(
		\begin{array}{c}\sD_i\oplus \\ \sD_j\oplus \\ \sD_i\oplus \\ \sD_j
		\end{array}
		\right)
		\rightarrow \left(
		\begin{array}{c}
			\sD_i\sD_j \oplus \\
			\sD_i\sD_i \oplus \\
			\sD_i\sD_j \oplus \\
			\sD_j\sD_i\oplus \\
			\sD_j\sD_j\oplus \\
			\sD_i \sD_j
		\end{array}
		\right)
		\rightarrow
		\left(
		\begin{array}{c}
			\sD_i\sD_j\sD_i \oplus \\
			\sD_i\sD_j\sD_j \oplus \\
			\sD_i\sD_i\sD_j \oplus \\
			\sD_j\sD_i\sD_j
		\end{array}
		\right)
		\rightarrow
		\sD_i\sD_j\sD_i\sD_j.
	\end{multline*}
	
	Consider the subcomplex $C_1$ where in the second column we replace two copies of $\sD_i$ by the kernel of the map $\sD_i\oplus\sD_i \to \sD_i\sD_i$ and the component $\sD_i\sD_i$ is replaced by the direct summand $\sD_i\sv[1]$. The same cancellation associated with the simple root $\beta$ leads to the quasi isomorphic subcomplex $C_2$.
	The next subcomplex $C_3$ corresponds to
	the replacement of the counit morphisms
	$$\sD_i\sD_j\oplus\sD_i\sD_j \oplus \sD_i\sD_j \longrightarrow \sD_i\sD_i\sD_j \oplus\sD_i\sD_j\sD_j$$
	by its kernel and cokernel which is isomorphic to
	$$
	\sD_i\sD_j \stackrel{0}{\longrightarrow} \sD_i\sD_j\sv[1] \oplus \sD_i\sD_j\sv[1]
	$$
	and the underlying space of the quasi isomorphic subcomplex $C_3$  looks as follows:
	$$
	\Id \rightarrow \left(
	\begin{array}{c}\sD_i\oplus \sD_j
	\end{array}
	\right)
	\rightarrow \left(
	\begin{array}{c}
		\sD_i\sv[1] \oplus \sD_j\sv[1] \oplus\\
		\sD_j\sD_i\oplus
		\sD_i \sD_j
	\end{array}
	\right)
	\rightarrow
	\left(
	\begin{array}{c}
		\sD_i\sD_j\sD_i \oplus \\
		\sD_i\sD_j\sv[1]\oplus \\
		\sD_i\sD_j\sv[1]\oplus \\
		\sD_j\sD_i\sD_j
	\end{array}
	\right)
	\rightarrow
	\sD_i\sD_j\sD_i\sD_j.
	$$
	Let us collect certain parts together due to the morphisms described in~\eqref{eq::counit::braid}:
	\begin{multline}
		\label{qis-complex-typeB2}
		\Id \rightarrow \left(
		\begin{array}{c}\sD_i\oplus \sD_j
		\end{array}
		\right)
		\rightarrow \left(
		\begin{array}{c}
			\sD_j\sD_i\oplus
			\sD_i \sD_j
		\end{array}
		\right)
		\rightarrow
		\\
		\rightarrow
		\left(
		\begin{array}{c}
			\cone(\sD_i\sv[1]\longrightarrow\sD_i\sD_j\sD_i) \oplus \\
			\cone(\sD_j\sv[1]\rightarrow \sD_j\sD_i\sD_j)
		\end{array}
		\right)
		\rightarrow
		\cone
		\left(
		\begin{array}{c}
			\sD_i\sD_j\sv[1] \oplus \\
			\sD_i\sD_j\sv[1]
		\end{array}
		\stackrel{\eta\circ \sD_j+\sD_i\circ\eta}\longrightarrow \sD_i\sD_j\sD_i\sD_j\right).
	\end{multline}
	By construction, this object also represents $C_3$ up to quasi-isomorphism.
	By replacing the cone in the rightmost part of \eqref{qis-complex-typeB2} using \eqref{dist-triangle-ZZZ}, we obtain a complex that we call $\BGG_{{ij}}$ with the following presentation:
	\begin{multline}
		\label{eq::BGG:B2}
		\Id \rightarrow \left(
		\begin{array}{c}\sD_i\oplus \sD_j
		\end{array}
		\right)
		\rightarrow \left(
		\begin{array}{c}
			\sD_j\sD_i\oplus
			\sD_i \sD_j
		\end{array}
		\right)
		\rightarrow
		\left(
		\begin{array}{c}
			\cone(\sD_i\sv[1]\longrightarrow\sD_i\sD_j\sD_i) \oplus \\
			\cone(\sD_j\sv[1]\rightarrow \sD_j\sD_i\sD_j)
		\end{array}
		\right)
		\rightarrow
		\sD_{{ij}}.
	\end{multline}
	By construction, $\BGG_{{ij}}$ is quasi-isomorphic to $C_3$. The Complex~\eqref{eq::BGG:B2} is invariant under the switch $\alpha_i\leftrightarrow \alpha_j$ that implies the desired equivalence~\eqref{eq::B2::rel} given by the following collection of quasiisomorphisms:
	$$
	\sT_i\sT_j\sT_i\sT_j\longleftarrow C_1^{i} \longleftarrow C_3^{i} \longrightarrow \BGG_{{ij}} \longleftarrow C_3^{j} \longrightarrow C_1^{j} \longrightarrow \sT_j\sT_i\sT_j\sT_i.
	$$
	
	\subsubsection{$G_2$ case.}
	We believe that the analogous proof can be worked out if the rank $2$ subsystem generated by $\alpha_i$ and $\alpha_j$ is $G_2$. However, the proof seems to be extremely technical and we want to save time and skip it.

	\section{Categorification of the symmetrization operator $\hP_{\tJ}$ via $\LInd_{\tJ}$ }
	\label{sec::symm::categorify}
	
	This section is devoted to the more detailed description of the induction $\Ind_{\tJ}$ and coinduction functors $\Coind_{\tJ}$ associated to an arbitrary subset $\tJ\subset \tI$ of the indexing set of the set of simple roots.
	Recall that they were defined in~\S\ref{sec::Demazure::functor::def} as the left and right adjoints to the restriction functor $\Res_{\tJ}:\cO(\fp_{\tJ}[\xi])\to \cO(\fb[\xi])$. Where $\fp_{\tJ}$ is a parabolic subalgebra associated with $\tJ$. Our main goal is to explain that the corresponding derived functor $\LInd_{\tJ}$ categorifies the Cherednik symmetrization functor $\hP_{\tJ}$.
	
	\subsection{Derived induction $\LInd_{\tJ}$ and coinduction $\RCoind_{\tJ}$ and duality between them}
	\label{sec::Duality}
	One can repeat the arguments of Section~\ref{sec::Ind} to explain the existence of the corresponding derived adjoint functors:
	\begin{prop}
		The category $\cO(\fb[\xi])$ has enough $\Ind_{\tJ}$-acyclic and $\Coind_{\tJ}$-acyclic modules, what shows the existence of the following adjunction:
		\[
		\begin{tikzcd}
			\bD^{b}(\cO({\fb}[\xi]))
			\arrow[rr, shift left = 3, "\LInd_{\tJ}"]
			\arrow[rr, shift right = 3, "\RCoind_{\tJ}"', "\perp" near end]
			&&
			\bD^{b}(\cO(\fp_{\tJ}[\xi]))
			\arrow[ll, "\Res_{\tJ}" description,"\perp"' near start]
		\end{tikzcd}
		\]
		Moreover, each module $M\in\cO(\fb[\xi])$ admits an $\Ind_{\tJ}$-acyclic resolution of length $\dim\fn_{\tJ}$.
	\end{prop}
	\begin{proof}
		Let us sketch the proof without details because it is completely analogous to the one suggested in~\S\ref{sec::Ind}.
		Let $\fb_{\tJ}:=\fh\oplus\fn_{\tJ}^{+}$ and let us denote by  $M_{\lambda}:=U(\fb_{\tJ})\otimes_{U(\fh)}\Bbbk_{\lambda}$ the module which is isomorphic to the Verma module if we flip $\fn^{+}$ and $\fn^{-}$. It is easy to define the action of $\fb[\xi]$ on $M_{\lambda}$ while posing the action of $\xi\fb$ and $\{e_j\colon j\notin \tJ\}$ by zero on the generator.
		This module does not belong to the category $\cO(\fb[\xi])$, however, if we truncate it by elements of sufficiently large degree in the PBW-filtration we will get a finite-dimensional $\Ind_{\tJ}$-acyclic module. 
		Next let us consider the $BGG_{\tJ}$-resolution of $\Bbbk_{0}$ and keep the action of the Borel subalgebra $\fb[\xi]$ (with the trivial action of odd elements):
		\begin{equation}
			\label{eq::BGG::J}
			\begin{tikzcd}
				M_{2\rho_{\tJ}} \ar[r] & \ldots \ar[r] & 
				\oplus_{\omega\in W_{\tJ}\colon l(\omega)=r} M_{-w.0} \ar[r] & \ldots\ar[r] & \oplus_{j\in\tJ} M_{-s_j.0}\ar[r]  & \Bbbk_0
			\end{tikzcd}    
		\end{equation}
		Here $\omega.0$ denotes the dot-action $\omega.\lambda = \omega(\lambda+\rho_{\tJ})-\rho_{\tJ}$. Moreover, $2\rho_{\tJ}=\omega_0^{\tJ}.0$, where $\omega_0^{\tJ}$ is the longest element in the Weyl group $W_{\tJ}$ whose length is equal to $\dim \fn_{\tJ}$. If we tensor the BGG-resolution~\eqref{eq::BGG::J} by a finite-dimensional $\fb[\xi]$-module $K$ we will obtain an $\LInd_{\tJ}$-acyclic resolution of $K$.
	\end{proof}

	Let us split (co)induction $\LInd_{\tJ}$ into the composition of the even $\LInd_{\tJ}^{\bar{0}}$ and the odd $\Ind_{\tJ}^{\bar{1}}$ (co)inductions while considering an intermediate Lie subalgebra $\fb[\xi] \subset \fb \oplus \fp_{\tJ}\xi \subset \fp_{\tJ}[\xi]$.
	\begin{prop}
		\label{prp::Serre}	
		The derived inductions and coinductions are equivalent up to appropriate shifts:
		\begin{gather}
			\label{eq::Serre::odd}
			\Coind_{\tJ}^{\bar{1}}\simeq  \Ind_{\tJ}^{\bar{1}} \circ \sX^{2\rho_{\tJ}} \sv^{-\dim \fn_{\tJ}} \simeq \sX^{2\rho_{{\tJ}}} \sv^{-\dim \fn_{\tJ}}\circ \Ind_{\tJ}^{\bar{1}} \\
			\label{eq::Serre::even}
			\RCoind_{\tJ}^{\bar{0}} \simeq \LInd_{\tJ}^{\bar{0}} \circ \sX^{-2\rho_{\tJ}} [-{\dim \fn_{\tJ}}]
		\end{gather}
		What follows the duality between the superinduction and coinduction:
		\begin{equation}
			\label{eq::Serre}
			\begin{array}{c}
				\RCoind_{\tJ} = \RCoind_{\tJ}^{\bar{0}}\circ \Coind_{\tJ}^{\bar{1}} \simeq 
				\left(  \LInd_{\tJ}^{\bar{0}}\circ \sX^{-2\rho_{\tJ}}[-\dim n_{\tJ}]\right) \circ\left( \sX^{2\rho_{\tJ}}\sv^{-\dim n_{\tJ}} \circ \Ind_{\tJ}^{\bar{1}} \right) \simeq \\
				\simeq \LInd_{\tJ}^{\bar{0}}\circ\Ind_{\tJ}^{\bar{1}} \circ \sv^{-\dim \fn_{\tJ}} [-{\dim \fn_{\tJ}}] \simeq
				\LInd_{\tJ} \circ \sv^{-\dim \fn_{\tJ}} [-{\dim \fn_{\tJ}}]
			\end{array}  
		\end{equation}
	\end{prop}	
	\begin{proof}
		The odd induction $\Ind_{\tJ}^{\bar{1}}$ is an exact functor on the level of abelian categories given by the tensor product with the symmetric algebra $\Lambda^{\udot}(\fn_{-}\xi)$. Respectively, the odd coinduction $\Coind_{\tJ}^{\bar{1}}$ is given by the tensor product with the dual space $\Hom(\Lambda^{\udot}(\fn_{\tJ}^{-}\xi),\Bbbk)$. The Cartan pairing between $\fn_{\tJ}^{-}$ and $\fn_{\tJ}^{+}$ leads the isomorphism
		\[ \Hom(\Lambda^{\udot}(\fn_{\tJ}^{-}),\Bbbk) \simeq \Lambda^{\udot}(\fn_{\tJ}^{+}) \simeq \sX^{2\rho_{{\tJ}}} \Lambda^{\udot}(\fn_{\tJ}^{-}),\]  what implies the Isomorphism~\eqref{eq::Serre::odd}.

		The existence of the even derived (co)induction functors $\LInd_{\tJ}^{\bar{0}}$ (resp. $\RCoind_{\tJ}^{\bar{0}}$) is a straightforward generalization of the arguments considered for the case $\fg=\msl_2$ in~\S\ref{sec::sl_2}.
		The Isomorphism~\eqref{eq::Serre::even} is the Serre duality isomorphism for the flag variety $\mathsf{Fl}:=\bP_{\tJ}/\bB$. $-2\rho_{\tJ}$ is the weight of the top exterior power of $\fn_{\tJ}^{-}$ -- the {co}tangent bundle $\mathcal{T}_{\mathsf{Fl}}$. The dimension $\dim \fn_{\tJ}$ is equal to the dimension of the flag manifold $\mathsf{Fl}$.
		See e.g.\cite{Yantzen} for the detailed proof of the Serre duality in this case and its comparison with derived (co)induction functors.
	\end{proof}

	\subsection{Some properties of $\LInd_{\tJ}$}
	\label{sec::Ind_J::property}
	Let us prove several properties of the induction functor $\LInd_{\tJ}$ that motivates the categorification theorem described in the subsequent subsection~\ref{sec::sym::categorify}.
	
	\begin{lem}
		\label{lem::Ind_I}	
		For all $i\in\tJ$ we have the following isomorphism of derived functors:
		\begin{gather}
			\label{eq::Ind_I:Res}
			\sD_{i} \circ \Res_{\tJ} \simeq \Res_{\tJ} \oplus \Res_{\tJ} \circ \sv[1] ; \\
			\label{eq::Ind_I:Ind}
			\LInd_{\tJ} \circ \sD_{i} \simeq \LInd_{\tJ} \oplus \LInd_{\tJ}\circ \sv[1].
		\end{gather}
	\end{lem}
	\begin{proof}
		Denote by $\Res_{\tJ}^i$ the restriction functor from the parabolic subalgebra  $\fp_{\tJ}[\xi]$ to the minimal parabolic subalgebra $\fp_i[\xi]$ and by $\LInd_{\tJ}^{i}$, $\RCoind_{\tJ}^{i}$ the corresponding derived induction and coinduction functors. The big restriction, (co)induction is the composition of the smaller ones:
		\[\Res_{\tJ}=\Res_i \circ \Res_{\tJ}^i, \quad \LInd_{\tJ} = \LInd_{\tJ}^{i}\circ \LInd_{i}, \quad 
		\RCoind_{\tJ}^{i}\circ\RCoind_{i}.\]
		Consequently, we have
		\[
		\sD_i\circ\Res_{\tJ} = (\Res_{i}\circ\LInd_{i})\circ(\Res_{i}\circ\Res_{\tJ}^i) \stackrel{\eqref{eq::Res::spherical}}{=} (\Res_{i}\oplus\Res_{i}\sv[1])\circ\Res_{\tJ}^{i} = \Res_{\tJ}\oplus\Res_{\tJ}\sv[1].  
		\]	
		and the similar isomorphism for the coinduction functors:	
		\begin{multline*}
			\RCoind_{\tJ}\circ \sD_{i} = \RCoind_{\tJ}^{i} \circ\RCoind_{i} \circ \Res_{i}\circ \LInd_{i} 
			\stackrel{\eqref{eq::CoInd::spherical}}{\simeq}
			\RCoind_{\tJ}^{i}  \circ (\RCoind_{i}\oplus \LInd_{i})  \simeq \\
			\stackrel{\eqref{eq::Serre}}{\simeq}
			\RCoind_{\tJ}^{i} \circ (\RCoind_{i} \oplus \RCoind_{i} \sv[1]) = \RCoind_{\tJ}\oplus\RCoind_{\tJ} \sv[1].
		\end{multline*}
		Thanks to the Serre duality isomorphism~\eqref{eq::Serre} we know that the induction and the coinduction functors differ by an appropriate shift, which implies Isomorphism~\eqref{eq::Ind_I:Ind}.
	\end{proof}	
	
	\begin{prop}
		\label{prp::tens::ind}	
		The (derived) induction $(\mathbf{L})\Ind_{\tJ}$ commutes with the tensor product with finite-dimensional $\fp_{\tJ}$-module. That is for all $M\in\cO(\fp_{\tJ}[\xi])$ we have a functorial in $M$ isomorphisms of functors:
		\begin{equation}
			\label{eq::tens::Ind}	
			(\mathbf{L})\psi_{M}:	(\mathbf{L})\Ind_{\tJ}(\Res_{\tJ}(M)\otimes -) \stackrel{\simeq}{\longrightarrow} M\otimes (\mathbf{L})\Ind_{\tJ}(-).
		\end{equation}	
	\end{prop}
	\begin{proof}
		The proof repeats the one suggested for the induction $\Ind_i$ (see Theorem~\ref{cor::H::Ind}).
		Note that the restriction functor commutes with the tensor product:
		\begin{equation}
			\label{eq::Res::Tensor}
			\Res_{\tJ}(M\otimes -) \simeq \Res_{\tJ}(M)\otimes\Res_{\tJ}(-).
		\end{equation}
		Since $M$ is finite-dimensional its linear dual $\fp_{\tJ}$-module $M^*$ is also a finite-dimensional and belongs to $\cO(\fp_{\tJ}[\xi])$.
		The (derived) left adjoint to the exact functor $\Res_{\tJ}(M\otimes -)$ is equal to $M^{*}\otimes (\mathbf{L})\Ind_{\tJ}(-)$ and the (derived) left adjoint to the exact functor $\Res_{\tJ}(M)\otimes\Res_{\tJ}(-)$ is isomorphic to $(\mathbf{L})\Ind_{\tJ}(\Res_{\tJ}(M^*)\otimes -)$. Replacing $M$ with $M^*$ we see that left and right hand sides of~\eqref{eq::tens::Ind} are left adjoint functors to the isomorphic ones and consequently they are also isomorphic.
	\end{proof}

	\subsection{Particular computation of $\LInd_{\tJ}$}
	In this subsection we will do some particular computations of the even induction $\LInd_{\tJ}^{\bar{0}}$ and the full induction $\LInd_{\tJ}$. (It is always clear how to apply the odd induction since it is an exact functor that does not have the higher derived components.) 
	
	First, recall that the additive functor $\Ind_{\tJ}$ gives the maximal $\fp_{\tJ}$-integrable quotient of the induced module. We call $\lambda \in \sfP$ be $\tJ$-dominant (resp. strictly $\tJ$-dominant) if we have $\langle\lambda,\alpha_j^{\vee}\rangle \geq 0$ (resp. $\langle\lambda,\alpha_j^{\vee}\rangle > 0$) for each $j \in \tJ$. In particular we have:
	$$
	\Ind_{\tJ}(\Bbbk_{\lambda}) = \begin{cases}
		L^{\tJ}(\lambda), \text{ if } \lambda \text{ is $\tJ$-dominant}, \\
		0, \text{ otherwise.}
	\end{cases}
	$$
	Here $L^{\tJ}(\lambda)$ denotes the irreducible finite-dimensional $\fp_{\tJ}$-module with the highest ($\tJ$-dominant) weight $\lambda$.
	In particular, any module whose weight support is $\tJ$-dominant is $\Ind_{\tJ}^{\bar{0}}$-acyclic. 
	Consequently, if $(\lambda-2\rho_{\tJ})$ is a strictly $\tJ$-dominant weight, then we have no higher derived factors:
	$$\LInd_{\tJ}(\Bbbk_{\lambda}) = \LInd_{\tJ}^{\bar{0}}(\Ind_{\tJ}^{\bar{1}}(\Bbbk_{\lambda})) = \LInd_{\tJ}(U(\fn_{\tJ}^{-}\xi)\otimes\Bbbk_{\lambda}) = \Ind_{\tJ}(U(\fn_{\tJ}^{-}\xi)\otimes\Bbbk_{\lambda}).
	$$
	We can easily compute its character:
	\begin{multline}
		\label{eq::PJ::k_lambda}	
		\gch(\LInd_{\tJ}(\Bbbk_{\lambda})) = \gch(\LInd_{\tJ}^{\bar{0}}(U(\fn_{\tJ}^{-}\xi)\otimes\Bbbk_\lambda)) =  
		\\ =
		\sum_{S\subset\Delta_{J}^{+}} (-t)^{|S|} \gch(\LInd_{\tJ}^{\bar{0}}(\Bbbk_{\lambda-\sum_{\alpha\in S}\alpha})) = 
		\sum_{S\subset\Delta_{J}^{+}} (-t)^{|S|} s^{\tJ}_{\lambda-\sum_{\alpha\in S}\alpha},   
	\end{multline}
	where $s_{\mu_+}^{\tJ}$ denotes the Schur functions -- the character of the unique finite-dimensional $\fp_{\tJ}$-irreducible module with the $\tJ$-dominant highest weight $\mu_+$. In particular, we see that all weights $\lambda-\sum_{\alpha\in S}\alpha\geq \lambda-2\rho_{\tJ}$ are (strictly) $\tJ$-dominant whenever $(\lambda-2\rho_{\tJ})$ is (strictly) $\tJ$-dominant.
	Moreover, since the highest weight in an irreducible representation is the greatest element with respect to the Cherednik partial order we conclude that whenever $\lambda-2\rho_{\tJ}$ is $\tJ$-dominant we have:
	$$\gch(\LInd_{\tJ}(\Bbbk_{\lambda})) = \hX^{\lambda} + \sum_{\mu\prec \lambda} c_{\lambda,\mu}\hX^{\mu}, \text{ for some }c_{\lambda,\mu}\in\bZ_{t}.$$
	
	Let us now consider the case $\lambda=0$. 
	\begin{lem}
		\label{lm::Ind::PJ::small}
		\begin{equation}
			\label{eq::Ind::J::small}
			\LInd_{\tJ}^{\bar{0}}(\Lambda^k(\fn^{-}_{\tJ})) = (\Bbbk_{0}[k])^{\# \{w\in W_{\tJ}\colon l(w) =k\}}.
		\end{equation}
	\end{lem}
	\begin{proof}
		By construction, $\LInd_{\tJ}^{\bar{0}}(\Lambda^k(\fn^{-}_{\tJ}))$ restricted to $\g_{\tJ}$ is the same as $\LInd^{\bar{0}}(\Lambda^k( \fn_{-} ))$ for $\tI = \tJ$. Since elements of $\fh$ that centralizes $\fn_{\tJ}^-$ act trivially on $\LInd_{\tJ}^{\bar{0}}(\Lambda^k(\fn_{-}^{\tJ}))$, the $\fh$-weight consideration reduces the assertion to the case of $\tI = \tJ$. Henceforth, we suppress scripts $\tJ$ in what follows. For each $\la \in \sfP_+$, we have
		$$\mathrm{Hom}_{(\g,\fh)} ( \LInd^{\bar{0}}(M), L ( \la ) ) \equiv \mathbf R^{\bullet}\mathrm{Hom}_{(\g,\fh)} ( \LInd^{\bar{0}}(M), L ( \la ) ) \cong \mathbf R^{\bullet}\mathrm{Hom}_{(\fb,\fh)} ( M, \Res \, L ( \la ) ),$$
		where $(\g,\fh)$ denotes the category of $\fh$-semisimple integrable $\g$-modules (that forms a semisimple category) and $(\fb,\fh)$ denotes the category of $\fh$-semisimple $\fb$-modules with locally finite $\fn_{+}$-action. In view of the semisimplicity of $\fh$-action, we can further rewrite this as:
		\begin{equation}\label{eqn:fcompare}
			\mathbf R^{\bullet}\mathrm{Hom}_{(\fb,\fh)} ( M, \Res \, L ( \la ) ) \cong \mathrm{Hom}_{\fh} ( \mathbf R^{\bullet}\mathrm{Hom}_{\fn_+} ( M, L ( \la ) ), \bC ),
		\end{equation}
		where the both of $M$ and $L (\la)$ in the RHS of (\ref{eqn:fcompare}) are restricted to $\fn_{+}$-module equipped with $\fh$-gradings. Now we set $M = \Lambda^k (\fn_{-})$, and replace $L(\la)$ with the dual of the BGG resolution (\ref{eq::BGG::J}) in which each term is injective as $U (\fn_+)$-modules. The $\fh$-weights of the cocyclic vector of indecomposable direct summand (as modules) of such injective resolution is of the shape $w . \la = w ( \la + \rho )- \rho$ for each $w \in W$ counted with multiplicities.
		
		The $\fh$-weights of $\Lambda^k (\fn_{-})$ is inside the convex cone spanned by $w.\rho$ by \cite[\S 5.9]{Kostant}. In view of the estimate \cite[(5.9.2)]{Kostant}, we have $w.\la = v.0$ for some $\la \in \sfP^+$, $w,v \in W$ if and only if $\la = 0$. Since $\mathrm{Stab}_W \rho = \{1\}$, we have necessarily $w = v$ in this case. Therefore, we conclude that
		$$\mathrm{Hom}_{\fh} ( \mathbf R^{\bullet}\mathrm{Hom}_{\fn} ( \Lambda^k (\fn_{-}), L ( \la ) ) \cong \begin{cases} 0 & (\la \neq 0)\\  \Bbbk^{\# \{w \mid l ( w ) = k \}}[k]& (\la = 0)\end{cases}.$$
		Note, that for $\lambda=0$ the $\fh$-homorphisms are concentrated in one homological degree $k$ and thus there is no differential. Thus, we conclude the result.
	\end{proof}
	
	\begin{prop}
		\label{prp::LInd::0}
		$\LInd_{\tJ}(\Bbbk_{0}) = \bigoplus_{w\in W_{\tJ}}\sv^{l(w)} \Bbbk_{0}[l(w)].$
	\end{prop}
	\begin{proof}
		As we mentioned earlier, the odd induction does not have the derived components:
		$$\LInd_{\tJ}^{\bar{1}}\Bbbk_{0} \cong U ( \g_{\tJ}\xi + \fb ) \otimes _{U (\fb [\xi])} \Bbbk_{0} \cong \Lambda^{\bullet} ( \fn_\tJ^-\xi) \cong\bigoplus_{k}\sv^{k}\Lambda^{k}(\fn_{\tJ}^{-}).$$
		Thus, we have
		\begin{multline*}
			\LInd_{\tJ}(\Bbbk_{0}) \cong \LInd_{\tJ}^{\bar{0}} (\LInd_{\tJ}^{\bar{1}}\Bbbk_{0}) \cong \LInd_{\tJ}^{\bar{0}} (\Lambda^{\bullet} ( \fn_\tJ^-\xi) ) \cong  \LInd_{\tJ}^{\bar{0}} (\bigoplus_{k}\sv^{k}\Lambda^{k} ( \fn_\tJ^-) )  \cong 
			\\
			\stackrel{Lemma~\ref{lm::Ind::PJ::small}}{\cong}
			\bigoplus_{k \in \bZ} \left(\sv^{k}L(0)[k]^{\#\{w|l(w)=k\}}\right) = 
			\bigoplus_{w\in W_{\tJ}}\sv^{l(w)}\Bbbk_{0}[l(w)].
		\end{multline*}
		Note that each homology of $\LInd_{\tJ}^{\bar{0}} (\bullet)$ has a structure of $\fg_{\tJ}$-module. In addition, the action of $\fn \xi$ preserves the homological degree, and hence we conclude the isomorphism as $\fb[\xi]$-modules as required.
	\end{proof}
	
	For each $\lambda\in\sfP$, we have assigned a convex subset $\sfP[\preceq\lambda]$ consisting of weights that are less or equal then $\lambda$ with respect to the Cherednik partial order (\S \ref{sec::order}). Let us denote by $\cO(\fb[\xi])_{\preceq\lambda}$ the Serre subcategory of $\cO(\fb[\xi])$ consisting of modules whose weight support belongs to $\sfP[\preceq\lambda]$.
	Respectively, by $\bD^{b}(\cO(\fb[\xi])_{[\preceq\lambda]}$ we denote the full subcategory of bounded complexes of $\fb[\xi]$-modules whose homology belongs to the subcategory $\cO(\fb[\xi])_{\preceq\lambda}$. 
	Recall that for all $\lambda\in\sfP$ the maximal element with respect to the Cherednik partial order $\preceq$ is denoted by
	$\lambda_-^{\tJ}$ and is the unique $\tJ$-antidominant weight in $W_{\tJ}\lambda$. We use the same notations for the Serre subcategory $\cO(\fp[\xi])_{\preceq\lambda_{-}}$ and the corresponding derived full subcategory $\bD^{b}(\cO(\fp_{\tJ}[\xi]))_{[\preceq \lambda_{-}^{\tJ}]}$ for the complexes with their homologies in $\cO(\fp[\xi])_{\preceq\lambda_{-}}$.
	\begin{thm}
		\label{thm::ind::convex}
		For all $\lambda\in\sfP$ we have the following restriction on the image of the derived subcategory
		$$
		\LInd_{\tJ}: \bD^{b}(\cO(\fb[\xi]))_{[\preceq \lambda]} \longrightarrow
		\bD^{b}(\cO(\fp_{\tJ}[\xi]))_{[\preceq \lambda_{-}^{\tJ}]}.
		$$
	\end{thm}
	\begin{proof}
		First, we notice that the (derived) induction $(\mathbf{L})\Ind_{\tJ}$ does not have an effect on the weights that are orthogonal to the subsystem $\Delta^{\tJ}$. 
		Thus, we can assume $\tJ=\tI$ without the loss of generality.
		Second, taking into account the fact that $\{\Bbbk_\lambda\}_{\lambda \in \sfP}$ exhausts the isomorphism classes of simple graded $\fb[\xi]$-modules (up to grading shifts), it suffices to show that the weight support of the cohomology of the complex $\LInd_{\tJ}(\Bbbk_{\lambda})$ ($\lambda\in\sfP$)
		belongs to $\sfP[\preceq{\lambda_{-}^{\tJ}}]$ by successive applications of the long exact sequences of cohomologies.
		
		Let us prove this statement by induction on $\lambda_{-}^{\tJ}$ with respect to the Cherednik partial order.
		The base of induction $\lambda=0$ was verified in Proposition~\ref{prp::LInd::0}. 
		For the induction step we use the properties of $\LInd_{\tJ}$ proved in the preceding subsection~\ref{sec::Ind_J::property}. Indeed, thanks to Isomorphism~\eqref{eq::Ind_I:Ind}, we have quasi-isomorphisms of complexes:
		$$
		\LInd_{\tJ}\sD_{i}(\Bbbk_{\lambda}) \simeq \LInd_{\tJ}(\Bbbk_{\lambda}) \oplus \sv\LInd_{\tJ}(\Bbbk_{\lambda})[1] \hskip 5mm  \Leftrightarrow \hskip 5mm \LInd_{\tJ}\sT_{i}(\Bbbk_{\lambda}) \simeq \sv\LInd_{\tJ}(\Bbbk_{\lambda})[1].
		$$
		Suppose that $\lambda\succ s_i\lambda$, then as we have seen from the direct inspection~\eqref{eq::T::k} the complex $\sT_i(\Bbbk_{\lambda})$ has exactly one vector of weight $s_i\lambda$ and all other $\fh$-eigenvectors belong to the interior of the interval connecting $\lambda$ and $s_i\lambda$. 
		Therefore, $\sT_i(\Bbbk_\lambda)$ is equivalent to $\Bbbk_{s_i\lambda}$ modulo the subcomplex that belongs to $\bD_{[\prec \lambda_{-}^\tJ]}$. Thus, thanks to the induction assumption we see that for any $\mu=\omega_{\mu} \lambda_{+}$ (with $\la_+$ is the $\tJ$-dominant element in $W_\tJ \la$ and $\omega_{\mu}$ of minimal length) the following complexes are equivalent modulo the subcategory $\bD_{[\prec \lambda_{-}^\tJ]}$:
		\begin{equation}
			\LInd_{\tJ}(\Bbbk_{\lambda_+^{\tJ}}) \sim \sv^{l(\omega_{\lambda})}\LInd_{\tJ}(\Bbbk_{\lambda})[l(\omega_{\lambda})].\label{eqn:Ind-equiv}
		\end{equation}
		Suppose, that $\LInd_{\tJ}(\Bbbk_{\lambda_+})$ does not belong to $\bD_{[\preceq{\lambda_{-}^\tJ}]}$. I.e. there exist the cohomology whose weight support does not belong to the subcategory $\cO_{[\preceq{\lambda_{-}^\tJ}]}$. 
		Since the complex $\LInd_{\tJ}(\Bbbk_{\lambda_+})$ is finite-dimensional, there exists the maximal $\xi$-degree $d$ 
		containing the cohomology of the wrong weight.
		Consequently, $\LInd_{\tJ}(\Bbbk_{\lambda})$ will have the maximal $\xi$-degree with wrong cohomology equal to $d+l(\omega_{\lambda})$.
		Let us consider the irreducible $\fp_{\tJ}$-representation $L^{\tJ}(\lambda_{+})$. 
		From one side, $L^{\tJ}(\lambda_{+})$ is equivalent to the extension of one-dimensional modules $\{\Bbbk_{\lambda}\colon \lambda\in W_{\tJ}\lambda_+\}$ modulo the subcategory $\cO_{[\prec\lambda_{-}^\tJ]}$.
		By~\eqref{eqn:Ind-equiv} and the induction hypothesis, for all $\lambda\in W_{\tJ}\lambda_{-}^{\tJ}$ different from $\lambda_{-}^\tJ$ all wrong cohomology of $\LInd_{\tJ}(\Bbbk_\lambda)$ has $\xi$-degree strictly less than $d+l(\omega_{\lambda_{-}^\tJ})$.
		In view of the fact that $\Bbbk_{\la_-}$ is a simple quotient of $L^\tJ( \la_+)$, a repeated application of the long exact sequences of cohomologies yields that $\LInd_{\tJ}(L^{\tJ}(\lambda_{+})) $ will have the same wrong cohomology in the top $\xi$-degree $d+l(\omega_{{\lambda_{-}^{\tJ}}})$ as the complex $\LInd_{\tJ}(\Bbbk_{{\lambda_{-}^{\tJ}}})$ has.
		On the other hand, we know that:
		\begin{multline*}
			\LInd_{\tJ}(\Res_{\tJ}(L^{\tJ}(\lambda_{+}^{\tJ}))) \simeq
			L^{\tJ}(\lambda_{+}^{\tJ})\otimes \LInd_{\tJ}(\Bbbk_{0}) 
			\stackrel{\ref{prp::LInd::0}}{\simeq}
			\\ \simeq 
			L^{\tJ}(\lambda_{+}^{\tJ}) \otimes \left(\oplus_{\omega\in W_{\tJ}} \sv^{l(\omega)}\Bbbk_{0}[l(\omega)] \right) \simeq \oplus_{\omega\in W_{\tJ}} \sv^{l(\omega)}L^{\tJ}(\lambda_+^{\tJ})[l(\omega)] \in\bD_{[\preceq \lambda_{-}]}
		\end{multline*}
		by the $\fp_\tJ$-action on $L^{\tJ}(\lambda_{+}^{\tJ})$ and Proposition~\ref{prp::tens::ind}. The comparison of these two gives rise to a contradiction on weight support of the cohomology of the complex $\LInd_{\tJ}(\Res_{\tJ}(L^{\tJ}(\lambda_{+}^{\tJ})))$, that forces $\LInd_{\tJ}(\Bbbk_{\lambda})\in \bD_{[\preceq\lambda_{-}^{\tJ}]}$ as required.
	\end{proof}

	\subsection{A categorification statement}
	\label{sec::sym::categorify}
	Finally, we are ready to outline all properties of the induction functor $\LInd_{\tJ}$ in Theorem~\ref{thm::PJ::categorify} which suggests a categorification of the Cherednik symmetrization operator $\hP_{\tJ}$ together with the properties~\ref{itm::P:prop2}--\ref{itm::P:prop6} that uniquelly itentify $\hP_{\tJ}$ as recalled in Section~\ref{sec::Sym::AHA}.
	\begin{thm}
		\label{thm::PJ::categorify}
		The endofunctor $\sD_{\tJ}:=\Res_{\tJ}\circ\LInd_{\tJ}$ of the derived category $\sD^{b}(\cO(\fb[\xi]))$ categorifies the Cherednik symmetrization functor $\hP_{\tJ}$. In detail this means, that on the level of characters, $\sD_{\tJ}$ act by $\hP_{\tJ}$ and, moreover, the following properties hold:
		\begin{enumerate}
			\item 
			\label{itm::PJ::categorify::TP=tP}		
			$\forall i\in\tJ$ we have isomorphisms of endofunctors:
			\begin{equation}
				\label{eq::DDJ=DJ+DJ}
				\sD_i\circ\sD_{\tJ} \simeq \sD_{\tJ}\oplus\sD_{\tJ}\sv[1] \simeq \sD_{\tJ}\circ\sD_{i} \quad \Rightarrow \quad
				\sT_i\circ\sD_{\tJ} \simeq \sD_{\tJ}\sv[1] \simeq \sD_{\tJ}\circ\sT_{i};
			\end{equation}
			\item 
			\label{itm::PJ::categorify::tensor}		
			The endofunctor $\sD_{\tJ}$ commutes with taking tensor products with finite-dimensional $\fp_\tJ$-modules:
			$$\sD_{\tJ}(\Res_{\tJ}(M)\otimes - ) \simeq \Res_{\tJ}(M)\otimes\sD_{\tJ}(-);$$
			\item 
			\label{itm::PJ::categorify::convex}  
			For all $\tJ$-dominant weight $\lambda_{+}^{\tJ}$ the endofunctor $\sD_{\tJ}$ preserves the triangulated subcategories $\bD_{[\preceq\lambda_{+}^{\tJ}]}$, moreover $\forall \lambda\in W_{\tJ}\lambda_{+}$ we have:
			$$
			\sD_{\tJ}: \bD^{b}(\cO(\fb[\xi]))_{[\preceq\lambda]} \rightarrow \bD^{b}(\cO(\fb[\xi]))_{[\preceq \lambda_{+}^{\tJ}]}.
			$$
		\end{enumerate}
	\end{thm}
	\begin{proof}
		Let us first explain the properties mentioned in Theorem~\ref{thm::PJ::categorify} and explain why they categorify properties known for the symmetrization operator $\hP_{J}$ mentioned on page~\pageref{itm::P:prop2}.
		We start from item~\eqref{itm::PJ::categorify::TP=tP} that categorifies property~\eqref{itm::P:prop2}.	
		Isomorphisms~\eqref{eq::DDJ=DJ+DJ} of endofunctors follows from appropriate Isomorphisms ~\eqref{eq::Ind_I:Ind} and~\eqref{eq::Ind_I:Res} of functors proved in Lemma~\ref{lem::Ind_I}. For example, we have:
		\begin{gather*}
			\sD_i\circ\sD_{\tJ} =	\sD_i  \circ \Res_{\tJ} \circ \LInd_{\tJ} \stackrel{\eqref{eq::Ind_I:Res}}{\simeq} (\Res_{\tJ} \oplus \Res_{\tJ}\sv[1])\circ \LInd_{\tJ} \simeq \sD_J\oplus\sD_J\sv[1];
			\\
			\sD_{\tJ}\circ\sD_i = \Res_{\tJ} \circ \LInd_{\tJ}\circ \sD_i \stackrel{\eqref{eq::Ind_I:Ind}}{\simeq}
			\Res_{\tJ} \circ(\LInd_{\tJ}\oplus\LInd_{\tJ}\sv[1]) \simeq \sD_J\oplus\sD_J\sv[1].
		\end{gather*}
		Since the characters of $\fp_{\tJ}$-integrable modules are $W_{\tJ}$-symmetric functions the categorical meaning of Property~\eqref{itm::P:prop3} is clear.
		While applying $\Res_{\tJ}$ to Isomorphism~\eqref{eq::tensor::ind} we see that 
		Item~\eqref{itm::PJ::categorify::tensor} follows from Proposition~\ref{prp::tens::ind} and this gives a categorification of Property~\eqref{itm::P:prop5}.
		Property~\eqref{itm::PJ::categorify::convex} categorifies Property~\eqref{itm::P:prop5} of the symmetrization operator and was verified in Theorem~\ref{thm::ind::convex}.
		Proposition~\ref{prp::LInd::0} shows that $\gch(\LInd_{\tJ}(\Bbbk_0))$ is equal to $\hP_{\tJ}(1)$.
		
		Finally, we end up that on the level of characters the endofunctor $\sD_{\tJ}$ satisfy all Properties~\eqref{itm::P:prop2}--\eqref{itm::P:prop::0} and, consequently, the character of $\sD_{\tJ}$ coincides with the symmetrization operator $\hP_{\tJ}$.
	\end{proof}

	\section{Categorification of DAHA and Macdonald polynomials}
	\label{sec::DAHA::Macdonald}

	\subsection{Lie superalgebra $\fI[\xi]$ and the category of graded modules}
	\label{sec::affine::O}
	Let us adopt the setup given in Section~\ref{sec::Setup::finite} to the affine case.
	Namely, we consider the Lie superalgebras $\fI[\xi]\subset \widehat{\fp}_{\tJ}[\xi]\subset \widehat{\fg}[\xi]$ that corresponds to the supercurrents from the Iwahori, parabolic assigned with $\tJ \subset \tI_\af$ and affine Lie algebras correspondingly.
	The untwisted affine Kac-Moody Lie algebra $\widehat{\fg}$ is the central extension of the Lie algebra of currents $\fg[z^{\pm}]$ and, therefore, all algebras we consider are bigraded with respect to the $(z,\xi)$-gradings.
	We set
	$$\bZ [\sfP]_+ := \bZ [\![q,t]\!] \otimes_{\bZ}  \bZ [\sfP]
	\text{ and }\bZ[\sfP] := \bZ [ q^{\pm 1}, t^{\pm 1}] \otimes_{\bZ [q,t]}\bZ [\sfP]_+.$$
	
	Let $\cO(\fI[\xi])$ be the category of  $(z,\xi)$-graded $\fI[\xi]$-modules with semi-simple $\fh$-action whose eigenvalues belongs to $\sfP$. Moreover we require that for each $M\in\cO(\fI[\xi])$ the subspace $M_{a,b}$
	of $z$-degree $a$ and $\xi$-degree $b$ is finite-dimensional  and the formal power sum
	\[\gch (M):=\sum_{a,b} q^a (-t)^b \ch \, M_{a,b} \]
	has to belong to $\bZ [\sfP]$ what means that $M_{a,b}=0$ for $a \ll 0$ or $b \ll 0$.
	
	We denote by $\cO(\widehat{\fp}_{\tJ}[\xi])$ the corresponding category of bigraded $\fh$-semisimple finitely generated modules for any parabolic subalgebra assigned with $\tJ\subset \tI_{\af}$.
	
	For a pair of formal series $f,g \in \bZ [\sfP]$ with respect to $q, t$ and $\hX^{\la}$ ($\la \in \sfP$) with integer coefficients, we say that $g \le f$ if and only if we have the corresponding inequality for all coefficients.
	
	To each weight $\lambda\in \sfP$ we assign a subcategory $\cO_{\prec \la}\subset \cO(\fI[\xi])$ consisting of graded modules whose weight support $\Psi(M)\subset \sfP[\prec\la]$ is bounded from above by $\la$ with respect to the Cherednik ordering $\prec$. Note that we do not take into account the $z$-grading while talking about weight for the representations of Iwahori or affine Lie algebras.
	
	\subsection{Derived category $\bD_{c}^{-}(\cO(\fI[\xi]))$}
	\label{sec::bD_c^-}
	The Macdonald polynomials are not polynomials in $q,t$ and to work out the corresponding categorification we have to change the derived category we are working with.
	
	\begin{definition}
		A bounded from above complex of bigraded $\fh$-semisimple $\fI[\xi]$-integrable modules $M^{\udot}$  is called \emph{convergent} if for each pair $(a,b)\in\bZ\times\bZ$ the $(z,\xi)$-graded component $M^{\udot}_{(a,b)}$ is a bounded complex of finite dimensional $\fb$-modules and for either $a \ll 0$ or $b \ll 0$ the graded component $M^{\udot}_{(a,b)}$ is empty.
		
		The triangulated category $\bD_{c}^{-}(\cO(\fI[\xi]))$  is the derived category associated with the category of convergent complexes. Similarly, we define the triangulated category $\bD_{c}^{-}(\cO(\widehat{\fp}_\tJ [\xi]))$ as the derived category of complex of $\widehat{\fp}_\tJ [\xi]$-modules with integrable $\g_\tJ$-action whose restriction to $\fI[\xi]$ gives rise to a convergent complex.
	\end{definition}
	
	We denote $\bD_{c}^{-}(\cO(\fI[\xi]))$ by $\bD_{c}^{-}$ for simplicity. Let $\bD_c^{\le 0} \subset \bD_{c}^{-}$ be the fullsubcategory consisting of objects presented by a complex concentrated in homological degree $\le 0$.
	Note that each module $M\in\cO(\fI[\xi])$ is convergent but might be infinite-dimensional. However, each complex $M^{\udot}\in \bD_{c}^{-}(\cO(\fI[\xi]))$ admits the well-defined Euler characteristics $\gch(M^{\udot})$ of graded supercharacters which belongs to $\bZ[\sfP]$. 
	
	We denote by $\sq$ and $\sv$ the endofunctors of $\cO$ (and various categories constructed from $\cO$) that raise the $z$-degree and $\xi$-degree by one, respectively.

	\begin{prop}
		The derived functors:  induction $\LInd_{i}$, restriction $\Res_{i}$ and coinduction $\RCoind_{i}$  are well-defined adjoint functors between the convergent category of complexes:
		\[
		\begin{tikzcd}
			\bD_{c}^{-}(\cO({\fI}[\xi]))
			\arrow[rr, shift left = 3, "\LInd_{i}"]
			\arrow[rr, shift right = 3, "\RCoind_{i}"']
			&&
			\bD_{c}^{-}(\cO(\widehat{\fp}_{i}[\xi])).
			\arrow[ll, "\Res_{i}" description]
		\end{tikzcd}
		\]
	\end{prop}
	\begin{proof}
		There exists a straightforward generalization of the arguments of Section~\ref{sec::Ind} that shows that each module $M\in\cO(\fI[\xi])$ admits a resolution of length at most $2$ by $\Ind_{i}$-acyclic (resp. $\Coind_{i}$-acyclic) modules from $\cO(\fI[\xi])$. 	
	\end{proof}
	Note that the functor $\Res_{i}$ commutes with limits and colimits.  Its right adjoint $\RCoind_{i}$ commutes with colimits, respectively the left adjoint $\LInd_{i}$ commutes with limits, whenever they exist.
	\begin{thm}
		\label{thm::DAHA::categorification}	
		The category  $\bD_c^{-}(\cO(\fI[\xi]))$ together with endofunctors  $\{\sT_{i}|i\in\tI_{\af}\}$, $\sX^{\lambda}$, the $\xi$-grading shift functor $\sv$, the $z$-grading shift functor $\sq$ and the group of diagram automorphism $\Omega:= \sfP/ \sfQ$  categorify the action of  the Double Affine Hecke Algebra $\DAHA$ on the polynomial representation in the sense of Theorem~\ref{thm::DAHA}.
	\end{thm}
	\begin{proof}
		The endofunctors $\sD_{i}$, $\sD_{i}'$ as well as the endofunctors $\sT_{i}$, $\sT_{i}'$ has at most two nonzero cohomological components and since they commute with direct sums/products it is enough to prove all the relations for the bounded complexes what was already done in the proof of Theorem~\ref{thm::DAHA}.
	\end{proof}

	Let $\cO [\prec \lambda]$ and $\cO [\preceq \lambda]$ denote the categories of graded $\fI[\xi]$-modules $M$ such that the weight support $\Psi ( M ) \subset \sfP [\prec \lambda]$ and $\Psi ( M ) \subset \sfP [\preceq \lambda]$ correspondingly.
	
	We denote by $\bD_{c}^{-}[\prec \la]$ and $\bD_{c}^{-}[\preceq \la]$ the full subcategories of $\bD_{c}^{-}$ that contain complexes of $\fI[\xi]$-modules whose homology belongs to $\cO [\prec \la]$ and $\cO [\preceq \la]$ correspondingly.
	
	\begin{thm}\label{D-content}
		Let $\la \in \sfP$ and let $i \in \tI_\af$.
		\begin{enumerate}
			\item If $s_i \la \prec \la$, then the endofunctor $\sD_i$ preserves the full subcategory $\bD_{c}^- [\preceq \la]$;
			\item If $s_i \la \succ \la$, then the endofunctor $\sD_i$ maps the subcategory $\bD_{c}^- [\preceq \la]$ to $\bD_{c}^- [\preceq s_i \la]$.
		\end{enumerate}
	\end{thm}
	\begin{proof}
		Since the functor $\sD_{i}$ commutes with limits, it is enough to check the statement of the theorem for modules whose weight support consists of one element. Therefore, the question reduces to the detailed observation of the $\msl_2$-case.
		We showed in Example~\ref{ex::sl_2::T_i} that the derived induction $\sD_{i}(M)$ (as well as $\sT_{i}(M)$) has no derived components whenever the weight support of a $\fI[\xi]$-module $M$ consists of one element $\lambda$ which is $\alpha_i$-dominant: $\langle \lambda,\alpha_i^{\vee}\rangle>0$. In other words,  $\sD_{i}(M)$ is a complex concentrated in zero homological degree and the weight support of the $\fI[\xi]$-modules $\sD_{i}(M)$ is $s_{i}$-symmetric.
		In particular, the subspace of weight $s_{i}(\lambda)$ of the modules $\sD_{i}(M)$ as well as $\sT_{i}(M)$ is isomorphic to $M$ and, moreover, $\sT_{i}(M)$ belongs to $\bD_{c}^{-}[\preceq \la]$.
		On the other hand, if $\langle\la,\alpha_i^{\vee}\rangle \leq 0$ and the weight support $\Psi(M)$ consists of one element $\lambda$ the complex $\sD_{i}(M)$ is concentrated in $-1$'st cohomological degree and consists of a module whose weight support is concentrated in $\sfP[\preceq\lambda]$.
	\end{proof}
	
	\begin{cor}
		The endofunctors $\sD_{i}$ preserve $\bD_{c}^{-}[\preceq\la_{+}]$ for all $i\in \tI_\af$ and all $\lambda\in\sfP$.
	\end{cor}

	\subsection{$\sY$-eigen objects in $\bD_{c}^-(\cO(\fI[\xi]))$}
	\label{sec::Y::eigen}
	
	For each fundamental coweight  $\omega_i \in \Pi^{\vee}$ with $i\in\tI$  we  have the translation element $t_{\omega_i}\in W^{\af}$ with a reduced decomposition:
	\begin{equation*}
		t_{\omega_i} = \pi s_{i_1} s_{i_2} \cdots s_{i_{\ell}} \hskip 5mm i_1,\ldots,i_{\ell} \in \tI_\af, \pi \in \Omega.
	\end{equation*}
	This follows the presentation of the functor $\sY^{\omega_i}$ which categorifies the $\hY$-elements of $\DAHA$:
	$$\sY^{\omega_i} = \pi \circ \sT_{i_1} \circ \sT_{i_2} \circ \cdots \circ \sT_{i_\ell}.
	$$
	Thanks to the categorification Theorem~\ref{thm::DAHA::categorification} we know that the endofunctor $\sY^{\omega_i}$ does not depend on the choice of a reduced decomposition.
	Moreover, for each $\mu \in \sfQ^{\vee}$, the endofunctor $\sY^\mu:=\sY^{\mu_{+}}\circ(\sY^{\mu_{-}})^{-1}$ is uniquely defined up to an isomorphism. Here $\mu=\mu_{+}-\mu_{-}$ with $\mu_+$ and $\mu_-$ being dominant weight.

	\begin{thm}
		\label{thm::Y_eigen}	
		For each  module $M\in \cO(\fI[\xi])$ whose weight support $\Psi(M)$ consist of a unique  dominant weight $\lambda$
		there exists a well defined object $\sEM_{\lambda}(M)\in \bD_{c}^{-}[\preceq \lambda]$ such that
		\begin{itemize}
			\setlength{\itemsep}{0em}	
			\item[$(\imath)$] $\sEM_{\lambda}(M)\in \bD^{\leq 0}$;
			\item[$(\imath\imath)$] The subspace of weight $\lambda$ in $H^{s}(\sEM_{\lambda}(M))$ is equal to zero for $s<0$ and coincides with $M$ for $s=0$;
			\item[$(\imath\imath\imath)$] For all  $\mu\in \sfQ^{\vee}_{+}$ there exists an isomorphism $\sY^{\mu}(\sEM_{\lambda}(M)) = \sq^{-\langle\mu,\lambda\rangle}(\sEM_{\lambda}(M)).$
		\end{itemize}
	\end{thm}
	The proof of Theorem~\ref{thm::Y_eigen} is based on the derived  version of the Banach fixed-point Theorem applied to the generating set of pairwise commuting endofunctors $\sY^{\mu_i}$, $i\in \tI$
	with $\mu_i$ assembling a basis of $\sfQ^{\vee}$. The following Proposition~\ref{prp::Y::map} is the key intermediate step in the construction of $\sEM_{\lambda}(M)$.
	
	Let us denote by $\widetilde{\sY}^{\mu}$ the grading shift  $\sq^{\langle\mu,\lambda\rangle}\circ\sY^{\mu}$ of the endofunctor $\sY^{\mu}$.
	We suppose that $\mu$ is dominant and the power $q^{\langle\mu,\lambda\rangle}$ is the eigen value of the operator $Y^{\mu}$ on the nonsymmetric Macdonald polynomial $E_{\lambda}$.

	\begin{prop}
		\label{prp::Y::map}	
		Suppose that $\lambda,\mu$ is a pair of dominant weights and $M\in \cO(\fI[\xi])$ is an $\fI[\xi]$-module whose weight support $\Psi(M)$ consists of $\lambda$.
		Then there exists a map $\bD_{c}^- [\preceq \la]$:
		\[\gamma:\widetilde{\sY}^{\mu_i}(M)\to M, \text{ for } i\in \tI\]
		and, moreover, the $z$-graded component of the cohomology of the $\cone(\gamma)$ differs from zero only for positive powers of $z$.
	\end{prop}
	\begin{proof}		
		Suppose $i\in\tI_{\af}$ and $\alpha_i$ is the corresponding simple root. Let $\lambda$ be $\alpha_i$-dominant and the weight support $\Psi(M)=\{\lambda\}$ for a $\fI[\xi]$-module $M$. As we already mentioned the complex $\sD_{i}(M)$ consists of a module concentrated in $0$-th homological degree, whose weight support is $s_{i}$-symmetric. In particular, the weight $s_i(\lambda)$
		belongs to the weight support  $\Psi(\sD_{i}(M))$ and is maximal with respect to the Cherednik order $\preceq$ and minimal with respect to the dominance order $\leq$. Therefore, we have a quotient map $\sD_{i}(M)\twoheadrightarrow s_{i}(M)$ such that $M$ belongs to the kernel by weight reasons. Consequently, we have  the map from the cone $\gamma_{i}:\sT_{i}(M)\to s_{i}M$. While iterating this map following the decomposition~\eqref{eq::t=ps} we obtain a collection of morphisms:
		\begin{multline}
			\label{eq::YM->M}
			\sY^{\mu}(M):=  \pi \circ \sT_{i_1} \circ \sT_{i_2} \circ \cdots \circ \sT_{i_\ell}(M) \stackrel{\gamma_{i_{\ell}}}\rightarrow \\ 
			\stackrel{\gamma_{i_{\ell}}}\rightarrow
			\pi \circ \sT_{i_1} \circ \sT_{i_2} \circ \cdots \circ \sT_{i_{\ell-1}}(s_{i_{\ell}}(M)) \stackrel{\gamma_{i_{\ell-1}}}\rightarrow \ldots 
			\stackrel{\gamma_{i_1}}\rightarrow \pi s_{i_1}\ldots s_{i_\ell}(M) = t_{\mu}(M) \simeq \sq^{-\langle\mu,\lambda\rangle}M.
		\end{multline}
		The module $t_{\mu}(M)$ differs from $M$ only by a shift grading along the loop parameter $z$ which is outlined in the last isomorphism of~\eqref{eq::YM->M}.
		Moreover, for all $k\leq \ell$ the weight support of the complex $N_k:=\sT_{i_k} \circ \cdots \circ \sT_{i_\ell}(M)$ contains a unique element $\nu$ with $\nu_{-}=\lambda_{-}$. One can show (by induction) that $\nu$ equals $s_{i_k}\ldots s_{i_\ell}(\lambda)$.
		All other weights of $H(N_k)$ belongs to the convex set $\sfP[\prec \lambda_{+}]$.
		Suppose that $\Bbbk_{\nu}$ is a one-dimensional module of weight $\nu$, then $z$-grading of $t_{\mu}(\Bbbk_{\nu})$ is greater or equal to $-\langle \mu, \nu_{+} \rangle$ and the same bounds take place for the $z$-grading of $\sY^{\mu}(\Bbbk_{\nu})$ as well as for the partial compositions $\pi\sT_{i_1}\ldots\sT_{i_k}(\Bbbk_{\nu})$. Consequently, the $z$-grading of the cone of a morphism $\tilde{\sY}^{\mu}(M)\to M$ is not less than
		$2 = {\langle \mu, \lambda\rangle} -max_{\nu\prec \mu_{+}}(\langle \mu,\nu_{+}\rangle).$
	\end{proof}

	\begin{proof}[Proof of Theorem~\ref{thm::Y_eigen}]
		Suppose that $\Psi(M)$ consists of one dominant weight $\lambda$ and dominant weight $\mu_1,\ldots,\mu_r$ constitute a basis of the weight lattice $\sfQ^{\vee}$.
		Thanks to Proposition~\ref{prp::Y::map} we have a collection of maps
		\begin{equation}
			\label{eq::Y::system}
			\begin{tikzcd}
				\ldots \arrow[r] &
				(\widetilde{\sY}^{\mu_i})^{m+1}(M) \arrow[r,"(\widetilde{\sY}^{\mu_i})^{m}(\gamma_i)"] &
				(\widetilde{\sY}^{\mu_i})^{m}(M) \arrow[r] & \ldots \arrow[r,"{\widetilde{\sY}^{\mu_i}(\gamma_i)}"] & 
				\widetilde{\sY}^{\mu_i}(M) \arrow[r,"{\gamma_i}"] & M.
			\end{tikzcd}
		\end{equation}
		We claim that the categorical version of the famous Banach fixed point Theorem can be easily formulated in the derived setting. The system~\eqref{eq::Y::system} satisfies the Mittag-Leffler condition because each $(z,\xi)$-bigraded component of a convergent complex has to be finite-dimensional. Therefore, the homotopy colimit $\hocolim_{m}(\widetilde{\sY}^{\mu_i})^{m}(M)$ is a well-defined object of the category of convergent complexes $\bD_{c}^{-}$.
		Moreover, thanks to the reduced decomposition of $\mu_i$ and Theorem~\ref{D-content} we know that this colimit belongs to $\bD_{c}^{-}[\preceq \lambda]$. The categorification Theorem~\ref{thm::DAHA::categorification} implies that endofunctors $\sY$'s commute and the iterated homotopy colimit
		\[\sEM_{\lambda}(M):=( \widetilde{\sY}^{\mu_1})^{\infty}\circ\ldots\circ (\widetilde{\sY}^{\mu_r})^{\infty}(M)
		\]
		is a well defined object in $\bD_{c}^{-}[\preceq \lambda]$.
		For each $\mu=\sum c_i \mu_i$ we have $\sY^{\mu}=(\sY^{\mu_1})^{c_1}\circ \ldots\circ (\sY^{\mu_r})^{c_r}$ and, consequently,
		$$\sY^{\mu}(\sEM_\lambda)(M) = \sq^{-\sum c_i \langle\mu_i,\lambda\rangle} \sEM_{\lambda}(M) = \sq^{-\langle\mu,\lambda\rangle} \sEM_\lambda(M). $$
		Therefore, the module $\sEM_{\lambda}(M)$ satisfies all required conditions $(\imath)-(\imath\imath\imath)$ of Theorem~\ref{thm::Y_eigen}.
	\end{proof}

	\begin{cor}
		\label{cor::Ext::periodicity}
		The constructed above objects $\sEM_{\lambda}(\Bbbk_{\lambda})$ associated with the one dimensional modules $\Bbbk_{\lambda}$ categorify the nonsymmetric Macdonald polynomials $E_{\lambda}$:
		$$\gch(\sEM_{\lambda}(\Bbbk_{\lambda}))=E_\lambda$$		
		and satisfy the Ext-periodicity property for each $\nu \in P$:
		\begin{equation}
			\label{eq::ext::vanishing}
			\sq^{\langle\nu,\lambda\rangle}\hom^{\udot}(\sEM_{\lambda},\sEM_{\mu})=\sq^{\langle\nu,\mu\rangle}\hom^{\udot}(\sEM_{\lambda},\sEM_{\mu}),
			\text{ if }\lambda\neq\mu.
		\end{equation}
	\end{cor}
	\begin{proof}
		First, we know that the nonsymmetric Macdonald polynomials are eigenfunctions of $Y$-operators and are uniquely defined by this property (up to a rational function on $q$ and $t$).
		Second, we know that the subspace of weight $\lambda$ of $\sEM_{\lambda}(\Bbbk_{\lambda})$ is one-dimensional.
		What follows that  $\gch(\sEM_{\lambda}(\Bbbk_{\lambda}))$ is equal to $E_\lambda$.

		Consequently, we have		
		\begin{multline*}
			\sq^{\langle\nu,\lambda\rangle}\hom^{\udot}(\sEM_{\lambda},\sEM_{\mu}) =
			\hom^{\udot}(\sq^{-\langle\nu,\lambda\rangle}\sEM_{\lambda},\sEM_{\mu}) =
			\hom^{\udot}(\sY^{\nu}(\sEM_{\lambda}),\sEM_{\mu}) =
			\\ =
			\hom^{\udot}(\pi\sT_{i_1}\ldots\sT_{i_{\ell}}(\sEM_{\lambda}), \sEM_{\mu}) =
			\hom^{\udot}((\sEM_{\lambda}),\sT_{i_\ell}'\ldots \sT_{i_1}'\pi^{-1}\sEM_{\mu}) =
			\\
			=	\hom^{\udot}(\sEM_{\lambda},(\sY^{\nu})^{-1}(\sEM_{\mu})) =
			\sq^{\langle\nu,\mu\rangle}\hom^{\udot}(\sEM_{\lambda},\sEM_{\mu}).
		\end{multline*}	
	\end{proof}
	It is worth mentioning that the $Ext$-periodicity property implies the vanishing of $Ext$ groups whenever we know any bounds on the complex $\hom(\sEM_{\lambda},\sEM_{\mu})$. 
	We will show in the preceding Section~\ref{sec::sl_2} that for $\fg=\msl_2$ the Macdonald complexes $\sEM_{\lambda}$ are modules concentrated in one homological degree. Therefore, the vanishing of $Ext(\sEM_{\lambda},\sEM_{\mu})$ holds for different $\lambda$ and $\mu$ for $\fg=\msl_2$ (Corollary \ref{extwanishing} below).
	
	Finally, we can describe modules that categorify symmetric Macdonald polynomials:
	\begin{cor}
		\label{cor::MM}	
		For each dominant weight $\lambda\in \sfP_{+}$	the character of the complex $$\sPM_\lambda:=\LInd_{\tI}(\sEM_\lambda)\in\bD_{c}^{-}(\fg[z,\xi])_{[\leq\lambda]}$$ 
		is proportional to the Macdonald (symmetric) polynomial $P_\lambda$.
	\end{cor}
	\begin{proof}
		We verified in Theorem~\ref{thm::PJ::categorify} that on the level of characters the induction functor $\LInd_{\tJ}$ acts as a Cherednik symmetrization functor $\hP_{\tJ}$ associated with a subsystem $\tJ\subset\tI$ and moreover preserve the subcategory with cohomology whose weight support belongs to $\sfP[\preceq\lambda]$.
		We apply it to the case of the embedding of a finite system to the affine one $\tI\subset\tI_{\af}$ and look for the induction $\LInd_{\tI}:\bD_{c}^{-}(\fI[\xi])\to \bD_{c}^{-}(\fg[z,\xi])$.
		The corresponding symmetrization operator $\hP_{\tI}$ maps a nonsymmetric Macdonald polynomial $E_\lambda$ to a polynomial which is proportional to a symmetric Macdonald polynomial $P_{\lambda_{+}}$ where $\lambda_{+}$ is the dominant weight in the $W$-orbit of $\lambda$ which coincides with $\lambda$ in our assumptions.
	\end{proof}

	\section{Macdonald modules for $\msl_2$}
	\label{sec::sl2::MM}
	Let us denote by $\fI:=\fb +z\msl_2[z]$ the Iwahori subalgebra of $\widehat{\msl_2}$.
	Define the following cyclic $\fI[\xi]$-modules.
	
	\begin{definition}\label{globalEM}
		Take $k \in \mathbb{Z}_{\geq 0}$. The cyclic module $\sEM_{-k \omega}$ is the graded cyclic module
		generated by the cyclic vector $v_{-k \omega}$ of $q,t$-degree $0$, weight $-k \omega$ subject to  the following list of relations:
		\begin{equation}
			(f z^a \xi^b) v_{-k \omega}=0, a >0, b ={0,1};~
			e^{k+1}v_{-k \omega}=0;~ (h\xi) v_{-k \omega}=0.
		\end{equation}
		
		The cyclic module $\sEM_{k \omega}$ is the graded cyclic module
		with the cyclic vector $v_{k \omega}$ of $q,t$-degree $0$, weight $k \omega$ and the following relations:
		
		\begin{equation}
			e z^a \xi^b v_{k \omega}=0, a \geq 0, b ={0,1};~
			(fz)^{k}v_{k \omega}=0;~ h\xi v_{k \omega}=0.
		\end{equation}
		
		We call these modules by  \emph{global nonsymmetric Macdonald modules}.
		
	\end{definition}
	
	\begin{rem}\label{pitwist}
		The module $\sEM_{-k \omega}$ is the $\pi$-automorphism twist of the module $\sEM_{(k+1) \omega}$.
	\end{rem}
	
	We need local versions of these modules.
	
	\begin{definition}\label{localEM}
		The quotient module
		$	\sEMloc_{k \omega}=\sEM_{k \omega}/\langle h \rangle[z]v_{k \omega}$
		is called \emph{the local nonsymmetric Macdonald module}.
	\end{definition}
	
	\begin{rem}
		Note that in difference to ordinary non super Weyl modules the $k \omega$-weight space of $\sEMloc_{k \omega}$ is not one dimensional
		because we have a nontrivial action of some elements $hq^a\xi$, $a > 0$ on the generator. One more difference is that
		$\sEMloc_{k \omega}$ is not a restriction of any module over the whole current algebra.
	\end{rem}
	
	\subsection{The module $\sEMloc_{- \omega}$ and its deformation}
	\label{sec::EM::sl2}
	
	In this subsection, we study the properties of the smallest nontrivial nonsymmetric Macdonald module $\sEMloc_{- \omega}$.
	\begin{prop}\label{firstmodule}
		The elements $v_{-\omega}, hz\xi v_{-\omega}, e v_{-\omega}, e\xi v_{-\omega}$ assemble a basis of the module $\sEMloc_{- \omega}$ and its  supercharacter is equal to:
		\[(1-qt)x^{-1}+(1-t)x=\widetilde E_{-\omega}(x, q, t).\]
	\end{prop}
	\begin{proof}
		Using PBW theorem we have that $\sEMloc_{- \omega}$ has a basis of the following form:
		\[\prod_{i=1}^{s_e} \left(ez^{c_i}\xi^{c'_i}\right)\prod_{i=1}^{s_h}\left( h z^{b_i}\xi^{b'_i}\right)\prod_{i=1}^{s_f} \left(fz^{a_i}\xi^{a'_i}\right) \cdot v_{-\omega}.\]
		Relations from Definition \ref{globalEM} imply $s_f=0$ and $b_i \geq 1$. Relations from Definition \ref{localEM} imply $b_i'=1$.
		Further we have for $d>0$:
		\begin{equation}\label{ez=0}
			ez^d v_{-\omega}=-\frac{1}{2}fz^d \cdot e^2 v_{-\omega}=0.
		\end{equation}
		Thus for $d>0$:
		\begin{equation}\label{hz2xi=0}
			hz^{d+1}\xi v_{-\omega}=f \xi \cdot ez^d v_{-\omega}=0.
		\end{equation}
		
		Thus ${-\omega}$-weight space of $\sEMloc_{- \omega}$ is linearly generated by $v_{-\omega}$ and $hz\xi v_{-\omega}$ that are linerly independent by degree reasons.
		
		Moreover we have for $d>0$:
		\begin{equation}\label{ezxi=0}
			ez^d\xi v_{-\omega}=\frac{1}{2}h \xi \cdot ez^d v_{-\omega}=0.
		\end{equation}
		
		Moreover, by the analogous computations, we obtain for any $d>0$, $d'=0,1$:
		\[ez^d \xi^{d'} \cdot hz\xi v_{-\omega}=0.\]
		
		What remains consists of monomials  $ev_{-\omega}$ and $e\xi v_{-\omega}$ that are linearly independent by degree reasons and span the  $\omega$-weight space of $\sEMloc_{- \omega}$.
		
		Finally, we have:
		\[e\xi\cdot ev_{-\omega}=\frac{1}{4}h\xi \cdot e^2 v_{-\omega}=0.\]
		
		This completes the proof.
	\end{proof}
	
	\begin{cor}
		$\sEMloc_{2\omega}$ has character equal to $\widetilde E_{2\omega}(x, q, t)$.
	\end{cor}
	\begin{proof}
		It follows from Remark \ref{pitwist}.
	\end{proof}
	The action of the free commutative algebra $\Bbbk[hz, hz^2, \dots] = U(h[z]) $ on the cyclic vector $v_{-\omega}\mapsto hz^{k} v_{-\omega}$ extends to the right action on the global nonsymmetric Macdonald module $\sEM_{-\omega}$ by automorphisms. The essential image of this action is isomorphic to $\Bbbk[hz]$ thanks to the equality:
	\[h z^k v_{-\omega}=(h z)^k v_{-\omega}\]
	that is well known in the nonsuper case \cite{CPweyl}.

	Thus the local module $\sEMloc_{-\omega}$ is the quotient of $\sEM_{-\omega}$ by the ideal generated by $h z v_{-\omega}$ that leads to the following deformation $\sEMloc_{-\omega}^\alpha$ of the module $\sEMloc_{-\omega}$:
	
	\begin{definition}\label{deformedmodule}
		\[\sEMloc_{-\omega}^\alpha=\sEM_{-\omega}/h (z-\alpha) v_{-\omega}.\]
	\end{definition}
	
	The module $\sEMloc_{-\omega}^\alpha$ is not $z$-graded but still inherits the $\xi$-degree and the weight decomposition.
	
	\begin{lem}\label{deformeddimension}
		$\sEMloc_{-\omega}^\alpha$ has a basis $v_{-\omega}$, $hz\xi v_{-\omega}$, $e v_{-\omega}$, $e \xi v_{-\omega}$.
	\end{lem}
	
	\begin{proof}
		As we know from Proposition~\ref{firstmodule} the elements 	$v_{-\omega}$, $hz\xi v_{-\omega}$, $e v_{-\omega}$, $e \xi v_{-\omega}$ are linearly independent for $a=0$, hence they are linerly independent for generic $\alpha$.
		Moreover, from the theory of Weyl modules for nonsuper Lie algebra $\fI\subset\widehat{\msl_2}$ we know that for $d>0$: $e{z-\alpha}^d v_{-\omega}=0$. Therefore,
		we have:
		\[h(z-\alpha)^{d+1} \xi v_{-\omega}=-f z \xi \cdot e(z-\alpha)^d v_{-\omega}=0.\]
		
		Thus $-\omega$-weight space has a basis $\{v_{-\omega},hz\xi v_{-\omega}\}$.
		
		Also, we have:
		\[e(z-\alpha)^d\xi v_{-\omega}=\frac{1}{2}h \xi \cdot e(z-\alpha)^d v_{-\omega}=0.\]
		The following relations	
		\[e \cdot h(z-\alpha)\xi v_{-\omega}=0 \quad \text{ and } \quad e \cdot e\xi v_{-\omega}=0\]
		are the straightforward generalizations of the one explained in the nondeformed case.
	\end{proof}
	As one can see from the proof of Lemma~\ref{deformeddimension} the following relations are satisfied in $\sEMloc_{-\omega}^{\alpha}$:
	\begin{equation}
		\label{deformedaction}
		\begin{array}{c}
			e z^k v_{-\omega}=\alpha^k e v_{-\omega};\\
			h z^k \xi v_{-\omega}=\alpha^k h \xi v_{-\omega}; \\
			e z^k\xi v_{-\omega}=\alpha^k e\xi v_{-\omega}.
		\end{array}
	\end{equation}
	
	\begin{cor}\label{freeactionfundamental}
		The algebra	$\Bbbk[hz]$ acts freely on $\sEM_{-\omega}$.
	\end{cor}
	
	\subsection{Fusion product}
	\label{sec::fusion::sl2}
	
	Take $k$ different numbers $\alpha_1, \dots, \alpha_k$.
	Consider the following module:
	\begin{equation}\label{evaluationmodule}
		\sEMloc_{-k\omega}^{\alpha_1, \dots, \alpha_k}:=\sEMloc_{-\omega}^{\alpha_1} \otimes \dots \otimes \sEMloc_{-\omega}^{\alpha_k}.
	\end{equation}

	\begin{prop}
		$\sEMloc_{-k\omega}^{\alpha_1, \dots, \alpha_k}$ is cyclic.
	\end{prop}
	\begin{proof}
		There is a standard trick introduced in~\cite{FL} showing that Relations~\ref{deformedaction} implies that
		the product of $k$ cyclic vectors $v:=v_{- \omega}\otimes \dots \otimes v_{-\omega}$
		is a cyclic for this module. 
	\end{proof}
	
	Next we prove that the module $\sEMloc_{-k\omega}^{\alpha_1, \dots, \alpha_k}$ satisfies the defining relations of $\sEM_{-k\omega}$.
	More precisely:
	\begin{lem}
		\begin{equation*}\label{relationsevaluationmodule1}
			f z^a \xi^b v=0; \
			h\xi v=0; \
			e^{k+1}v=0.
		\end{equation*}
	\end{lem}
	\begin{proof}
		The first two relations hold because elements $f z^a \xi^b$ and $h \xi$ annihilate the cyclic vectors in all tensor factors.
		The last equality holds because the element $e^{1+1}$ annihilates these vectors.
	\end{proof}
	
	The universal enveloping algebra $U(\msl_2[z,\xi])$ is bigraded with respect to current parameters $z$, $\xi$. Denote by $F_m$ the sum of subspaces of this algebra whose $z$-degree is $\leq m$, by $F_{\ldot}$ the corresponding filtration of $U(\msl_2[z,\xi])$ and we denote by $F_m(\sEMloc_{-k\omega}^{\alpha_1, \dots, \alpha_k})$
	the subspace $F_m v$ given by filtration on the cyclic vector $v$.
	We put $F_{-1}:=\{0\}$.  Define:
	\begin{equation*}
		\mathsf{gr}^{F}\sEMloc_{-k\omega}^{\alpha_1, \dots, \alpha_k}:=\bigoplus_{m \geq 0}F_m/F_{m-1}.
	\end{equation*}
	
	The module $\mathsf{gr}^{F}\sEMloc_{-k\omega}^{\alpha_1, \dots, \alpha_k}$ is graded by $z$-degree, $\xi$-degree and weight. By construction, it is generated
	by the image of the element $v$.
	
	\begin{prop}
		\label{prp::upper:bound:MM}	
		The assignment $v_{-k\omega}\mapsto v$ extends to the 	
		surjective homomorphism of $\msl_2[z,\xi]$-modules:
		\begin{equation}\label{lowerbound}
			\sEMloc_{-k\omega} \twoheadrightarrow \mathsf{gr}^{F}\sEMloc_{-k\omega}^{\alpha_1, \dots, \alpha_k}.
		\end{equation}
		In particular, $\dim \sEMloc_{-k\omega} \geq 4^{k}$.
	\end{prop}
	\begin{proof}
		It is sufficient to show that the cyclic vector of $\mathsf{gr}^{F}\sEMloc_{-k\omega}^{\alpha_1, \dots, \alpha_k}$ satisfies the defining relations of
		$\sEMloc_{-k\omega}$.
		We know that elements $f z^a \xi^b$, $h \xi$ and $e^{2}$ annihilate the cyclic vectors in all tensor factors. Consequently, we have
		\begin{equation*}\label{relationsevaluationmodule2}
			f z^a \xi^b v=0; \quad
			h\xi v=0; \quad
			e^{k+1}v=0.
		\end{equation*}
		The subspace of $\sEMloc_{-k\omega}^{\alpha_1, \dots, \alpha_k}$ spanned by  vectors of $\xi$-degree zero and weight $-k \omega$ is one-dimensional. Therefore, 
		the element $hz^{k}$ annihilates the cyclic vector for any $k>0$ in the graded module 
		$\mathsf{gr}^{F}\sEMloc_{-k\omega}^{\alpha_1, \dots, \alpha_k}$.
	\end{proof}
	
	\begin{cor}
		$\dim \sEMloc_{(k+1)\omega} \geq 4^{k}$.
	\end{cor}
	\begin{proof}
		This follows from the action of the diagram automorphism $\pi$.
	\end{proof}
	
	\subsection{Decomposition procedure}
	\label{sec::sl2::decompose}
	In this subsection, we construct four component filtration on the module $\sEMloc_{-k\omega}$ such that
	the corresponding subquotients are isomorphic to modules $\sEMloc_{-(k+1)\omega}$ and $\sEMloc_{k\omega}$ with shifted
	degree. We use this result to compute the supercharacters of these modules.
	
	Take the following four elements of the module $\sEMloc_{(-k+1)\omega}$:
	\begin{equation}\label{filtrationgenerators}
		v_1:=v;~ v_2:=e^{k+1} v_1;~ v_3:=h\xi v_2; ~ v_4:=(fz)^{k+1}v_3.
	\end{equation}
	We define the submodules generated by the aforementioned vectors:
	\begin{equation}\label{filtrationdecompopsition}
		\sDM_i:=U(\fI[\xi]) v_i \subset \sEMloc_{(-k+1)\omega}.
	\end{equation}
	
	We have the following sequence of inclusions of modules:
	\[\sDM_1 \supset \sDM_2 \supset \sDM_3 \supset \sDM_4\]
	and the subsequent Lemma explains the inductive description of the successive quotients
	\begin{prop}\label{filtr}
		The assignments $v_{\lambda}\mapsto v_i$ of the cyclic vectors extend to the following list of  epimorphisms	
		\begin{gather}
			\label{filtr1}
			\sX^{-\omega} \sEMloc_{-k\omega} \twoheadrightarrow \sDM_1/\sDM_2; \\
			\label{filtr2}
			\sX^{\omega} \sEMloc_{(k+1)\omega} \twoheadrightarrow \sDM_2/\sDM_3; \\
			\label{filtr3}
			\sX^{\omega}\sv \sEMloc_{(k+1)\omega} \twoheadrightarrow \sDM_3/\sDM_4; \\
			\label{filtr4}
			\sX^{-\omega}\sq^{k+1} \sEMloc_{-k\omega}\twoheadrightarrow \sDM_4;
		\end{gather}
	\end{prop}
	\begin{proof}
		Note that all successive quotients $\sDM_i/\sDM_{i+1}$ are cyclic modules by construction and the grading shifts are adapted to have a cyclic vector of the same weight and $(z,\xi)$-degree in the corresponding nonsymmetric Macdonald module.
		So to have well-defined morphisms~\eqref{filtr1}-\eqref{filtr4} one has to check the defining relations of the Macdonald module.
		\begin{itemize}
			\item
			The relations for $v_1$ follow directly from the defining relations of $\sEMloc_{(-k)\omega}$. 	
			\item 	
			From the theory of Weyl modules we have:
			\[hz^l v_2=0,~l >0;~(fz)^{k+1}v_2=0;~ ez^l v_2=0,~l \geq 0.\]
			Moreover from the theory of Weyl modules for $l>0$ we have:
			$e^k ez^l w=0$.
			Now we get:
			\[e\xi v_2= e\xi\cdot e^{k+1}\cdot v=\frac{1}{2(k+2)}h_\xi e^{k+2}\cdot v=0\]
			and for $l>0$
			\[ez^l\xi v_2= ez^l\xi \cdot e^{k+1} w=\frac{1}{2}( hz^l\cdot e\xi\cdot e^{k+1}- 2(k+1)e\xi\cdot e^k\cdot ez^l )v=0.\]
			Moreover, by definition, we have:
			\(h\xi v_2 \in \sDM_3\)
			what implies the map~\eqref{filtr2}.
			\item
			We have $[hz^l, h\xi]=[ez^l\xi, h\xi]=0$.
			Therefore we get for $l>0$:
			\begin{gather*}
				hz^l v_3= hz^l \cdot h\xi v_2=h\xi \cdot hz^l v_2=0;
				\\			
				ez^l\xi v_3= ez^l\xi \cdot h\xi v_2=h\xi \cdot ez^l\xi v_2=0.
				\\
				h\xi v_3=h\xi \cdot h\xi v_2=0,
			\end{gather*}
			because $h\xi$ is odd.
			Note that $v_3=\frac{1}{2(k+1)}e^k \cdot e\xi v$. Hence:
			\(ez^lv_3=0,\)
			and, by definition, we have:
			\[(fz)^{k+1}v_3 \in \sDM_4.\]
			
			\item 	
			PBW theorem implies that $v_4 \in U(\langle h \rangle[z,\xi])v$. Therefore we obtain that the elements $ez^l\xi^{l'}$, $l \geq 0$
			$hz^l$, $l>0$ and $h \xi$ annihilate this element.
			Consequently, we have:
			\[e^{k+1}v_4=e^{k+1}(fz)^{k+1}v_3\in \Bbbk[hz, hz^2, \dots]v_3.\]			
			Thus $e^{k+1}v_4=0$ because of relations satisfied by the element $v_3$.
		\end{itemize}
		
	\end{proof}	
	The diagrammatic automorphism $\pi$ allows one to define the $4$-term  filtration of the nonsymmetric local Macdonald module $\sEMloc_{(k+2)\omega}$ with surjections of smaller Macdonald modules on the successive quotients:
	\begin{gather*}
		\sEMloc_{(k+1)\omega}\twoheadrightarrow \sDM_1/\sDM_2; \\
		\sX^{-\omega}\sq^{k+1}\sEMloc_{-k\omega} \twoheadrightarrow \sDM_2/\sDM_3; \\
		\sX^{-\omega}\sq^{k+1}\sv \sEMloc_{-k\omega}\twoheadrightarrow \sDM_3/\sDM_4; \\
		\sX^{\omega}\sq^{k+1}\sv\sEMloc_{(k+1)\omega} \twoheadrightarrow \sDM_4. \\
	\end{gather*}
	
	\begin{thm}\label{character}
		The character of the nonsymmetric Macdonald module $\sEMloc_{k\omega}$ is equal to the integral form of the nonsymmetric Macdonald polynomial
		$\widetilde E_{k\omega} (x, q, t).$
		In particular, 	
		$$\dim \sEMloc_{-k\omega}=\dim \sEMloc_{(k+1)\omega}=4^k$$
		and  all surjections $\sEMloc_{\cdot}\twoheadrightarrow \sDM_i/\sDM_{i+1}$ are isomorphisms.
	\end{thm}
	\begin{proof}
		The inductive description of $\sEMloc_{-k\omega}$ explained in Proposition~\ref{filtr} implies the upper bound $\dim \sEMloc_{-k\omega} \leq 4 \dim \sEMloc_{(1-k)\omega}$. Together with the base of induction $\dim \sEMloc_{\omega} = \dim EM_{0} = 1$ we have
		\begin{equation}
			\label{eq::dim::EM}
			\dim \sEMloc_{-k\omega}=\dim \sEMloc_{(k+1)\omega}\leq 4^k.
		\end{equation}
		The lower bound discovered in Proposition~\ref{prp::upper:bound:MM} implies the equality in~\eqref{eq::dim::EM}. Moreover, all successive quotients $\sDM_i/\sDM_{i+1}$ have to be isomorphic to the appropriate shift of the smaller nonsymmetric Macdonald module.
		The corresponding recursive formula for the characters of Macdonald modules that follows from the decomposition~\eqref{filtr1}-\eqref{filtr4} coincides with the well-known recurrent formulas outlined in~\cite{RY}.
	\end{proof}

	\subsection{Eigenvalues}
	\label{sec::sl2::Y}
	
	The goal of this subsection is to prove that the modules $\sEMloc_{m\omega}$ are eigen objects for the endofunctor $\sY^{\omega}:=\pi \sT_1$ and its inverse $\sY^{-\omega}:= \sT_1'\pi$ in the categorical sense:
	\begin{thm}
		\label{thm::Y_eigen::sl2}	
		There are isomorphisms
		\begin{gather}
			\sY^\omega (\sEMloc_{m\omega})\simeq
			\sq^{-m}\sEMloc_{m\omega};  \text{ if } m > 0\\
			\sY^{-\omega} (\sEMloc_{m\omega})\simeq	\sv^{-1}\sq^{m}  \sEMloc_{m\omega}[-1];  \text{ if } m \leq 0.
		\end{gather}
	\end{thm}
	
	\begin{lem}\label{IndtoMacdonald}
		For $m>0$
		the module $\Ind_1 (\sEMloc_{m\omega}) = H^{0}(\LInd_{1}(\sEMloc_{m\omega}))$ is the cyclic module with the generator $v_{-m\omega}$ yielding the following	relations:
		\begin{equation}\label{indRelations}
			f z^a \xi^b v_{-m \omega}=0, a \geq 0, b ={0,1};~
			(e)^{m+1}v_{-m \omega}=0;~ h\xi v_{-m \omega}=0.
		\end{equation}
	\end{lem}
	\begin{proof}
		By definition we have that the $\msl_2[z, \xi]$ induced module $U(\msl_2[z,\xi])\otimes_{U(\fI[z,\xi])} \sEMloc_{m\omega}$ is the cyclic module generated by the element $v_{m \omega}$ satisfying the following relations:
		\begin{equation}
			\label{eq::rel::D1::EM}
			e z^a \xi^b v_{m \omega}=0, a \geq 0, b ={0,1};~
			(f)^{m+1}v_{m \omega}=0;~ h\xi v_{m \omega}=0.
		\end{equation}
		The straightforward check shows that the $\msl_2[z, \xi]$-module, yielding Relations~\eqref{eq::rel::D1::EM} is finite-dimensional and, therefore, is integrable.
		Thus, the module defined by Relations~\eqref{eq::rel::D1::EM} is isomorphic to $\Ind_1 (\sEMloc_{m\omega})$.
		The Weyl group automorphism $s_1$ keeps the module $\Ind_1(\sEMloc_{m\omega})$,  $s_1(v_{m\omega}) = v_{-m\omega}$ and $s_1$ maps Relations~\eqref{eq::rel::D1::EM} to Relations~\eqref{indRelations}.
	\end{proof}
	\begin{proof}[(proof of Theorem~\ref{thm::Y_eigen::sl2})]
		Suppose that $m>0$.	
		The presentation~\eqref{indRelations} implies the isomorphism of $\fI[\xi]$-modules:
		$$\Res_1\circ\Ind_1(\sEMloc_{\omega}) \simeq {\sDM}_2/{\sDM}_3 \stackrel{\eqref{filtr2}}{\simeq} \sX^{-\omega}\sq^{m}\sEMloc_{(-m+1)\omega}$$
		Consequently, we have:
		\[H^0(\sT_1\sEMloc_{m\omega})=H^0(\cone (\sEMloc_{m\omega} \rightarrow \sD_1 (\sEMloc_{m\omega})))\]
		and is isomorphic to the quotient $X^{-\omega}\sEMloc_{(-m+1)\omega}$.
		The direct comparison of relations explains the following isomorphism:
		\[\pi X^{-\omega}\sEMloc_{(-m+1)\omega} \simeq q^{-m}\sEMloc_{m\omega}.\]
		The functor $\sT_1$ has only $0$ and $-1$'st nontrivial cohomology and 	
		it remains to show that $-1$'st cohomologies of $\sT_1\sEMloc_{m\omega}$ vanishes.
		The latter follows from the equality of characters that makes sense due to the main categorification Theorem~\ref{thm::DAHA}:	
		\begin{multline*}
			\gch(H^{0}(\pi\sT_1 (\sEMloc_{m\omega}))) = \gch(\sq^{-m}\sEMloc_{m\omega}) = q^{-m}\widetilde E_{m\omega}= \\ = Y^{\omega}(\widetilde E_{m\omega}) = \gch(\sY^{\omega}\sEMloc_{m\omega}) =
			\gch(\pi \sT_1 \sEMloc_{m\omega})
		\end{multline*}
		Here, as before, $\widetilde E_{m\omega} = \widetilde E_{m\omega}(x, q, t)$ denotes the integral form of  the nonsymmetric Macdonald polynomial $E_{m}$. $\widetilde{E_{m}}$ differs from $E_m$ by a certain polynomial depending on $q,t$ such that $\widetilde{E_{m}}$ is a polynomial in $q$,$t$ with integral coefficient and, therefore, is an eigenvector of the endofunctor $Y^{\omega}$.
		
		The proof for $m\leq 0$ is completely analogous and is based on the isomorphism:
		$$
		\Res_1 \circ \Ind_1 (\sEMloc_{-k\omega}) = {\sDM_3}\sv[1].$$
	\end{proof}
	
	\begin{cor}\label{extwanishing}
		For all $\lambda\neq\mu$ we have the vanishing of the cohomology in the derived category $\bD_{c}^{-}(\cO(\fI[\xi]))$:
		$$Ext_{\cO(\fI[\xi])}^{\udot}(\sEMloc_{\lambda},\sEMloc_{\mu})=0.$$
	\end{cor}
	\begin{proof}
		Suppose that a $z$-graded vector space $V=\oplus_{n\in\bZ} V_n$ is isomorphic to itself with a shifted grading. I.e. $V\simeq \sq^{m} V$ for $m\neq 0$. Then whenever $V\neq 0$ $\dim V=\infty$ because $V_{n}\simeq V_{n+m} \simeq V_{n+2m} \simeq \ldots$.  
		However, we know that the finite-dimensional modules  $\sEMloc_{\lambda}$ are eigen objects with respect to the action of endofunctor $\sY^{\omega}$. Therefore, we have the Ext-periodicity property discovered in Corollary~\ref{cor::Ext::periodicity}.  In particular, this implies an isomorphism for different $\lambda\neq\nu \geq 0$: 
		$$
		Ext_{\fI[\xi]}(\sEMloc_{\lambda},\sEMloc_{\nu}) \simeq \sq^{m}
		Ext_{\fI[\xi]}(\sEMloc_{\lambda},\sEMloc_{\nu}).
		$$
		Here $m:=\langle\omega,\lambda-\mu\rangle$ is a different from zero integer number. Each $Ext^p$ is a finite-dimensional $z$-graded vector space what implies that it is zero. The case of negative $\lambda$ or $\mu$ is covered by the same arguments and uses also a homological shift.
	\end{proof}

\end{document}